\documentclass[a4paper]{amsart}

\usepackage[all]{xy}
\usepackage{amsmath}
\usepackage{amssymb}
\usepackage{amsthm}
\usepackage{latexsym}
\usepackage{enumitem}
\usepackage{prettyref}
\usepackage{hyperref}
\usepackage[abbrev,lite,alphabetic]{amsrefs}
\usepackage{bm}
\usepackage{tikz}
\usepackage{tikz-cd}
\usetikzlibrary{matrix,arrows,decorations.pathmorphing}%\usetikzlibrary{calc}
\usepackage{mathrsfs}
\usepackage{youngtab}
\usepackage{a4wide}
\usepackage{tikz}
\usepackage{tikz-cd}
\usepackage{comment}
\usepackage{xcolor}
\newtheorem{theorem}{Theorem}[section]
\newrefformat{thm}{\hyperref[{#1}]{Theorem~\ref*{#1}}}
\newtheorem{definition}[theorem]{Definition}
\newrefformat{def}{\hyperref[{#1}]{Definition~\ref*{#1}}}
\newtheorem{lemma}[theorem]{Lemma}
\newrefformat{lem}{\hyperref[{#1}]{Lemma~\ref*{#1}}}
\newtheorem{proposition}[theorem]{Proposition}
\newrefformat{prop}{\hyperref[{#1}]{Proposition~\ref*{#1}}}
\newtheorem{corollary}[theorem]{Corollary}
\newrefformat{cor}{\hyperref[{#1}]{Corollary~\ref*{#1}}}
\newtheorem{remark}[theorem]{Remark}
\newrefformat{rem}{\hyperref[{#1}]{Remark~\ref*{#1}}}

\newtheorem{convention}[theorem]{Convention}

{%\theorembodyfont{\upshape}
\newtheorem{examplecore}[theorem]{Example}}
\newrefformat{ex}{\hyperref[{#1}]{Example~\ref*{#1}}}

\newenvironment{example}{\begin{examplecore} \rm}{\hspace*{\fill}
$\square$\par\vspace{.1cm}\end{examplecore}}

\newcommand{\op}{\operatorname}
\newcommand{\scr}{\mathscr}
\newcommand{\et}{\mathrm{\acute{e}t}}
\newcommand{\ret}{\mathrm{r\acute{e}t}}

\newcommand{\spec}{\mathrm{Spec}\;}

\newcommand{\iso}{\cong}

\newcommand{\R}{\mathbb{R}}
\newcommand{\I}{\mathcal{I}}

\newcommand{\C}{\mathbb{C}}

\newcommand{\ga}{\gamma}

\newcommand{\de}{\delta}

\newcommand{\ra}{\rightarrow}
\newcommand{\si}{\sigma}

\newcommand{\Aut}{\operatorname{Aut}}

\newcommand{\II}{\mathcal{I}}

\newcommand{\pt}{{\rm pt}}

\newcommand{\intt}{\operatorname{int}}
\newcommand{\Out}{{\rm Out}}
\newcommand{\Int}{{\rm Inn}}
\newcommand{\inv}{{\rm inv}}
\begin{document}

\title[Witt-sheaf cohomology and homotopy fixed points]{Equivariant real cycle class map and\\ Witt-sheaf cohomology of classifying spaces}

\author{Lorenzo Mantovani, \'Akos K.\ Matszangosz and Matthias Wendt}

\date{February 2026}

\address{\'Akos K. Matszangosz, HUN-REN Alfr\'ed R\'enyi Institute of Mathematics, Re\' altanoda utca 13-15, 1053 Budapest, Hungary}
\email{matszangosz.akos@gmail.com}

\address{Matthias Wendt, Fachgruppe Mathematik und Informatik, Bergische Universit\"at Wuppertal, Gauss\-strasse 20, 42119 Wuppertal, Germany}
\email{m.wendt.c@gmail.com}


\thanks{\'A. K. M. is supported by the Hungarian National Research, Development and Innovation Office, NKFIH K 138828 and NKFIH PD 145995.}

\begin{abstract}
  In this paper, we study equivariant real cycle class maps for group actions on real schemes, with a view toward Witt-sheaf characteristic classes. The cycle class maps take values in singular cohomology of the real points of the quotient stack, which are identified with the homotopy fixed points of complex conjugation on the complex points. This provides a strong relation between Witt-sheaf cohomology of geometric classifying spaces of a real algebraic group $G$ and singular cohomology of classifying spaces of its strong real forms, which we discuss in a number of examples. As a sample application, we compute the number of Witt-sheaf cohomological invariants of spin groups over the reals.
\end{abstract}

\maketitle
\setcounter{tocdepth}{1}
\tableofcontents

\section{Introduction}

In algebraic geometry and motivic homotopy, there has been recent interest in cohomology theories related to quadratic forms, such as hermitian K-theory aka higher Grothendieck--Witt groups, Chow--Witt groups or Witt-sheaf cohomology. Besides their relevance in connection to quadratic forms, these cohomology theories have found applications in refined enumerative geometry, e.g.{}~\cite{levine:enumerative} or \cite{bachmann:wickelgren}, and the study of vector bundles or more generally torsors for linear algebraic groups, e.g.{}~\cite{asok:fasel:secondary}. There are also strong connections to real algebraic geometry and the topology of real points of real varieties, via real realization functors in motivic homotopy or real cycle class maps, e.g.{}~\cites{jacobson,4real}.

Equivariant versions of cohomology theories naturally appear when studying questions related to algebraic groups, and consequently equivariant cohomology for reductive group actions appears in many areas from enumerative geometry to geometric representation theory. Already the simplest case of the trivial $G$-action on a point is very interesting -- the cohomology of the classifying space of a group $G$ encodes the characteristic classes of $G$-bundles. In fact, the current project arose from attempts toward conceptual understanding of Witt-sheaf characteristic classes for quadratic forms, i.e., the computation of Witt-sheaf cohomology of ${\rm B}_\et{\rm O}(n)$ (over fields of characteristic $\neq 2$). 

The setting we consider for the paper is the following: we are interested in the \emph{Borel-equivariant} versions of  Witt-sheaf cohomology or ${\bf I}$-cohomology. Based on previous work in \cite{4real}, we describe an \emph{equivariant real cycle class map} from ${\bf I}$-cohomology to the singular cohomology of real points of a quotient stack associated to a group action $G\looparrowright X$. Moreover, we identify the target of the real cycle class map as the homotopy fixed points of complex conjugation on the complex points of the stack. This might not be too surprising since homotopy fixed point spaces have been much considered in the context of geometric representation theory of real Lie groups, and the current paper draws quite some inspiration from papers of Ben~Zvi--Nadler \cite{ben-zvi:nadler:loopreps} and Virk \cite{virk:hc2}. As a consequence, real forms of groups and actions play a significant role in equivariant Witt-sheaf cohomology, but as far as we know the consequences of this observation for quadratic cohomology theories haven't been spelled out so far. We also discuss a number of examples in Section~\ref{sec:examples}. As a particular case, our results provide a conceptual explanation of the Witt-sheaf characteristic classes of orthogonal groups in terms of Pontryagin and Euler classes of its real forms, the indefinite orthogonal groups. 

\subsection{Equivariant real cycle class map}

Recall that for a (separated, finite type) real scheme $X/\mathbb{R}$, Jacobson \cite{jacobson} defined a real cycle class map ${\rm H}^p_{\rm Zar}(X,{\bf I}^q)\to {\rm H}^p_{\rm sing}(X(\mathbb{R}),\mathbb{Z})$ from the Zariski cohomology of the sheaf ${\bf I}^q$ of $q$-th powers of fundamental ideals to singular cohomology of the real points $X(\mathbb{R})$ with the analytic topology. The real cycle class map is induced by a morphism $\op{sgn}\colon{\bf I}^q\to a_\ret\mathbb{Z}$ from ${\bf I}^q$ to the real-\'etale sheafification $a_\ret\mathbb{Z}$, essentially given by the signature of quadratic forms over real closed fields, cf.{}~\cite{bachmann:real-etale}. In \cite{4real}, this cycle class map was extended to incorporate line bundle twists, and various compatibilities with pullbacks, pushforwards (along closed immersions) and intersection products were established. In the present paper, we now want to consider a Borel-equivariant version of the real cycle class map. Alternatively, we want to extend the real cycle class map to quotient stacks, see Section~\ref{sec:quot-stacks} for terminology and background on quotient stacks. 

There are different ways to view (quotient) stacks as objects in motivic homotopy theory, see e.g.{}~\cite{hoskins:pepinlehalleur} or \cite{choudhury:deshmukh:hogadi:stacks}. On the one hand, one can consider approximations of the Borel construction by smooth schemes, following Totaro \cite{totaro:bg} and Edidin--Graham~\cite{edidin:graham:equivariant}. Using this perspective allows for a fairly direct definition of the real cycle class maps, and all compatibilities with pullbacks, pushforwards (along closed immersions) and intersection products from \cite{4real} carry over to the equivariant situation. On the other hand, the alternative way via realization of topological groupoids (or homotopy fixed point spaces) makes it easier to identify (or compute) the real points of spaces in terms of homotopy fixed points, which we will discuss in Section~\ref{sec:intro-hfp} below.

The following provides a synopsis of the basic properties of the equivariant real cycle class map defined in terms of approximations of the Borel construction, cf.{}~Theorem~\ref{thm:real-cycle-class}. Precise formulations of the compatibility statements can be found in Section~\ref{sec:equivariant-cycle-class}. To explain the notation in the theorem, for a group action $G\looparrowright X$ we denote by $[G\backslash X]$ the associated quotient stack as an object in the unstable motivic homotopy category, cf.{}~Section~\ref{sec:prelims}, and by ${\rm Re}_{\mathbb{R}}[G\backslash X]$ its real realization, cf.{}~Section~\ref{sec:real-pts-hc2}. 

\begin{theorem}
  \label{thm:main-real-cycle}
Let $G$ be a linear algebraic group over $\mathbb{R}$ and let $G\looparrowright X$ be a smooth scheme with $G$-action, and let $\mathscr{L}\in {\rm CH}^1_G(X)$ be a $G$-equivariant line bundle on $X$. Then the signature map ${\rm sgn}\colon {\bf I}^q\to a_{\ret}\mathbb{Z}$ induces a well-defined equivariant real cycle class map
\[
{\rm H}^p_G(X,{\bf I}^q(\mathscr{L}))={\rm H}^p([G\backslash X],{\bf I}^q(\mathscr{L}))\to {\rm H}^p_{\rm sing}({\rm Re}_{\mathbb{R}}[G\backslash X],\mathbb{Z}(\mathscr{L}))
\]
which satisfies the following properties:
\begin{enumerate}[label=\roman*)]

  \item The equivariant real cycle class map only depends on the stack $[G\backslash X]$. 
  \item The case of the trivial group $G=\{1\}$ recovers the non-equivariant cycle class map of \cites{jacobson,4real}. For a free action $G\looparrowright X$, the equivariant real cycle class map agrees with the non-equivariant cycle class map for the quotient scheme $G\backslash X$.
  \item The equivariant real cycle class map is compatible with the intersection product, cf.{}~Proposition~\ref{prop:compat-product}. 
  \item The equivariant real cycle class map is compatible with (lci) pullbacks and proper pushforwards (along representable closed immersions of quotient stacks), cf.{}~Propositions~\ref{prop:compat-pullback} and \ref{prop:compat-pushforward}. It is compatible with localization sequences and the quotient/induction equivalences, cf.{}~Corollaries~\ref{cor:quotient-equiv} and \ref{cor:induction-equiv}.
  \item The real cycle class map is an isomorphism for $q\gg 0$, and it is an isomorphism after tensoring with $\mathbb{Z}[1/2]$.
\end{enumerate}
\end{theorem}

Similarly to the non-equivariant cycle class map, it is possible to twist the equivariant quadratic cohomology theories by line bundles on the quotient stack, i.e., equivariant line bundles. For an equivariant line bundle $\mathscr{L}$, the $\mathscr{L}$-twisted equivariant ${\bf I}$-cohomology maps to the singular cohomology of ${\rm Re}_{\mathbb{R}}[G\backslash X]$ with coefficients in the local system $\mathbb{Z}(\mathscr{L})$. This assignment $\mathscr{L}\to\mathbb{Z}(\mathscr{L})$ is encoded in the mod 2 equivariant cycle class map
\[
  {\rm Ch}^1_G(X)\to {\rm H}^1_{\rm sing}({\rm Re}_{\mathbb{R}}[G\backslash X],\mathbb{F}_2)\colon \mathscr{L}\mapsto \mathbb{Z}(\mathscr{L}).
\]
One reason for incorporating twists by equivariant line bundles is that they have some relevance in the localization sequences, where they appear as (determinants of) equivariant normal bundles.

The behaviour of twists in the real cycle class maps can be quite subtle. One particularly interesting case is the real cycle class map for the classifying space of the normalizer $N={\rm N}_{{\rm SL}_2}(T)$ of the maximal torus $T$ in ${\rm SL}_2$. The Witt-sheaf cohomology in this case was computed by Marc Levine in \cite{levine:normalizer}, and we discuss the relation between his computation and the real realization ${\rm Re}_{\mathbb{R}}[N\backslash *]={\rm BS}^1\sqcup {\rm BC}_4$ in Section~\ref{sec:normalizer}.

\begin{remark}
  Obviously, it is also possible to extend cycle class maps on Chow groups to the equivariant setting. For our purposes, we also provide an extension of the Borel--Haefliger mod 2 cycle class map ${\rm Ch}^p(X)\to{\rm H}^p(X(\mathbb{R}),\mathbb{F}_2)$ to the equivariant setting. In fact, we discuss a cycle class map on equivariant mod 2 Milnor K-cohomology for a real scheme $X/\mathbb{R}$, cf.{}~Theorem~\ref{thm:borel-haefliger-equivariant}:
  \[
  {\rm H}^p_G(X,{\bf K}^{\rm M}_q/2)\to {\rm H}^p_{\rm sing}({\rm Re}_{\mathbb{R}}[G\backslash X],\mathbb{F}_2).
  \]
  The case $p=q=1$, providing real realization of equivariant line bundles, was already used above in the formulation of the twisted equivariant real cycle class map. 
\end{remark}

\subsection{Realizations of quotient stacks and homotopy fixed points}
\label{sec:intro-hfp}

The equivariant real cycle class map can provide some intuition of what the equivariant Witt-sheaf cohomology or ${\bf I}$-cohomology should look like, but to make full use of this, we need to understand real points of quotient stacks. In general, for an action $G\looparrowright X$ of a non-special group $G$, the real points $[G\backslash X](\mathbb{R})$ of a quotient stack are not simply given by $[G(\mathbb{R})\backslash X(\mathbb{R})]$ (i.e., the Borel construction for the real Lie group action $G(\mathbb{R})\looparrowright X(\mathbb{R})$). This is already apparent for $\mu_2\looparrowright *$. In this case, the nerve of the groupoid $[\mu_2\backslash*](\mathbb{R})$ is identified with ${\rm BC}_2\sqcup{\rm BC}_2$, compatible with the formulas for Witt-sheaf cohomology in \cite{dilorenzo:mantovani}, cf. also the discussion in Section~\ref{sec:finite}. This means that the target of the equivariant cycle class map is not simply equivariant singular cohomology ${\rm H}^*_{{\rm sing}, G(\mathbb{R})}(X(\mathbb{R}),\mathbb{Z})$ for the action $G(\mathbb{R})\looparrowright X(\mathbb{R})$.

Instead, viewing the quotient stack as an object in motivic homotopy, there are two ways to understand its real realization. On the one hand, one can consider scheme approximations of the Borel construction and take their real points. For the example $\mu_2$, this produces the correct answer; the two connected components arise from the points in $\left(\mathbb{C}^n\setminus\{0\}\right)/\{\pm 1\}$ having all coordinates real resp. all coordinates imaginary. However, this procedure might be difficult to carry out in practice, depending on the group $G$. On the other hand, we can take the geometric realization of the (topological) groupoid of real points. This can generally be identified with the homotopy fixed points for complex conjugation on the geometric realization of the (topological) groupoid $[G\backslash X](\mathbb{C})$ of complex points of the stack.

In fact, we can get an even better perspective on real realization functors by factoring through ${\rm C}_2$-equivariant homotopy.\footnote{We apologize that there are two different equivariant settings here. Whenever we are in the context of equivariant or real realization and our group is ${\rm C}_2$, we typically talk about Bredon-type (genuine) equivariant topology. On the other hand, whenever the group of equivariance is a linear group acting on a scheme, we are talking Borel-type equivariant topology. This shouldn't cause more than baseline confusion.} In the $\Gamma$-equivariant setting, there are two types of equivalences that can be considered: coarse $\Gamma$-equivalences only care about the underlying space, while fine $\Gamma$-equivalences are maps which induce equivalences on fixed-point spaces for all subgroups $H\leq \Gamma$. Over $\mathbb{R}$, there are well-known realization functors which we discuss in Section~\ref{sec:equivariant-htpy-realization}: from the (Nisnevich) motivic homotopy category to the fine ${\rm C}_2$-equivariant homotopy category, and similarly from \'etale motivic to coarse ${\rm C}_2$-equivariant homotopy. We can get the real realization by first taking fine ${\rm C}_2$-equivariant realization, followed by (genuine) ${\rm C}_2$-fixed points. The key point we want to make is that the equivariant realization of a quotient stack is also an equivariant quotient stack, i.e., a \emph{coarse} homotopy orbit space. This way, homotopy fixed points naturally appear in the real realization of \'etale classifying spaces. 

The result describing equivariant and real realizations of quotient stacks is then the following, cf.{}~Section~\ref{sec:real-pts-hc2}, in particular Propositions~\ref{prop:equiv-realization-bar} and \ref{prop:krishna-equiv-realization} for the Nisnevich case, and Propositions~\ref{prop:real-admissible-gadget} and \ref{prop:real-simplicial} for the \'etale case.

\begin{theorem}
  \label{thm:main1}
  Let $G$ be a linear algebraic group over $\mathbb{R}$, let $G\looparrowright X$ be a smooth scheme with $G$-action, and denote by $[G\backslash X]$ the quotient stack. Denote by $G(\mathbb{C})\looparrowright X(\mathbb{C})$ the equivariant realization, taking complex points with the analytic topology and ${\rm C}_2$-action by complex conjugation. Similarly, we denote by $G(\mathbb{R})\looparrowright X(\mathbb{R})$ the real realization, taking real points with the analytic topology.

  \begin{description}
  \item[equivariant] The ${\rm C}_2$-equivariant realization of Nisnevich homotopy orbits $X_{{\rm h}G}$ is given by the fine homotopy orbits $X(\mathbb{C})_{{\rm h}G(\mathbb{C})}$. The ${\rm C}_2$-equivariant realization of \'etale homotopy orbits $[G\backslash X]$ is given by the coarse homotopy orbits $[G(\mathbb{C})\backslash X(\mathbb{C})]$. In particular, ${\rm Re}_{\rm C_2}{\rm B}_\et G$ is the space classifying equivariant principal $({\rm C}_2,G(\mathbb{C}))$-bundles. 
  \item[real]
    The real realization of Nisnevich homotopy orbits $X_{{\rm h}G}$ is given by the homotopy orbits $X(\mathbb{R})_{{\rm h}G(\mathbb{R})}$. The real realization of \'etale homotopy orbits, aka the quotient stack $[G\backslash X]$, is given by the homotopy fixed-points of the equivariant realization:
  \[
    {\rm Re}_{\mathbb{R}}([G\backslash X])\simeq {\rm Re}_{\mathbb{C}}([G\backslash X])^{{\rm hC}_2},
  \]
  \end{description}
  In particular, we have the following equivalent descriptions of the real realization of a quotient stack in motivic homotopy:
  \begin{itemize}
  \item as the geometric realization of the (topological) groupoid $[G\backslash X](\mathbb{R})$ of real points of the quotient stack,
  \item as the homotopy-fixed points of the ${\rm C}_2$-equivariant realization (which is the geometric realization of the topological groupoid of complex points, with ${\rm C}_2$-action by complex conjugation), or
  \item as colimit of the real points of smooth scheme approximations of $[G\backslash X]$ using an admissible gadget (i.e., real points of an algebraic Borel construction).
  \end{itemize}
\end{theorem}

Modulo some discussion of equivariant and real realization functors, this is ultimately a consequence of the \'etale descent property for quotient stacks. Versions of the homotopy-fixed-point description have been noted in other places, such as \cite{GuillouMayMerling2017} from an equivariant perspective, \cite{acosta} from a GIT perspective (though somewhat hidden in Galois cohomology statements), and most recently \cites{realDM1,realDM2}.

We discuss this and more background on fixed points for group actions on stacks in Sections~\ref{sec:real-pts-fp-gpd}, \ref{sec:coarse} and \ref{sec:hc2-galois}. The complex and real realization of quotient stacks are discussed in Section~\ref{sec:real-pts-hc2}. In Section~\ref{sec:coarse}, we also indicate a proof that the ${\rm C}_2$-equivariant realization of a real quotient stack $[G\backslash *]$ is the universal ${\rm C}_2$-equivariant $G$-bundle.

The description in terms of homotopy fixed points allows very explicit computations of the real realization for classifying spaces ${\rm B}_\et G\simeq [G\backslash*]$ of a linear algebraic group $G$. We will discuss this (as well as connections to Galois cohomology and strong real forms) in Section~\ref{sec:hc2-galois}, cf. in particular Proposition~\ref{prop:bg-formula}.

\begin{corollary}
  Let $G$ be a smooth affine algebraic group over $\mathbb{R}$. We consider $G(\mathbb{C})$ as ${\rm C}_2$-group by means of the Galois automorphism of $G(\mathbb{C})$, which we denote by $\sigma$. Then there is an equivalence
  \[
    {\rm Re}_{\mathbb{R}}({\rm B}_\et G)\simeq \left({\rm B}G(\mathbb{C})\right)^{{\rm hC_2}}\simeq  \bigsqcup_{[g]\in {\rm H}^1({\rm C_2},G)}{\rm B}[G(\C)^{\rm C_2}],
  \]
  where $[g]\in {\rm H}^1({\rm C}_2,G)$ runs through the Galois cohomology set. On the right-hand side,  the group $G(\mathbb{C})$ in the component corresponding to $[g]$ has ${\rm C}_2$-action by the $[g]$-twisted complex conjugation $\si_g=\intt(g)\circ \si$.  
\end{corollary}

The Galois cohomology set ${\rm H}^1({\rm C}_2,G)$ can be identified with the set of strong real forms with a given central invariant; for more information on strong real forms, cf.{}~\cite{adams:taibi} or our discussion in Section~\ref{sec:strong-real-forms}.\footnote{Note that \cite{adams:taibi} denote the Galois cohomology set by ${\rm H}^1(\sigma,G)$, emphasizing the action of the generator of ${\rm C}_2$ by the automorphism $\sigma$ of $G(\mathbb{C})$.} With this terminology, the above result shows that the real realization of the geometric/\'etale classifying space ${\rm B}_\et G$ for a linear algebraic group $G$ over $\mathbb{R}$ is a disjoint union of classifying spaces of strong real forms of $G$ with given central invariant. In particular, this means that the real forms of $G$ with the same central invariant have equivalent real realizations.

In fact, using non-abelian cohomology and the theory of bitorsors, we discuss more generally the invariance of the classifying stack $[G\backslash*]$ under strong inner forms in Theorem~\ref{thm:strong-inner--eq_formulations}. As consequences, we get that cohomology and real realization of classifying spaces for strong inner forms will be equivalent, cf.{}~Corollary~\ref{cor:cohom-inv-strong-forms}, cf. also Section~\ref{sec:bitorsors} for discussion of the relation of strong real forms and strong inner forms. 

\begin{corollary}
  Let $G$ and $G'$ be linear algebraic groups over a field $F$ which are strong inner forms of each other, and let $E$ be a cohomology theory representable in stable motivic homotopy. Then any $(G,G')$-bitorsor induces isomorphisms
  \[
  E^*([G\backslash\pt])\cong E^*([G'\backslash\pt]).
  \]
  Similarly,  over $F=\mathbb{R}$, the equivariant real cycle class map of Theorem~\ref{thm:main-real-cycle} is invariant under strong inner forms. 
\end{corollary}

\subsection{Examples: orthogonal and spin groups}
We discuss several instances of these descriptions and their consequences for equivariant real cycle class maps in Section~\ref{sec:examples}. One example is the case $G=\mu_2$, in which case the real realization is
\[
{\rm Re}_{\mathbb{R}}\left({\rm B}_\et\mu_2\right)\simeq \mathbb{RP}^\infty\sqcup\mathbb{RP}^\infty,
\]
see the more general discussion of finite group examples in Section~\ref{sec:finite}. More generally, in the particular case of orthogonal groups, we get the following description, cf. the discussion in Section~\ref{sec:orthogonal} and in particular Theorem~\ref{thm:realization-orthogonal}:
\begin{equation}
  \label{eq:realization-bon}
{\rm Re}_{\mathbb{R}}({\rm B}_\et{\rm O}(n))\simeq \bigsqcup_{p+q=n}{\rm B}{\rm O}(p,q),
\end{equation}
where ${\rm O}(p,q)$ on the right-hand side denotes the indefinite orthogonal group associated to the symmetric bilinear form over $\mathbb{R}$ of rank $p+q$ and signature $p-q$ (viewed as Lie group). A similar statement is true for special orthogonal groups. This formula provides the key intuition that the characteristic classes for orthogonal groups in Witt-sheaf cohomology should be related to Pontryagin and Euler classes of the strong real forms. Alternatively, this shows that the characteristic classes which are not in the image of restriction along ${\rm B}_\et{\rm O}(n)\to{\rm BGL}_n$ are detected on the fiber ${\rm GL}_n/{\rm O}(n)$. Note that the fiber is the space of invertible symmetric matrices, and the space of real points is
\[
\op{Sym}(n;\mathbb{R})\simeq \bigsqcup_{0\leq k\leq n}\op{Gr}_k(n;\mathbb{R}).
\]
For a proof of this homotopy equivalence, cf. \cite{jones}; essentially, the components are the orbits of $\op{diag}(-1_k,1_{n-k})$, and the identification maps a symmetric matrix to the subspace spanned by eigenvectors for negative eigenvalues.\footnote{MW would like to acknowledge discussions with Anubhav Nanavaty about varieties of symmetric matrices.}

The takeaway is that, in some sense, Witt-sheaf cohomology knows about the real forms of orthogonal groups. This can also be seen in the real realization: for an algebraic ${\rm O}(n)$-torsor over a real scheme $X/\mathbb{R}$, the real realization over each component of $X(\mathbb{R})$ will be an ${\rm O}(p,q)$-torsor, but the decomposition $p+q=n$ can vary from one real component to the next. 

Combining the equivariant real cycle class map with the description of the real realization of a quotient stack can produce insights into the Witt-sheaf cohomology for classifying spaces. We formulate one specific consequence concerning Witt-sheaf cohomological invariants\footnote{As a note on terminology, \emph{Witt-sheaf cohomological invariants} for a group $G$ will mean degree 0 Witt-sheaf cohomology classes on the classifying space, i.e., elements of ${\rm H}^0({\rm B}_\et G,{\bf W})$.}  of real reductive groups, which by our results are related to strong real forms, cf.{}~the discussion in Section~\ref{sec:cohom-inv}:

\begin{corollary}
  \label{cor:spin-cohom-inv}
  Let $G$ be a split reductive group over $\mathbb{R}$. Then there is an isomorphism
  \[
  {\rm H}^0({\rm B}_\et G,\mathbf{W}[1/2])\cong {\rm H}^0({\rm B}G(\mathbb{C})^{\rm hC_2},\mathbb{Z}[1/2]).
  \]
  Consequently, up to torsion, the number of Witt-sheaf cohomological invariants for $G$-torsors over $\mathbb{R}$ equals the number of strong real forms of $G(\mathbb{R})$ with trivial central invariant.

  In particular, for the split spin groups, we get the following formulas for the number of Witt-sheaf cohomological invariants (up to torsion) from the tables in \cite{adams:taibi}:
  \begin{align}
  \op{rk}_{\mathbb{Z}}{\rm H}^0({\rm B}_\et \op{Spin}(n,n+1),\mathbf{W})&\cong
  \left\lfloor \frac{2n+1}{4}\right\rfloor+\left\{\begin{array}{ll}
  2 & n\equiv 0,3\bmod 4\\
  1 & n\equiv 1\bmod 4\\
  0 & n\equiv 2\bmod 4\\
  \end{array}\right.\\
  \op{rk}_{\mathbb{Z}}{\rm H}^0({\rm B}_\et \op{Spin}(n,n),\mathbf{W})&\cong
  \left\lfloor \frac{n}{2}\right\rfloor+\left\{\begin{array}{ll}
  3 & n\equiv 0\bmod 4\\
  1 & n\equiv 1\bmod 4\\
  0 & n\equiv 2,3\bmod 4\\
  \end{array}\right.
  \end{align}
\end{corollary}

This result for the spin groups seems to be new. To the best of our knowledge, Witt-sheaf cohomological invariants for $\op{Spin}_n$ have not been determined in general, the closest thing is the computation of mod 2 cohomological invariants for $n\leq 14$, cf.{}~\cite{garibaldi:spin} and the discussion of cohomological invariants in the thesis of Weinzierl \cite{weinzierl}, or Section~\ref{sec:cohom-inv} for further discussion.\footnote{Some of these computations in loc.{}~cit.{}~are for general fields, but some of these computations are for fields containing a square root of $-1$. It can be checked that our formulas agree with the formulas in loc.{}~cit.{}~whenever applicable. The consistent difference between our formulas and the ones in loc.{}~cit.{}~is due to the fact that the trivial cohomological invariant $1\in{\rm H}^0({\rm B}_\et G,{\bf W})$ contributes to our formulas, but is not counted as a cohomological invariant in loc.{}~cit..}

\subsection{Conjectural picture of Witt-sheaf cohomology of classifying spaces}
\label{sec:conjectures}

Having the equivariant real cycle class map provides some first approximation of what the Witt-sheaf cohomology of classifying spaces should look like. This is a good start, but of course we would ultimately like to understand on a more conceptual level how to produce cohomology classes, or possible general descriptions of Witt-sheaf cohomology of classifying spaces (think of the usual Weyl-group formulas ${\rm H}^\bullet({\rm B}G;\mathbb{Q})\cong{\rm H}^\bullet({\rm B}T;\mathbb{Q})^W$ for rational cohomology of semisimple Lie groups). We outline some parts of the bigger picture as far as we understand them for now. 

The question of Witt-sheaf cohomology is already very interesting in degree 0, as we discussed in the previous section. Revisiting the case of orthogonal groups in Equation~\eqref{eq:realization-bon} above, we note that the number of connected components of ${\rm Re}_{\mathbb{R}}{\rm B}_\et {\rm O}(n)$ agrees with Serre's computation of Witt-sheaf cohomological invariants of orthogonal groups in \cite{garibaldi:merkurjev:serre}. We can interpret this in terms of an augmentation morphism
\begin{align}
  \label{eq:augmentation}
  {\rm W}^0_G(\ast)\to {\rm H}^0({\rm B}_\et G,{\bf W})
\end{align}
from the $G$-equivariant Witt group of the point, i.e., the Grothendieck ring of symmetric $G$-representations, to the unramified Witt-group. The morphism can be viewed as the (Zariski or Nisnevich) sheafification morphism for the presheaf of Witt groups, and it is one source of Witt-sheaf cohomological invariants. More precisely, we can interpret Serre's computation as saying that the images of the fundamental ${\rm O}(n)$-representations (of which there are $n+1$ many, given as exterior powers of the defining $n$-dimensional representation) are a basis of the Witt-sheaf cohomological invariants. In light of our results, in particular Corollary~\ref{cor:spin-cohom-inv}, it seems interesting to investigate the connection between Witt-sheaf cohomological invariants and strong real forms (suggested by the equivariant real cycle class map) in more detail. In the same direction, one can wonder about general formulas for Witt-sheaf cohomological invariants for a semisimple group $G$, expressed in terms of Weyl-group invariants on the Witt groups of toral or non-toral elementary abelian 2-subgroups of $G$.

The augmentation morphism in \eqref{eq:augmentation} above also seems to be a good choice for the augmentation morphism in possible Atiyah--Segal-type completion results for Witt groups of classifying spaces, as in \cite{hornbostel:rohrbach:zibrowius} (although at the moment all the known cases of quadratic Atiyah--Segal completion results are for special groups where the augmentation morphism in \eqref{eq:augmentation} really only encodes rank information). As a first guess, one could ask if the natural map
\begin{align}
{\rm W}^0_G(\ast)\to{\rm W}^0({\rm B}_\et G)
\end{align}
is indeed a completion at the augmentation ideal for the sheafification map in Equation~\eqref{eq:augmentation}.

In this context, we can also ask about various possible filtrations on ${\rm W}^0_G(\ast)$ and their relation with Witt-sheaf cohomology. Likely, the Witt-sheaf case is comparable to the corresponding, quite classical, situation for K-theory and Chow groups where different filtrations on representation rings and their relations are discussed e.g.{}~in \cite{karpenko:merkurjev}. From our discussion above, we have the augmentation filtration by powers of the kernel of the sheafification map \eqref{eq:augmentation}. The symmetric representation ring ${\rm W}^0_G(\ast)$ also has a $\lambda$-ring structure described and discussed in \cite{zibrowius:symmetric}, giving rise to a corresponding $\gamma$-filtration. Finally, we should expect some geometric filtration related to the coniveau filtration and the Gersten--Witt spectral sequence of Balmer and Walter \cite{balmer:walter}:\footnote{Note that strictly speaking the assumptions in loc.{}~cit.{}~don't apply to $X={\rm B}_\et G$ and Witt groups don't stabilize in admissible gadgets like Witt-sheaf cohomology groups do. So, while the results of loc.{}~cit.{}~are not immediately applicable, it seems plausible to assume that the relation between Witt groups and Witt-sheaf cohomology extends to classifying spaces, or quotient stacks for that matter.}\footnote{Also, the indexing of the spectral sequence is a bit subtle. The $E_2$-page of the spectral sequence has Witt-sheaf cohomology, as indicated, in rows $q\equiv 0\bmod 4$ and nothing in the other rows.}
\begin{align}
E_2^{p,q}\colon {\rm H}^p(X,{\bf W}(\mathscr{L}))\Rightarrow{\rm W}^{p+q}(X,\mathscr{L}).
\end{align}
As in the Chow-ring case discussed in \cite{karpenko:merkurjev}, we should probably not expect the filtrations to agree, but rather expect quite subtle differences between the filtrations. Nevertheless, as in the K-theory/Chow ring case, it seems a reasonable guess that the three filtrations induce the same topology on the symmetric representation ring ${\rm W}^0_G(\ast)$. In the case of Chow rings, the $\lambda$-ring structure of the representation ring provides Chern classes for representations, producing many interesting elements in ${\rm CH}^\bullet({\rm B}_\et G)$. Very likely, not all Witt-sheaf cohomology classes are characteristic classes; already in degree 0 this is related to the purity question for quadratic forms (for which examples based on classifying spaces of finite 2-groups over the complex numbers will be given somewhere else). Still, the $\lambda$-ring structure and $\gamma$-operations on the symmetric representation ring should be a rich source of classes in the Witt-sheaf cohomology ${\rm H}^\bullet({\rm B}_\et G,{\bf W})$ as characteristic classes of symmetric representations. 

Finally, we could also ask for an analogue of the rational identification of K-theory and Chow rings. For Witt-rings and Witt-sheaf cohomology, the related question would be if there is a rational isomorphism\footnote{We leave open precisely which filtration we consider on the left-hand side. If we think of the isomorphism as the rational degeneration of the Gersten--Witt spectral sequence of \cite{balmer:walter}, then it would be the geometric/coniveau filtration. In the K-theory case, the difference between geometric and $\gamma$-filtration is only torsion and therefore would not be visible in a rational statement.}
\begin{align}
  \label{eq:borel-iso}
{\rm gr}^\bullet{\rm W}^0_G(\ast)\otimes\mathbb{Q}\cong{\rm H}^\bullet({\rm B}_\et G,{\bf W})\otimes\mathbb{Q}.
\end{align}
This question would be strongly related to the Borel character in \cite{deglise:fasel:borel-character}. If true, we could combine it with the equivariant real cycle class map to relate the real representations of a real reductive group $G$ with the cohomologies of its strong real forms (related to the appearance of strong real forms in the representation theory of real Lie groups, cf.{}~\cite{adams:taibi}). In the case of ${\rm SO}(n)$ for small $n$, at least a preliminary dimension count does seem to work out. On the other hand, the case of finite groups suggests some subtleties -- if there is an isomorphism as in \eqref{eq:borel-iso}, then the numbers of real representations and conjugacy classes of involutions would need to agree (which they don't in some examples, cf. the discussion in section~\ref{sec:finite}). Possibly rational isomorphisms using the Borel character require to work in a proper hermitian K-theory setting, and something is lost in the translation to Witt groups and Witt-sheaf cohomology.

We will discuss some of the aspects of symmetric representation theory in our examples in Section~\ref{sec:examples}, in particular in the case of orthogonal groups. Generally, it would also be interesting to see if there is a version of this story for general quotient stacks, and what that should look like. 

\subsection{Further directions}

As we will discuss in Section~\ref{sec:examples}, the equivariant real cycle class map suggests a concrete description for the Witt-sheaf cohomology of ${\rm B}_\et {\rm O}(n)$. Based on the stratification method as in \cite{RojasVistoli}, it is possible to carry out the actual computation of the Witt-sheaf cohomology ring of the (special) orthogonal groups (over a general base field of characteristic $\neq 2$). Essentially, the Witt-sheaf cohomology ring looks like the cohomology ring of the real realization $\bigoplus_{p+q=n}{\rm BO}(p,q)$, up to some spurious 2-torsion arising from some nontrivial boundary map computations. As a more long-term goal, one could hope that a computation of quadratic cohomologies of ${\rm B}_\et {\rm O}(n)$ could improve our understanding of operations in variants of hermitian K-theory. 

While we have focused on classifying spaces, i.e., quotient stacks for group actions $G\looparrowright \ast$, there are loads of interesting examples of other quotient stacks one could investigate using the equivariant real cycle class map. One particularly fun example, also related to geometric representation theory, is the character/inertia stack $[G/\!/_{\rm ad}G]$ associated to the adjoint action of $G$ on itself. Taking inspiration from the Freed--Hopkins--Teleman theorem and its Real version in \cite{fok:verlinde}, we could think of ${\rm W}_G(G)$ as an algebraic version of the Real Verlinde algebra. In a version of the classifying-space-story outlined above, we could consider the sheafification augmentation ${\rm W}_G(G)\to {\rm H}^0([G/\!/_{\rm ad}G],{\bf W})$ and hope that the associated graded for the corresponding augmentation filtration would rationally agree with ${\rm H}^\bullet([G/\!/_{\rm ad}G],{\bf W})$. The equivariant real cycle class map would then suggest a relation between the algebraic Real Verlinde algebra and the combined Real Verlinde algebras of its strong real forms, thus providing an interesting interplay between Real Verlinde algebras of strong real forms. We believe it could be fun to work out the details of this.

\subsection*{Structure of the paper}

We begin by covering some background related to stacks in Section~\ref{sec:prelims}, and we recall the realization of stacks in motivic homotopy in Section~\ref{sec:stacks-motivic}. We recall the description of the real points of a stack in terms of fixed-point groupoids in Section~\ref{sec:real-pts-fp-gpd}. Then we discuss ${\rm C}_2$-equivariant homotopy and (equivariant and real) realization functors in Section~\ref{sec:equivariant-htpy-realization}. In Section~\ref{sec:coarse} we offer a coarse homotopy perspective on equivariant principal bundles which is well-suited to study ${\rm C}_2$-equivariant realization of \'etale homotopy orbits. Then we prove the main results on equivariant and real realization of quotient stacks in Section~\ref{sec:real-pts-hc2}. We discuss the relation between homotopy fixed points, Galois cohomology and strong real forms in Section~\ref{sec:hc2-galois}, where we also recall the main formula for the real realization of \'etale classifying spaces. Section~\ref{sec:equivariant-cycle-class} then describes the equivariant real cycle class map and establishes its compatibilities with natural structures in (Borel-)equivariant cohomology. In Section~\ref{sec:bitorsors}, we prove various results on invariance of quotient stacks (as well as their realizations and cycle class maps) under strong inner forms. Finally, we discuss a number of specific example computations in Section~\ref{sec:examples}. All the other interesting stuff alluded to will be deferred to sequels. Or prequels. Or maybe not. You know.

\subsection*{Acknowledgements:} MW would like to thank Fr\'ed\'eric D\'eglise, Jens Eberhardt, Vincent Gajda, Jens Hornbostel, Marc Levine and Kirsten Wickelgren for discussions on the topics of the paper. We thank Jens Hornbostel for helpful comments on a preliminary draft of the paper.

\section{Preliminaries on (quotient) stacks}
\label{sec:prelims}

In this preliminary section, we'll be collecting some of the material on (quotient) stacks we'll need in the paper. In particular, we  will need real and complex points of quotient stacks and various types of morphisms of quotient stacks for our later discussions. Throughout the paper, we work with schemes over a base field $F$; in the context of motivic homotopy or quadratic forms, the base fields are assumed perfect of characteristic $\neq 2$.

\subsection{Groups and actions}

When dealing with group actions, we will follow the notation of \cite{BGK}*{I.2.1}:

\begin{definition}
  By a \emph{variety with action} $(G\looparrowright X)$ or a $G$-variety, we mean a triple $(G,X,\rho)$ consisting of a linear algebraic group $G$, a (quasi-projective) variety $X$ and a $G$-action $\rho\colon G\times X\to X$. A \emph{morphism $(\varphi,f)\colon(H\looparrowright Y)\to (G\looparrowright X)$ of varieties with action} is a pair consisting of a homomorphism $\varphi\colon H\to G$ of linear agebraic groups and a $\varphi$-equivariant morphism $f\colon Y\to X$ of varieties. 
\end{definition}

\begin{remark}
  We will typically consider left actions $G\looparrowright X$ and accordingly write scheme quotients or quotient stacks as $G\backslash X$ or $[G\backslash X]$, respectively.
\end{remark}

We recall the classical notion of torsors which is the analogue of principal bundles in algebraic geometry. 

\begin{definition}
  \label{def:torsors}
  For an algebraic group $G$, a \emph{$G$-torsor} over a base $B$ is a variety with action $\rho\colon G\looparrowright E$ such the induced morphism $(\rho,{\rm pr}_2)\colon G\times E\to E\times_BE$ is an isomorphism. 
\end{definition}

\begin{remark}
  The isomorphism in the condition encodes that the $G$-action on the fiber is free and transitive. One typically requires that $G$-torsors are locally trivial (i.e., locally isomorphic to the trivial $G$-torsor, $G$ acting on itself by multiplication) for a suitable topology. For example, when we speak of \'etale $G$-torsors later, this means that the torsor is locally trivial in the \'etale topology. 
\end{remark}

For the construction of approximations of Borel constructions, we use the same conventions as in \cite{dilorenzo:mantovani}*{Section~2.2.2}: 

\begin{convention}
  \label{conventions-actions}
Let $F$ be a perfect field of characteristic $\neq 2$, $G$ be a smooth affine algebraic group of finite type over $F$. Throughout the paper, we assume $X$ to be a finite type scheme over $F$, such that one of the following conditions is satisfied:
\begin{enumerate}[label=\roman*)]
\item The reduced scheme $X_{\rm red}$ is quasi-projective and the action of $G$ is linearized w.r.t. some quasi-projective embedding.
\item The group $G$ is connected and $X_{\rm red}$ embeds equivariantly as closed subscheme of a normal variety.
\item The group $G$ is special in the sense of Serre, i.e., \'etale-locally trivial torsors are already Nisnevich-locally trivial. 
\end{enumerate}
\end{convention}

In fact, for all the key examples, our schemes will be quasi-projective of finite type over a field (we will occasionally call such schemes \emph{varieties}). For the study of real realization functors and equivariant real cycle class maps, we will mostly be interested in varieties over $\mathbb{R}$. Some of the general discussion of equivariant Witt-sheaf cohomology is valid over perfect base fields of characteristic $\neq 2$. 

For two varieties with action $(G\looparrowright X)$ and $(U\looparrowleft G)$, i.e., the latter one with right $G$-action, we have the \emph{balanced product}
\[
U\times_{/G} X:=G\backslash (U\times X)
\]
with the $G$-action given by $g\cdot(u,x)=(ug^{-1},gx)$, i.e., we make the identifications $(ug,x)\sim(u,gx)$ in the orbit space $G\backslash(U\times X)$. Again, we follow the notation $U\times_{/G}X$ of \cite{BGK} to distinguish the balanced product from pullbacks. 

\subsection{Quotient stacks, morphisms and relevant examples}
\label{sec:quot-stacks}

We now recall the basics concerning quotient stacks we will need. For general reference, see \cite{stacks-project}*{Tags~0268,02ZH,044O} and go from there. In particular, stacks are categories fibered in groupoids (over the base category $\mathbf{Sch}_F$), and morphisms of stacks are functors compatible with that structure.

Let $G\looparrowright X$ be a variety with action over the base field $F$. The associated quotient stack $[G\backslash X]$ is the category whose objects are spans $S\xleftarrow{p} P\xrightarrow{f} X$ where $p\colon P\to S$ is a $G$-torsor and $f\colon P\to X$ is a $G$-equivariant morphism. A morphism from $T\xleftarrow{q} Q\xrightarrow{g} X$ to $S\xleftarrow{p} P\xrightarrow{f} X$ is the datum of a map $\psi\colon T\rightarrow S$ of $F$-schemes and of a $G$-equivariant map $\varphi\colon Q\rightarrow P$ such that $f\circ \varphi=g$ and such that the square 
\begin{equation*}
  \begin{tikzcd}
    Q \arrow[r,"\varphi"] \arrow[d,"p'"'] & P \arrow[d,"p"]\\
    T \arrow[r,"\psi"'] & S
  \end{tikzcd}
\end{equation*}
is cartesian.

Note that for every $F$-scheme $S$, giving an object $(S\leftarrow P \rightarrow X)\in [G\backslash X](S)$ is equivalent to giving the datum of a $G_S$-torsor $P$ over $S$ together with a $G_S$-equivariant map of $S$-schemes $P \ra X_S$, and similarly for morphisms. 

The projection
\[
  [G\backslash X]\ra \mathbf{Sch}_F, \quad (S\leftarrow P\rightarrow X)\mapsto S
\]
equips $[G\backslash X]$ with the structure of a category fibered in groupoids. The fiber over a scheme $S$ is the groupoid $[G\backslash X](S)$ whose objects are spans $S \leftarrow P\rightarrow X$, and whose morphisms are $G$-equivariant isomorphisms $\varphi\colon Q\to P$ over $S$ compatible with the maps to $X$. If $T\to S$ is a map of $F$-schemes we have an induced functor
\[
  [G\backslash X](S) \ra [G\backslash X](T)
\]
given on objects by 
\[
\quad \left(S\overset{p}{\leftarrow} P\overset{f}{\rightarrow} X\right) \mapsto \left(T\overset{p'}{\leftarrow} P' \overset{f'}{\rightarrow} X\right),
\]
where $P':=T\times_S P$, the morphism $p'\colon P' \rightarrow T$ is the pull-back of the arrow $p\colon P\ra S$ along $T\ra S$ and where $f'$ is the composition of the projection $P'\rightarrow P$ and $f\colon P\ra X$. In a picture 
\begin{equation*}
  \begin{tikzcd}[column sep=4pt, row sep=8pt]
    &  X \\
    P \arrow[ur,"f"] \arrow[dr,"p"'] & \\
    & S
  \end{tikzcd}
  \mapsto
  \begin{tikzcd}[column sep=4pt, row sep=8pt]
    & & & X \\
    P' \arrow[dr,"p'"'] \arrow[rr] \arrow[urrr,"f'"]& & P \arrow[ur,"f"'] \arrow[dr,"p"]& \\
    & T \arrow[rr] & & S
  \end{tikzcd}
\end{equation*}
The functor is defined on morphisms by declaring that it sends a morphism $\varphi\colon P\ra Q$ of $G$-torsors over $S$, compatible with maps to $X$, to the associated morphism 
\[T\times_S \varphi\colon P'=T\times_S P \ra  T\times_S Q = Q'.\]

\begin{example}
  The most important example for us at this point is the trivial action $G\looparrowright\ast$, giving rise to the classifying stack $[G\backslash *]$ of a linear algebraic group $G/F$. In this case, for a smooth $F$-scheme $S$, the groupoid $[G\backslash *](S)$ is identified with the groupoid of \'etale $G$-torsors over $S$.\footnote{This uses smoothness of the base $S$, in general $G$-torsors would only be locally trivial in the fppf topology, cf.\cite{stacks-project}*{Tag 044O}.} We have $\pi_0([G\backslash *](S))\cong{\rm H}^1_\et(S,G)$, and $\pi_1([G\backslash*](S),P)\cong{\rm Aut}_S(P)$, i.e., the fundamental group of a connected component corresponding to the $G$-torsor $P/S$ is the group of equivariant $S$-automorphisms of the torsor.
\end{example}
  
Given a morphism of varieties with action $(\varphi,f)\colon (H\looparrowright Y)\to (G\looparrowright X)$ over the base field $F$, there is an induced morphism of quotient stacks $f\colon[H\backslash Y]\to [G\backslash X]$. For an $F$-scheme $S$, the morphism $[H\backslash Y](S)\to[G\backslash X](S)$ is given as follows:
\[
\left(S\xleftarrow{p}P\xrightarrow{g}Y\right)\mapsto \left(S\xleftarrow{q}(G\times_{/H}P)\xrightarrow{\tilde{g}} X\right),
\]
i.e., it sends a span $S\xleftarrow{p}P\xrightarrow{g}Y$, consisting of an $H$-torsor $p\colon P\to S$ and an equivariant map $g\colon P\to Y$, to the span consisting of the induced $G$-torsor $q\colon (G\times_{/H} P)\to S$ together with the morphism $\tilde{g}\colon G\times_{/H}P\to X$ which is adjoint (under the induction-restriction adjunction) to the $H$-equivariant morphism $P\xrightarrow{g} Y\xrightarrow{f} X$ where $X$ is viewed as an $H$-scheme via restriction along $\varphi\colon H\to G$.

For our purposes, a morphism of stacks $f\colon \mathcal{X}\to\mathcal{Y}$ over $F$ is \emph{representable} if for each $F$-scheme $S$ and morphism $g\colon S\to\mathcal{Y}$, the pullback
\[
\xymatrix{
  \mathcal{X}\times_{\mathcal{Y}} S\ar[r] \ar[d]_{\tilde{f}} & \mathcal{X} \ar[d]^f\\
  S\ar[r]_g & \mathcal{Y}
}
\]
is again a scheme. In this case, we say that $f\colon \mathcal{X}\to\mathcal{Y}$ is \emph{smooth/proper/finite type/etc} if for each $g\colon S\to\mathcal{Y}$ the base-change $\tilde{f}\colon\mathcal{X}\times_{\mathcal{Y}}S\to S$ is also smooth/proper/finite type/etc.

We will discuss some criteria for morphisms of quotient stacks to be proper and smooth in Example~\ref{ex:quotient-stack-morphism-properties} below.

\begin{example}
  \label{ex:representable}
  Note that a morphism $(H\looparrowright*)\to(G\looparrowright*)$ of varieties with actions is nothing but a group homomorphism. The corresponding morphism $[H\backslash*]\to[G\backslash *]$ of quotient stacks is representable if $H\leq G$ is a subgroup. In this case, we can think of the morphism of quotient stacks as a $G/H$-fiber bundle over $[G\backslash*]$: for any smooth scheme $X$ and morphism $X\to[G\backslash\pt]$, the pullback of $[H\backslash*]\to[G\backslash *]$ to $X$ is a $G/H$-fiber bundle over $X$, locally trivial in the \'etale topology.

  For a general group homomorphism $\varphi\colon H\to G$, the corresponding morphism of quotient stacks will not be representable. For an exact sequence $1\to K\to H\to G\to 1$ of groups, we should get a fiber bundle $[K\backslash*]\to[H\backslash*]\to[G\backslash*]$ of quotient stacks, and the representability of the morphism $[H\backslash*]\to[G\backslash*]$ would require $[K\backslash*]$ to be a scheme. 

  For an inclusion of subgroups $H\leq G$ the corresponding representable morphism $[H\backslash \ast]\to [G\backslash \ast]$ is smooth/proper if $G/H$ is. The relative dimension of this morphism is $d=\dim(G/H)$. 
\end{example}

We close the section with two well-known and relevant isomorphisms of quotient stacks, related to the quotient and induction equivalences in equivariant cohomology. In the situation of algebraic Borel constructions, the proofs of these can be found in \cite{krishna}*{Corollaries~2.6 and 2.7}.

\begin{proposition}[quotient and induction equivalence for stacks]
  \label{prop:quotient-induction-stacks}\,%Hack
  \begin{enumerate}
  \item Let $G\looparrowright X$ be a variety with action and let $N\subset G$ be a closed normal subgroup such that the restricted action $N\looparrowright X$ is free. Then the morphism $(\pi,p)\colon (G\looparrowright X)\to (G/N\looparrowright N\backslash X)$ induces an isomorphism of quotient stacks
    \[
      [G\backslash X]\cong \left[\left(G/N\right)\backslash\left(N\backslash X\right)\right].
    \]
  \item Let $i\colon H\hookrightarrow G$ be the inclusion of a smooth closed subgroup into a smooth affine algebraic group $G$ and let $H\looparrowright X$ be a variety with action such that the quotient $G\times_{/H}X$ of the diagonal $H$-action on $G\times X$ exists as a smooth quasi-projective variety. Then the obvious morphism $(i,s)\colon (H\looparrowright X)\to (G\looparrowright G\times_{/H}X)$ induces an isomorphism of quotient stacks
  \[
    [H\backslash X]\cong [G\backslash (G\times_{/H}X)].
  \]
  \end{enumerate}
\end{proposition}

\begin{proof}
  We indicate the proof of (1),  the arguments for (2) are similar.

  For a scheme $T$, there is a natural functor
  \[
    [G\backslash X](T)\to[(G/N)\backslash(N\backslash X)](T)\colon
    \left(T\xleftarrow{p} P\xrightarrow{f}X\right)\mapsto \left(T\xleftarrow{\overline{p}} P/N\xrightarrow{\overline{f}}N\backslash X\right)
  \]
  where $p\colon P\to T$ is a $G$-torsor, $f$ is $G$-equivariant, and $\overline{f}$ is the induced map on quotients by the free $N$-actions. The functors are natural in $T$, giving the natural morphism of stacks $[G\backslash X]\to[(G/N)\backslash(N\backslash X)]$. Since we have a morphism of stacks, it suffices to check locally in the \'etale topology that the morphism is an equivalence, i.e., we can assume that \'etale $G$-torsors over $T$ are trivial. On Hom-sets, the functor sends a morphism of spans in $[G\backslash X](T)$, corresponding to a commutative triangle 
  \[
  \xymatrix{
    G\times T \ar[rr]^{(-)\cdot g^{-1}} \ar[rd]_{x}&& G\times T \ar[ld]^{g\cdot x}\\
    & X
  }
  \]
  to the corresponding triangle for $(G/N)$ and $N\backslash X$, with the top horizontal morphism $(-)\cdot g^{-1}$ replaced by $(-)\cdot\overline{g}^{-1}$.

  The functor is full because any morphism between $G/N$-orbit maps in $N\backslash X$, given by right multiplication by an element $\overline{g}^{-1}\in G/N$ can be lifted to a morphism between $G$-orbit maps in $X$, given by right multiplication by $\overline{g}^{-1}\in G$.

  The functor is faithful: given two morphisms $(-)\cdot g_1^{-1},(-)\cdot g_2^{-1}$ between $G$-orbit maps to $X$, if the corresponding morphisms between $G/N$-orbit maps to $N\backslash X$ agree, then $g_1\cdot g_2^{-1}\in N$ is in the stabilizer of a $G$-orbit on $X$. But since $N$ acts freely on $X$, this implies $g_1=g_2$.

  The functor is essentially surjective: given an orbit map $G/N\times T\to N\backslash X$, i.e.,  a scheme morphism $T\to N\backslash X$, we can lift to a scheme morphism $T\to X$ because the $N$-torsor $T\times_{N\backslash X} X$ is trivial. The given orbit map can then easily be lifted to an orbit map $G\times T\to X$.
\end{proof}

\section{Stacks in motivic homotopy}
\label{sec:stacks-motivic}

In this section, we recall different ways of associating objects in unstable motivic homotopy theory to group actions $G\looparrowright X$. There are at least two popular options for this. One is the Nisnevich homotopy orbit space, which most commonly can be described in terms of a simplicial Borel construction which we discuss in Section~\ref{sec:stacks-motivic-simplicial}. The other one realizes the quotient stack associated to the group action as an \'etale homotopy orbit space. This can be described as ind-smooth scheme, and we discuss this in Section~\ref{sec:stacks-motivic-approx}. For special groups $G$, both approaches coincide. The discussion here is going to be central in the second half of the paper, as a key input in our description of the real realization of quotient stacks in Section~\ref{sec:real-pts-hc2},  the definition and properties of the equivariant real cycle class map in Section~\ref{sec:equivariant-cycle-class} as well as the explicit example computations in Section~\ref{sec:examples}. 

\begin{remark}
  Before the more detailed discussion of homotopy orbits, we remark that one can consider a more general setting of group actions $G\looparrowright X$ in $\infty$-topoi. Taking the appropriate quotient yields the (homotopy) orbit space $X_{{\rm h}G}$ which for $X=\pt$ provides an object classifying principal $G$-bundles in the given $\infty$-topos. For a version formulated in simplicial presheaves, cf.{}~\cite{wendt:jhrs}, for an $\infty$-categorical reformulation of that, cf.{}~\cite{nss:bundles}. For us, the relevant settings would be $\infty$-topoi of Nisnevich resp.~ \'etale sheaves on ${\bf Sm}_F$. The non-formal part of our results below is that we are interested in realization functors in the Nisnevich setting, but applied to the \'etale homotopy orbit spaces, hence the general results aren't directly applicable.
\end{remark}

\subsection{Stacks in motivic homotopy}

We begin by recalling (from \cite{choudhury:deshmukh:hogadi:stacks}) a definition of unstable motivic homotopy type of stacks, based on viewing the stack as a sheaf of groupoids and applying the nerve construction to get a simplicial sheaf. The following is defined in Remark~2.1 of loc.{}~cit.

\begin{definition}
  \label{def:stack-motivic-nerve}
  For an algebraic stack $\mathcal{X}$ (on the category ${\bf Sm}_F$ of smooth schemes over a field $F$), there is an associated sheaf of groupoids also denoted by $\mathcal{X}$. Taking sectionwise nerve of the sheaf of groupoids defines a simplicial sheaf and therefore an object in the unstable motivic homotopy category ${\bf H}(F)$. This is the unstable motivic homotopy type of $\mathcal{X}$, and it is again denoted by $\mathcal{X}$.
\end{definition}

\begin{remark}
  Stabilizing we also get a stable motivic homotopy type $\Sigma^\infty_{\mathbb{P}^1}\mathcal{X}$ in $\mathbf{SH}(F)$, and similarly a motive ${\rm M}(\mathcal{X})$ in $\mathbf{DM}(F)$. This is the definition of motives of stacks, cf.{}~\cite{choudhury:deshmukh:hogadi:stacks}*{Definition~2.4}. The definitions apply more generally, not just to quotient stacks. For our purposes in this paper, we will only need quotient stacks, but the description of real realization for quotient stacks can then naturally be extended to algebraic stacks which are Nisnevich-locally quotient stacks. In any case, the usual quasi-projectivity or quasi-separatedness assumptions in the literature on motives of stacks are all satisfied in our intended quotient stacks setting.
\end{remark}

\subsection{Homotopy orbit spaces and simplicial approach}
\label{sec:stacks-motivic-simplicial}

We next discuss the simplicial approach, considering the Nisnevich homotopy orbits, aka simplicial Borel construction, associated to a group action $G\looparrowright X$ in ${\bf Sm}_F$. The quotient stack $[G\backslash X]$ associated to the group action is given by the ``stackification'' of the simplicial Borel construction. Our discussion including choice of notation is based loosely on \cite{krishna}.

\begin{definition}
  \label{def:action_groupoid}
  Let $X$ be a set with a left $G$-action. The \emph{action groupoid} $\mathbf E_G(X)$ is the groupoid with object set $X$, and where the set of morphisms $x\rightarrow y$ (for $x,y\in X$) is the set of $g\in G$ such that $y=g\cdot x$. We denote by $(g,x)$ the arrow $g\colon x\rightarrow g\cdot x$. The composition is defined by setting $(h,y)\circ(g,x)=(hg,x)$, assuming that $y=g\cdot x$.

  We will denote by ${\rm E}_G(X)$ the \emph{nerve of the action groupoid}. Concretely, ${\rm E}_G(X)$ is the simplicial set with $n$-simplices $G^{\times n}\times X$, with face maps
  \[d_i\colon  G^{\times n} \times X \rightarrow  G^{\times {n-1}}\times X\colon
  (g_{n},\dots,g_{1},x)\mapsto 
  \begin{cases}
    (g_{n},\dots,g_{2},g_{1}\cdot x) & $i=0$,\\
    (g_{n}, \dots, g_{i+2},g_{i+1}g_i,g_{i-1},\dots, g_{1},x) &  0<i<n,\\
    (g_{n-1},\dots,g_{1},x) &   $i=n$.\\
  \end{cases}\]
  and with degeneracy maps
  \[s_i\colon G^{\times n}\times X \ra G^{\times {n+1}}\times X\colon (g_{n},\dots,g_{1},x)\mapsto (g_{n},\dots,g_{i+1},e,g_{i},\dots,g_{1},x) \textrm{ for } i=0,\dots,n.\]
\end{definition}

\begin{remark}
  \begin{itemize}
  \item The construction ${\rm E}_G(-)$ is functorial in the category of left $G$-sets, and recovers the usual construction of ${\rm B}G={\rm E}_G(\ast)$ and of ${\rm E}G={\rm E}_G(G)$, once $G$ is endowed with the left action on itself by translation. 
  \item The definition works more generally, for a category $\mathbf{C}$, a group object $G\in\mathbf{C}$ and an object $X\in\mathbf{C}$ with a $G$-action. The action groupoid is then a groupoid object in $\mathbf{C}$, and its nerve is a simplicial object in $\mathbf{C}$. There is also a straightforward $\infty$-categorical formulation. 
  \item This applies in particular to varieties/schemes with action $G\looparrowright X$ as defined above. We note that \cite{krishna} uses the notation $X_G^\bullet={\rm E}_G(X)$ which is closer to homotopy orbit notation $X_{{\rm h}G}$. 
  \end{itemize}
\end{remark}

\begin{remark}
The simplicial set ${\rm E}_G(G)$ can be also constructed in a different way, cf. also \cite{krishna}*{Section~3}. One can consider the simplicial set $\underline n \mapsto \mathrm{Hom}_{\mathbf{Set}}(\underline n,G)=\check C(G/\ast)$, i.e., the \v{C}ech nerve of the map $G\to\ast$. Explicitly we have $\check C(G/\ast)_n=G^{n+1}$ with face maps
\[d_i\colon  G^{n+1} \rightarrow G^n\colon (g_0, \dots,g_n) \mapsto (g_0, \dots,g_{i-1},\hat g_i,g_{i+1},g_n)
\]
and degeneracy maps
\[s_i\colon G^{\times n+1} \rightarrow G^{\times {n+2}}\colon (g_0,\dots,g_n)\mapsto (g_{0},\dots,g_{i},g_{i},\dots,g_{n-1},g_n)\]
for $i=0,\dots,n$. The map $\check C(G/\ast) \rightarrow {\rm E}_G(G)$ defined on $n$-simplices by
\[\check C(G/\ast)_n \rightarrow {\rm E}_G(G)_n,\quad (g_0, \dots,g_n)\mapsto (g_0g_1^{-1},g_1g_2^{-1},\dots,g_{n-1}g_n^{-1},g_n)\]
is a simplicial isomorphism. Along this isomorphism the component-wise left translation action of $G$ on $\check C(G/\ast)$ corresponds to the left action of $G$ on ${\rm E}_G(G)$ given on $n$-simplices by 
\[(h,(g_0, \dots,g_n))\mapsto (c_h(g_0),\dots,c_h(g_{n-1}),hg_n).\]
On the other hand the component-wise right translation action of $G$ on $\check C(G/\ast)$ corresponds to the right action of $G$ on ${\rm E}_G(G)$ given on $n$-simplices by 
\[((g_0, \dots,g_n),h)\mapsto (g_0, \dots,g_nh).\]
For every left $G$-set $X$ we have natural simplicial isomorphisms
\[\check C(G/\ast)\times X \rightarrow {\rm E}_G(G)\times X,
\quad (g_0,\dots,g_n,x)\mapsto (g_0g_1^{-1},g_1g_2^{-1},\dots,g_{n-1}g_n^{-1},g_n,x)\]
\[{\rm E}_G(G)\times X\rightarrow {\rm E}_G(X),
\quad (g_0,\dots,g_n,x)\mapsto (g_0,\dots,g_{n-1},g_n\cdot x)\]
which are equivariant for the "mixed" action of $G$ on the first two, and for the trivial action on the third set. In particular when $X=\ast$ this gives
\[\check C(G/\ast)/G\simeq {\rm E}_G(G)/G\simeq {\rm E}_G(\ast)={\rm B}G.\]
\end{remark}

\begin{remark}
  \label{rmk:relation_with_Breen_GF1}
  Our notational convention is the reverse of that of Breen in \cite{GF1}. Whenever $X$ is a left $G$-set, the left action of $G$ on $X$ defines tautologically a right action of $G^{\op o}$ on $X$. The action groupoids $\mathbf E_G(X)$, $\mathbf E_{G^{\op o}}(X)$ for the left action of $G$ and the right action of $G^{\op o}$ respectively are equal, and so ${\rm E}_G(X) ={\rm E}_{G^{\op o}}(X)$. Definition \ref{def:action_groupoid} (with respect to $G\looparrowright X$) and formula (3.4.2) of \cite{GF1} (with respect to $ X\looparrowleft G^{\op o}$) thus define isomorphic simplicial objects, once the reversed choice of order of factors is taken into account.
\end{remark}

For the relation between the simplicial Borel construction, i.e., the nerve of the action groupoid, and the quotient stack, there is the following comparison statement, cf.{}~\cite{krishna}:

\begin{proposition}
  \label{prop:simplicial-comparison}
  Let $G\looparrowright X$ be a variety with action over the field $F$. Then there is a natural morphism of simplicial presheaves ${\rm E}_G(X)\to[G\backslash X]$ exhibiting $[G\backslash X]$ as the \'etale stackification of ${\rm E}_G(X)$. The induced morphism ${\rm E}_G(X)\to[G\backslash X]$ in the homotopy category $\mathbf{H}(F)$ is an isomorphism if $G$ is special, i.e., if all \'etale torsors over smooth schemes are already Zariski-locally trivial.
\end{proposition}

\begin{proof}
  For every $F$-scheme $T$ there is always a functor
  \[\Phi_T\colon\mathbf E_G(X)(T)\rightarrow [G\backslash X](T)\]
  mapping a point $x\in X(T)$ to the corresponding orbit map $G_T\to X_T\colon h\mapsto h\cdot x$, and a morphism $(g,x)\colon x \rightarrow g\cdot x$ to the commutative diagram
  \begin{equation*}
    \begin{tikzcd}
      G_{T} \arrow[rr,"(-)\cdot g^{-1}"] \arrow[dr,"\cdot x"'] & & G_{T} \arrow[dl,"\cdot g\cdot x"]\\
      & X_T. & 
    \end{tikzcd}
  \end{equation*}
  These functors $\Phi_T$ are easily checked to be natural in $T$. The morphism ${\rm E}_G(X)\to[G\backslash X]$ is obtained from this by taking sectionwise nerves, and the associated natural transformation $\Phi$ exhibits $[G\backslash X]$ as the stackification of the functor $\mathbf E_G(X)(-)$. 

  The functor $\Phi$ is fully faithful. Indeed giving an automorphism $\varphi$ of the trivial $G$-torsor $G_T$ over $T$ is equivalent to giving the element $\varphi(e_T)\in G(T)$, where $e_T$ denotes the identity element of the group $G(T)$, and in fact $\varphi=(-)\cdot \varphi(e_T)$. Thus the condition that the diagram
  \begin{equation}
    \label{eq:morphism_in_X/G}
    \begin{tikzcd}
      G_{T} \arrow[rr,"\varphi"] \arrow[dr,"\cdot x"'] & & G_{T} \arrow[dl,"\cdot g\cdot x"]\\
      & X_T. & 
    \end{tikzcd}
  \end{equation}
  be a morphism in $[G\backslash X](T)$ is translated into the condition that $x=\varphi(e)\cdot g\cdot x$ in $X(T)$, in other words the arrow $\varphi(e_T)^{-1}\colon x \ra g\cdot x$ is the unique morphism in $\mathbf E_G(X)(T)$ with the property that $\Phi_T(\varphi(e_T)^{-1})$ be the diagram \eqref{eq:morphism_in_X/G}.

  The functor $\Phi_T$ is essentially surjective, if all $G$-torsors over $T$ are trivial, but is not generally an equivalence. However, if $G$ is special, then the functor $\Phi$ is an equivalence locally in the Nisnevich topology, which implies that ${\rm E}_G(X)\to [G\backslash X]$ induces an isomorphism in $\mathbf{H}(F)$. Compare the local arguments in the proof of \cite{krishna}*{Proposition~3.2}. 
\end{proof}

\begin{remark}
One application of this is that for special groups $G$, the realization (complex or real) of the quotient stack can be computed from the simplicial scheme ${\rm E}_G(X)$ by taking levelwise realization. We will use this in Section~\ref{sec:real-pts-hc2}.
\end{remark}
  
\begin{example}
  \label{ex:X/G_vs_action_groupoid}
  The arguments in the proof of Proposition~\ref{prop:simplicial-comparison} allow us to describe the values of $[G\backslash X]$ over $T=\op{Spec}(k)$ for $k=\bar{k}$ an algebraically closed field. In this case, all $G$-torsors over $T$ are trivial, so $\Phi_T$ is essentially surjective.   In particular, for any algebraically closed extension $\bar k$ of the base field $F$, the connected components of $[G\backslash X](\bar k)$ are the $G(\bar k)$-orbits of $X(\bar k)$, and the automorphism group of an object $f\colon G_{\bar k}\to X$ is given by the stabilizer group of the corresponding point $f(e)$.
\end{example}

\begin{example}
  \label{ex:real-points}
  Let $G\looparrowright X$ be a group action in varieties over $\mathbb{R}$. 
  We next describe the groupoid $[G\backslash X](\mathbb{R})$  of real points of the quotient stack $[G\backslash X]$. The objects of $[G\backslash X](\mathbb{R})$ are spans $\op{Spec}\mathbb{R}\xleftarrow{p}P\xrightarrow{f} X$, consisting of a $G$-torsor $p\colon P\to \op{Spec}\mathbb{R}$ and an equivariant morphism $f\colon P\to X$. Under the natural morphism
  \[
    {\rm H}^1({\rm C}_2,G)\to {\rm H}^1({\rm C}_2,G^{\rm ad})\to {\rm H}^1({\rm C}_2,\op{Aut}(G)),
  \]
  the $G$-torsor $P$ corresponds to a strong inner form of $G$, cf. also the discussion in Section~\ref{sec:strong-real-forms}. Note that the ${\rm C}_2$-action appearing in the Galois cohomology sets above is the complex conjugation action on $G(\mathbb{C})$ for the real group $G$.

  Essentially, the objects of $[G\backslash X](\mathbb{R})$ are orbit maps from strong inner forms of $G$ to $X$. The morphisms are commutative triangles of orbit maps, so that the fundamental group of $[G\backslash X](\mathbb{R})$ at the base point $\op{Spec}\mathbb{R}\xleftarrow{p}P\xrightarrow{f} X$ is the stabilizer group of the corresponding orbit in the strong inner form $\op{Aut}(P)$. We will discuss this in more detail in Section~\ref{sec:hc2-galois}.
\end{example}

\begin{remark}
  Using this example, we can now explain how ${\rm E}_G(X)\to[G\backslash X]$ is not an equivalence in $\mathbf{H}(F)$ in general. For example, in the case $G={\rm O}(n)$, the action groupoid ${\rm E}_G(\ast)$ over $\mathbb{R}$ is connected, only containing the orbit map $G\to \ast$ for the trivial $G$-torsor. On the other hand, the groupoid $[{\rm O}(n)\backslash *](\mathbb{R})$ is not connected, rather its connected components are in bijective correspondence with the real forms of ${\rm O}(n)$. In the unstable motivic homotopy category, ${\rm E}_G(\ast)$ is equivalent to ${\rm B}_{\rm Nis}G$, while $[G\backslash *]$ is equivalent to ${\rm B}_\et G$.
\end{remark}

\begin{example}
  \label{ex:complexification-functor}
  To prepare our discussion of homotopy fixed points in Section~\ref{sec:hc2-galois}, we give a brief description of the complexification functor $[G\backslash X](\mathbb{R})\to [G\backslash X](\mathbb{C})$, i.e., the restriction along $\op{Spec}\mathbb{C}\to \op{Spec}\mathbb{R}$. Recall from Example~\ref{ex:real-points} that the objects of $[G\backslash X](\mathbb{R})$ are $G$-equivariant morphisms $f\colon P\to X$ where $p\colon P\to\op{Spec}\mathbb{R}$ is a $G$-torsor over $\op{Spec}\mathbb{R}$, and morphisms are $G$-equivariant maps $P\to P'$ commuting with the morphisms to $X$. The restriction $[G\backslash X](\mathbb{R})\to [G\backslash X](\mathbb{C})$ is given by base-change along $\op{Spec}\mathbb{C}\to\op{Spec}\mathbb{R}$. Therefore, as expected, the restriction is a complexification functor, mapping an object $f\colon P\to X$ to the complexification $P_{\mathbb{C}}\cong G_{\mathbb{C}}\to X_{\mathbb{C}}$. Similarly, a $G$-equivariant morphism $P\to P'$ is mapped to its complexification $P_{\mathbb{C}}\to P'_{\mathbb{C}}$ which will then commute with the complexified morphisms to $X_{\mathbb{C}}$.
\end{example}

\begin{remark}
  There are other ways to relate the quotient stack $[G\backslash X]$ with a simplicial scheme, given e.g.{}~by \cite{choudhury:deshmukh:hogadi:stacks}*{Lemma~2.6}. Choosing a local epimorphism $p\colon Y\to\mathcal{X}$ of simplicial presheaves induces a local equivalence $\check{C}(Y)\to \mathcal{X}$, allowing to identify stacks as simplicial schemes in motivic homotopy. For our purposes, we want to work with the simplicial Borel constructions (even if some identifications are restricted to special groups) because they provide the connection to the homotopy fixed point picture. 
\end{remark}

\subsection{Stacks in motivic homotopy: approximation by smooth schemes}
\label{sec:stacks-motivic-approx}

Another way to understand a quotient stack $[G\backslash X]$ associated to a variety with action $G\looparrowright X$ is in terms of algebraic approximations of Borel constructions, the approach originally used by Totaro \cite{totaro:bg} and Edidin--Graham \cite{edidin:graham:equivariant}. Their ideas were expanded into the notion of admissible gadgets by Morel and Voevodsky \cite{MorelVoevodsky1999}, and used in various papers to define motives of stacks, e.g.{}~\cite{hoskins:pepinlehalleur}. The key feature of the approximation approach is that it allows to write the stack as a colimit of smooth schemes (taken in motivic spaces), which is a good way to compute (complex and real) realization and to define the equivariant real cycle class maps later. This description also makes more apparent that the quotient stacks are homotopy orbit spaces in the $\infty$-topos of \'etale sheaves of spaces on smooth schemes.

We also discuss the equivalence of all the various possibilities to define the unstable homotopy type of a (quotient) stack, building on \cite{choudhury:deshmukh:hogadi:stacks} and previous work of Krishna \cite{krishna}. Some comparison between approximations and simplicial Borel constructions for equivariant motives can also be found in \cite{BGK} (for group actions, without stack language).

Approximation of the Borel construction ${\rm E}G\times_{/G} X$ by smooth schemes can be done using $G$-representations, as in \cite{totaro:bg} or \cite{edidin:graham:equivariant}. To fix a specific colimit description, one can use admissible gadgets, originating in \cite{MorelVoevodsky1999}, see also \cite{krishna}*{Definition~2.1} or \cite{choudhury:deshmukh:hogadi:stacks}*{Definition~3.2}.

\begin{definition}
  \label{def:admissible-gadget}
  Let $G$ be a linear algebraic group over a field $F$. A pair $(V,U)$ of smooth $F$-schemes is said to be a \emph{good pair} for $G$ if $V$ is a right $G$-representation defined over $F$ and $U\subset V$ is a $G$-stable open subset on which $G$ acts freely, and the quotient $U/G$ exists as a smooth quasi-projective scheme.

  A sequence of pairs $\rho=(V_i,U_i)_{i\geq 1}$ is said to be an \emph{admissible gadget} for $G$ if there exists a good pair $(V,U)$ for $G$ such that $V_i=V^{\oplus i}$ and $U_i\subset V_i$ is a $G$-stable open subscheme on which $G$ acts freely and the quotient exists as a smooth quasi-projective scheme such that the following conditions are satisfied:
  \begin{itemize}
  \item $(U_i\oplus V)\cup(V\oplus U_i)\subseteq U_{i+1}$
  \item $\op{codim}_{U_{i+2}}(U_{i+2}\setminus (U_{i+1}\oplus V))>\op{codim}_{U_{i+1}}(U_{i+1}\setminus (U_i\oplus V))$
  \item $\op{codim}_{V_{i+1}}(V_{i+1}\setminus U_{i+1})>\op{codim}_{V_i}(V_i\setminus U_i)$
  \end{itemize}
\end{definition}

\begin{definition}
  Let $G\looparrowright X$ be a smooth quasi-projective variety with action over $F$. We define the "approximations"
  \[B^i_G(X,\rho)=X^i_G(\rho):=U_i\times_{/G}X \]
  as (the \'etale sheaf associated to) the quotient of $U_i\times X$ with respect to the left action
  \[G\times U_i\times X \ra U_i\times X\colon \;(g,u,x)\mapsto (ug^{-1},gx).\]
  Similarly we denote 
  \[E^i_G(X,\rho)=U_i\times X\simeq B_G(G\times X,\rho).\]
\end{definition}

For a variety with action $G\looparrowright X$ and an admissible gadget $\rho$, the colimit
\[
X_G(\rho)=\op{colim}_i\left(U_i\times_{/G}X\right),
\]
taken in the category of motivic spaces, is an algebraic version of the Borel construction, cf. also \cite{choudhury:deshmukh:hogadi:stacks}*{p.~256}. We can compare this to the quotient stack $[G\backslash X]$ via the natural morphism $U_i\times_{/G}X\to [G\backslash X]$ sending $U_i\times_{/G}X$ to the span
\[
U_i\times_{/G}X\xleftarrow{q} U_i\times X\xrightarrow{{\rm pr}_X} X.
\]
Here $q\colon U_i\times X\to U_i\times_{/G}X$ is the quotient projection for the balanced product, and the span itself is viewed as an element in $[G\backslash X](U_i\times_{/G}X)$, i.e., a morphism $U_i\times_{/G}X\to[G\backslash X]$. The naturality of these morphisms implies a well-defined morphism
\[
X_G(\rho)=\op{colim}_i\left(U_i\times_{/G}X\right)\to [G\backslash X].
\]

The comparison result is then the following:

\begin{proposition}
  \label{prop:approx-comparison}
  For a variety with action $G\looparrowright X$ and an admissible gadget $\rho$, the natural morphism $X_G(\rho)\to [G\backslash X]$ is an equivalence in $\mathbf{H}(F)$. 
\end{proposition}

\begin{proof}
  This is essentially \cite{krishna}*{Proposition~3.2} or \cite{choudhury:deshmukh:hogadi:stacks}*{Lemma~3.3}. We want to note, however, that the results cited rather state the equivalence between $X_G(\rho)$ and the simplicial Borel construction $X_G^\bullet={\rm E}_G(X)$ discussed in the previous section, under the assumption that $G$ is special. The proofs of these results, however, actually go through the zig-zag $X_G(\rho)\rightarrow [G\backslash X]\leftarrow X_G^\bullet$ and show that the first arrow is an equivalence in general, while the second arrow is an equivalence for special groups. In any case, knowing that $X_G(\rho)\to[G\backslash X]$ is an equivalence for special groups implies the general case: we choose a faithful representation $G\hookrightarrow {\rm GL}_n$ and consider the induced action ${\rm GL}_n\looparrowright ({\rm GL}_n\times_{/G} X)$. For a choice of admissible gadget $\rho$ for ${\rm GL}_n\looparrowright({\rm GL}_n\times_{/G}X)$, we have a commutative diagram
  \[
  \xymatrix{
    X_G(\rho) \ar[r] \ar[d] & [G\backslash X] \ar[d] \\
    \left({\rm GL}_n\times_{/G} X\right)_{{\rm GL}_n}(\rho) \ar[r] & \left[{\rm GL}_n\backslash\left({\rm GL}_n\times_{/G}X\right)\right]
  }
  \]
  Here, the admissible gadget $\rho$ for ${\rm GL}_n$ is also an admissible gadget for the subgroup $G\leq{\rm GL}_n$, the horizontal morphisms are the canonical morphisms discussed before the statement of Proposition~\ref{prop:approx-comparison}. The vertical morphisms are induction equivalences, the left is an isomorphism by \cite{krishna}*{Corollary~2.7}, the right one by Proposition~\ref{prop:quotient-induction-stacks} (2). The lower horizontal morphism is an equivalence, this is Krishna's equivalence \cite{krishna}*{Proposition~3.2} for special groups, therefore the upper horizontal morphism is an equivalence. 
\end{proof}

\begin{remark}   
  Another way to see the equivalence would be to identify the $G$-principal bundles. On the stack side, $[G\backslash (G\times X)]\to[G\backslash X]$ is the stack version of the universal $G$-bundle. We can show that $[G\backslash (G\times X)]$ is equivalent to $X$, and then the natural map $\op{colim}_i(U_i\times X)\to[G\backslash (G\times X)]$ is an equivalence. The claim follows by passing to homotopy $G$-orbits. 
\end{remark}

In any case, the above result shows that the unstable motivic homotopy type of a quotient stack can be described in terms of the algebraic Borel construction and thus is concretely written as a filtered colimit of smooth schemes. The same is true for representable morphisms of quotient stacks, i.e., representable morphisms of stacks can be written as colimits of morphisms of approximations. This can be used to produce relevant examples of proper/smooth/lci representable morphisms of quotient stacks, and will be relevant for our discussion of compatibilites of the equivariant real cycle class map in  Section~\ref{sec:equivariant-cycle-class}:

\begin{example}
  \label{ex:quotient-stack-morphism-properties}
  According to \cite{edidin:graham:equivariant}*{Proposition~2}, if a morphism $({\rm id},f)\colon (G\looparrowright X)\to (G\looparrowright Y)$ has one of the properties (proper, flat, smooth, regular embedding, lci), then so does any approximation of the Borel construction $U\times_{/G}X$. In particular, the corresponding morphism $f\colon [G\backslash X]\to[G\backslash Y]$ will also have the property.
\end{example}

\subsection{Vector bundles and Picard group for stacks}
\label{sec:equivariant-vb}

We recall some basics concerning vector bundles on quotient stacks which we'll need. On the one hand, the cohomology theories we consider can be twisted by line bundles, and so their equivariant versions will be twisted by equivariant line bundles (or more generally by virtual vector bundle classes in $\underline{K}([G\backslash X])$, as in \cite{dilorenzo:mantovani}). On the other hand, we will also discuss real realization of (symmetric) vector bundles in Section~\ref{sec:vb-realization}, which will allow us to compare algebraic characteristic classes to their topological counterparts. 

Let $G\looparrowright X$ be a variety with action. Then a \emph{$G$-vector bundle}, or \emph{equivariant vector bundle}, on $X$ is a morphism $(G\looparrowright E)\to (G\looparrowright X)$ of varieties with action, whose underlying morphism $E\to X$ of varieties is a vector bundle, and such that the $G$-action on the fibers of $E$ is by linear isomorphisms.\footnote{In the GIT literature, this is called $G$-linearized vector bundle.}

Alternatively, for a smooth quotient stack $[G\backslash X]$, equivariant vector bundles over the stack can be described in terms of morphisms of stacks $[G\backslash X]\to[{\rm GL}_n\backslash *]$. Essentially, this means that for each smooth scheme $S$ and any morphism $f\colon S\to[G\backslash X]$, we have an associated vector bundle $E_f\to S$, and these are compatible with pullbacks along morphisms $S\to S'$. For a $G$-equivariant vector bundle $E\to X$, the corresponding morphism $[G\backslash X]\to [{\rm GL}_n\backslash *]$ is locally given by associated bundle constructions as in Example~\ref{ex:vb-bg} below.

More generally, we can also describe $H$-principal bundles on quotient stacks in terms of morphisms $[G\backslash X]\to [H\backslash *]$ from the quotient stack to the classifying stack of $H$. The relevant case for us will be $G$-equivariant symmetric vector bundles, given by morphisms of the form $[G\backslash X]\to [{\rm O}(n)\backslash*]$. We will discuss their real realization in Section~\ref{sec:vb-realization}.

We denote by $\op{Pic}_G(X)=\op{Pic}([G\backslash X])$ the group of $G$-line bundles on $X$. This group can be identified with an equivariant Chow group by \cite{edidin:graham:equivariant}*{Corollary~1 to Theorem~1, p.~602}: if $G\looparrowright X$ is a smooth variety with action, then the map $c_1\colon {\rm Pic}_G(X)\to {\rm CH}^1_G(X)$ is an isomorphism. Due to this identification, we will mostly denote the group of $G$-equivariant line bundles over $X$ by ${\rm CH}^1_G(X)$. 

\begin{example}
  \label{ex:vb-bg}
  For a linear algebraic group $G$, the $G$-vector bundles of rank $n$ on $[G\backslash *]$ are identified with the representations $G\to {\rm GL}_n$. The morphism $[G\backslash *](S)\to [{\rm GL}_n\backslash *](S)$ corresponding to a representation $\rho\colon G\to{\rm GL}_n$ sends a span $S\leftarrow P\to *$ (of a $G$-principal bundle $P\to S$ and a $G$-equivariant morphism $P\to X$) to the span where $P$ is replaced by the vector bundle associated to $P$ and $\rho$.

  In particular, the group ${\rm CH}^1_G(*)$ of equivariant line bundles on the point is identified with the group of characters $G\to\mathbb{G}_{\rm m}$ of $G$.
\end{example}

Vector bundles (or better, symmetric vector bundles) on classifying stacks $[G\backslash *]$ will later be our preferred method of describing elements in the cohomology of $[G\backslash *]$ as characteristic classes. 

\section{Real points of stacks as fixed-point groupoids}
\label{sec:real-pts-fp-gpd}
        
In the following section, we will recall the description of the groupoid of real points of a quotient stack $[G\backslash X]$ as the fixed-point groupoid of its complex points. 
The discussion is based largely on a discussion of homotopy fixed points in Virk's paper \cite{virk:hc2}, discussions of group actions on stacks by Romagny \cite{romagny} and in the context of descent theory by Vistoli \cite{Vistoli:stacks}*{\S~3.8 and \S~4.4}. A similar discussion of descent for vector spaces and homotopy fixed-points of Galois actions can be found in a very clear blog post of annoying precision by Qiaochu Yuan
\begin{center}
  {\small \verb!https://qchu.wordpress.com/2015/11/11/fixed-points-of-group-actions-on-categories/!}
\end{center}
Although most of what we say applies to general (finite) group actions on groupoid, we will not provide the most general discussion possible and only discuss the ${\rm C}_2$-situation in detail.

We first recall the natural ${\rm C}_2$-action (given by complex conjugation) on the complex points of a real quotient stack $[G\backslash X]$. For the general definition of group actions on categories, cf.{}~\cite{romagny}*{Section~1}.

\begin{example}
  \label{ex:cplx-conjugation}
  Let $G\looparrowright X$ be a smooth scheme with action by a (smooth) group scheme $G$, everything defined over $\mathbb R$. We wish to explicitly describe the complex conjugation action of ${\rm C}_2={\rm Gal}(\mathbb{C}/\mathbb{R})$ on $[G\backslash X](\mathbb{C})$.

  Observe that there is a natural Galois action of $\mathrm C_2$ on the set $Y(\mathbb C)$ of complex points of any real variety $Y$, by precomposing a point $y\colon\op{Spec}\mathbb{C}\to Y$  with $\op{Spec}(\sigma)\colon\op{Spec}\mathbb{C}\to\op{Spec}\mathbb{C}$.   This Galois action commutes with all structure maps of the action groupoid $\mathbf E_G(X)(\mathbb C)$ of Definition~\ref{def:action_groupoid}, which thus inherits a strict action of the group $\mathrm C_2$. Explicitly, $\sigma \in \mathrm C_2$ acts on objects of $\mathbf E_G(X)(\mathbb C)$ by $x\mapsto \sigma(x)$ as just explained, and on morphisms by
  \[
  (g\colon x\ra g\cdot x) \mapsto (\sigma(g)\colon \sigma(x)\ra \sigma(g)\cdot \sigma(x)).
  \]
  Using the equivalence $\Phi_{\mathbb C}\colon[G\backslash X](\mathbb C)\cong \mathbf E_G(X)(\mathbb C)$, this equips the complex points $[G\backslash X](\mathbb{C})$ with a Galois action. In explicit terms, the image along $\Phi_{\mathbb C}$ of the object $x\in \mathbf E_G(X)(\mathbb C)$ is $G_{\mathbb C} \overset{\cdot x}{\ra}X_{\mathbb C}$ while the image of the object $\sigma(x)$ is $G_{\mathbb C} \overset{\cdot \sigma(x)}{\ra}X_{\mathbb C}$. If $g\colon x\ra g\cdot x$ is an arrow of $\mathbf E_G(X)(\mathbb C)$, the image of $\sigma(g)$ along $\Phi_{\mathbb C}$ is the diagram 
  \begin{equation*}
    \begin{tikzcd}
          G_{\mathbb C} \arrow[rr,"\cdot \sigma(g)^{-1}"] \arrow[dr,"\cdot \sigma(x)"'] & & G_{\mathbb C} \arrow[dl,"\cdot \sigma(g\cdot x)"]\\
          & X_{\mathbb C}. & 
    \end{tikzcd}
  \end{equation*}
  This can be identified with the induced pullback action of $\sigma\colon \op{Spec}(\mathbb{C})\to\op{Spec}(\mathbb{C})$ on $[G\backslash X](\mathbb{C})$ as discussed in Section~\ref{sec:quot-stacks}, mapping $(\op{Spec}(\mathbb C) \overset{\pi}{\leftarrow} P\overset{f'}{\rightarrow} X)$ to $(\op{Spec}(\mathbb C) \overset{\pi'}{\leftarrow} \sigma. P\overset{f}{\rightarrow} X)$.
  \end{example}

Next, we recall the definition of fixed-point groupoid $\mathcal{C}^{\rm hC_2}$ for a group action of ${\rm C}_2$ on a groupoid $\mathcal{C}$. We will discuss in Section~\ref{sec:fp-gpd-hfp} that the geometric realization of the fixed-point groupoid will be the  ${\rm C}_2$-homotopy-fixed points of the realization of the groupoid $\mathcal{C}$. From this, it will be easy to compute fixed-point groupoids for complex conjugation on classifying stacks, cf.{}~Section~\ref{sec:hc2-galois}.

\begin{definition}
  \label{def:hG-groupoid}
  For a category $\mathcal{C}$ with a left action of a group $\Gamma$, the fixed-point category $\mathcal{C}^{{\rm h}\Gamma}$ is defined as follows. The objects are collections $(X,(\varphi_\gamma)_{\gamma\in \Gamma})$ of an object $X\in\mathcal{C}$ and isomorphisms $\varphi_\gamma\colon X\to F(\gamma)(X)$ satisfying the following compatibility. For each pair $\gamma,\delta\in \Gamma$ we require that the following square commutes:
  \begin{equation*}
    \begin{tikzcd}[column sep=40pt, row sep=15pt]
          X \arrow[r,"\varphi_{\gamma\delta}"] \arrow[d,"\varphi_\gamma"] &  F(\gamma\delta)(X)  \\
           F(\gamma)(X) \arrow[r,"F(\gamma)(\varphi_\delta)"] & F(\gamma)( F(\delta)(X)) \arrow[u,"\eta_{\gamma,\delta}"].
    \end{tikzcd}
  \end{equation*}
Morphisms $(X,(\varphi_\gamma)_{\gamma\in \Gamma})\to (Y,(\psi_\gamma)_{\gamma\in \Gamma})$ in $\mathcal{C}^{{\rm h}\Gamma}$ are morphisms $f\colon X\to Y$ such that for every $\gamma\in \Gamma$ the following diagram commutes:
  \begin{equation*}
    \begin{tikzcd}[column sep=40pt, row sep=15pt]
          X \arrow[r,"\varphi_{\gamma}"] \arrow[d,"f"] &  F(\gamma)(X) \arrow[d,"F(f)"]  \\
          Y \arrow[r,"\psi_\gamma"] & F(\gamma)(Y).
    \end{tikzcd}
  \end{equation*}
  There is a natural forgetful functor $\mathcal{C}^{{\rm h}\Gamma}\to \mathcal{C}$ which maps a homotopy fixed point $(X,(\varphi_\gamma)_\gamma)$ to its underlying object $X$.
\end{definition}

\begin{example}
  \label{ex:hc2-groupoid}
  In the situation of definition \ref{def:hG-groupoid} (where we have a strict action of $\mathrm C_2=\langle\sigma\rangle$ on a category $\mathcal{C}$), a fixed point is simply a pair $(X,\varphi)$ of an object $X\in\mathcal{C}$ equipped with an isomorphism $\varphi\colon X\to F(\sigma)(X)$ such that $F(\sigma)(\varphi)=\varphi^{-1}$. A morphism $(X,\varphi)\to (Y,\psi)$ in $\mathcal{C}^{{\rm hC}_2}$ is a morphism $f\colon X\to Y$ in $\mathcal{C}$ such that the following diagram commutes:
  \[
  \xymatrix{
    X \ar[r]^\varphi \ar[d]_f & F(\sigma)(X) \ar[d]^{F(\sigma)(f)} \\
    Y \ar[r]_\psi & F(\sigma)(Y)
  }
  \]
\end{example}

\begin{comment}
\begin{example}
  \label{ex:hc2-vect}
  We briefly discuss homotopy fixed points in Example~\ref{ex:galois-vect} of Galois actions on vector spaces. For a $K$-vector space $V$, the scalar extension $V\otimes_KL$ can be endowed with a natural structure of homotopy fixed point with respect to the Galois action of $\mathrm{Gal}(L/K)$ on ${\rm Mod}(L)$ as follows. \textcolor{blue}{I'm not sure of the meaning of the following sentence: For any $\gamma\in \Gamma$, $\lambda\in L$ and $v\in V\otimes_KL$, we have $\gamma(v\lambda)=\gamma(v)\gamma(\lambda)$.} The assignment $v\otimes \lambda\mapsto v\otimes\gamma^{-1}(\lambda)$ can then be interpreted either as $\gamma^{-1}$-semilinear map $V\otimes_KL\to V\otimes_KL$ or as $L$-linear isomorphism $V\otimes_KL\to \gamma.(V\otimes_KL)$. We remind the reader that the target the vector space structure is twisted so that $\mu\in L$ acts by scalar multiplication with $\gamma^{-1}(\mu)$   on (the second tensor factor of) $V\otimes_KL$. This is the isomorphism in the datum of homotopy fixed point. Conversely, such an $L$-linear isomorphism provides the semi-linear map giving the descent datum for a $K$-structure on an $L$-vector space.
\end{example}
\end{comment}

\begin{example}
  \label{ex:hc2-orbits}
  We return to Example~\ref{ex:cplx-conjugation} of the ${\rm C}_2$-action on $[G\backslash X](\mathbb{C})$ for a real group action $G\looparrowright X$. From the action groupoid viewpoint in \cite{virk:hc2}, a homotopy fixed point is a pair $(g,x)\in G(\mathbb{C})\times X(\mathbb{C})$ such that $g\cdot x=\sigma(x)$ and $\sigma(g)=g^{-1}$. We can then write a morphism of homotopy fixed points as a commutative diagram
  \[
  \xymatrix{
    x \ar[r]^g \ar[d]_h & \sigma(x) \ar[d]^{\sigma(h)} \\
    y \ar[r]_{g'} & \sigma(y)
  }
  \]
  In other words, a morphism $(g,x)\to (g',y)$ of homotopy fixed points is an element $h\in G(\mathbb C)$ for which the relation $g'=\sigma(h)gh^{-1}$ holds.
\end{example}

Let $G\looparrowright X$ be a scheme with group action over $\mathbb{R}$. We finally come to the identification of the groupoid of real points of the quotient stack $[G\backslash X]$ with the fixed-point groupoid of the complex points. In the interest of exposition, we describe two ways of doing that: one high-level argument using \'etale descent of quotient stacks, and one argument showing more explicitly how to use \'etale descent for torsors to write down an equivalence of groupoids.

\begin{theorem}
  \label{thm:etale-descent}
  Let $\scr X$ be a stack in groupoids over the big \'etale site of smooth $k$-schemes. Let $\Gamma$ be a finite group, and let $\underline \Gamma$ be the associated constant group scheme. If $T'\rightarrow T$ is a $\underline \Gamma$-torsor, then there is a canonical equivalence of groupoids $\scr X(T')^{{\rm h}\Gamma} \simeq \scr X(T)$.
\end{theorem}

\begin{proof}
  The result is a reformulation of descent statements discussed in \cite{Vistoli:stacks}. More precisely, it is a combination of two identifications: the first identification, given by \cite{Vistoli:stacks}*{Theorem~4.46}, is a canonical equivalence between  $\scr X(T)$ and the category of $\underline \Gamma$-equivariant objects in $\scr X(T')$, denoted by $\scr X(T')^{\underline \Gamma}$. The second identification, given in \cite{Vistoli:stacks}*{Proposition~3.49} is a canonical equivalence  $\scr X(T')^{\underline \Gamma}\simeq\scr X(T')^{{\rm h}\Gamma}$ between $\underline{\Gamma}$-equivariant objects and homotopy fixed points.
\end{proof}

More specifically, this implies that for an \'etale stack $\mathscr{X}$ over $\op{Spec}(\mathbb{R})$, the real points $\mathscr{X}(\mathbb{R})$ can be identified as fixed-point data on complex points, as formulated in the following:

\begin{corollary}
  \label{cor:main-hc2}
  Let $G\looparrowright X$ be a smooth scheme with group action over $\mathbb{R}$. The complexification functor induces an equivalence of topological groupoids
\[
[G\backslash X](\mathbb{R})\xrightarrow{\simeq} [G\backslash X](\mathbb{C})^{{\rm hC}_2}
\]
\end{corollary}

This follows directly from Theorem~\ref{thm:etale-descent}, applied to the ${\rm C}_2$-torsor $\op{Spec}(\mathbb{C})\to \op{Spec}(\mathbb{R})$ and the quotient stack $[G\backslash X]$. In the rest of the section, we will spell out more concretely what happens in the proof. The key ingredients are essentially statements on Galois cohomology and Galois descent contained in~\cite{serre:galois-cohom}*{Chapter I, 5.2 and 5.3}.
  
The first thing to note is that the complexification factors through $[G\backslash X](\mathbb{C})^{{\rm hC}_2}$. This involves understanding the homotopy fixed point structure on complexified orbit maps, elaborating on Example~\ref{ex:hc2-orbits}. Given a $G$-torsor $p\colon P\to \op{Spec}\mathbb{R}$ and a $G$-equivariant morphism $f\colon P\to X$, we get the complexified morphism $f_{\mathbb{C}}\colon P_{\mathbb{C}}\to X_{\mathbb{C}}$. This morphism commutes with complex conjugation and is therefore Galois-equivariant. Therefore, we can change the orbit map $f_{\mathbb{C}}$ by precomposition with complex conjugation.\footnote{Alternatively, we could use postcomposition, that wouldn't make a difference because of the Galois equivariance.} We call the resulting orbit map $\overline{f}_{\mathbb{C}}=F(\sigma)(f_{\mathbb{C}}\colon P_{\mathbb{C}}\to X_{\mathbb{C}})$. Complex conjugation induces a natural map
\[
(f_{\mathbb{C}}\colon P_{\mathbb{C}}\to X_{\mathbb{C}})\to F(\sigma)(f_{\mathbb{C}}\colon P_{\mathbb{C}}\to X_{\mathbb{C}}).
\]
Note that complex conjugation doesn't provide a $G$-equivariant automorphism of $f_{\mathbb{C}}\colon P_{\mathbb{C}}\to X_{\mathbb{C}}$ but rather a kind of semilinear action $\sigma(g\cdot x)=\sigma(g)\cdot\sigma(x)$, which we can view as an isomorphism from $f_{\mathbb{C}}$ to the twisted orbit map $\overline{f}_{\mathbb{C}}$. This provides a canonical homotopy fixed point datum. From the construction, we can also see that for a morphism $f\colon P_1\to P_2$ of $G$-torsors over $\mathbb{R}$ and corresponding orbit maps, the complexified morphisms will be compatible with the canonical homotopy fixed point structure. Therefore, we find that the complexification functor $[G\backslash X](\mathbb{R})\to [G\backslash X](\mathbb{C})$ factors through a functor $[G\backslash X](\mathbb{R})\to [G\backslash X](\mathbb{C})^{{\rm hC}_2}$ to the fixed-point groupoid.

At the expense of functoriality, we can make this a bit more explicit (and relate it to the action groupoid) by choosing an isomorphism $\varphi\colon G_{\mathbb{C}}\xrightarrow{\cong}P_{\mathbb{C}}$ and consider the morphism $f\circ\varphi\colon G_{\mathbb{C}}\to X_{\mathbb{C}}$. The isomorphism $\varphi$ doesn't necessarily commute with complex conjugation, in fact, it only does so when $P$ is the trivial $G$-torsor over $\op{Spec}\mathbb{R}$. Denoting by $x$ the image of $1_G$ under $f\circ \varphi$, we know that $\overline{x}$ is also in the image of $f\circ\varphi$ because $f\colon P_{\mathbb{C}}\to X_{\mathbb{C}}$ commutes with complex conjugation. We need to find $g\in G(\mathbb{C})$ with $\overline{g}=g^{-1}$ such that $\overline{x}=g\cdot x$. By
  \[
  G_{\mathbb{C}}\xrightarrow{\varphi}P_{\mathbb{C}}\xrightarrow{\sigma}P_{\mathbb{C}}\xrightarrow{\varphi^{-1}}G_{\mathbb{C}}
  \]
  we get an involution on $G_{\mathbb{C}}$ which differs from complex conjugation $\sigma$ by conjugation with some $g\in G(\mathbb{C})$. Then multiplication by $g$ provides the isomorphism  $f\circ \varphi\cong f\circ \overline{\varphi}$.

  We next observe that the functor $[G\backslash X](\mathbb{R})\to [G\backslash X](\mathbb{C})^{{\rm hC}_2}$ is fully faithful. To see this, suppose we have two $G$-torsors $P_1,P_2$ over $\op{Spec}\mathbb{R}$ and $G$-equivariant morphisms $f_i\colon P_i\to X$. Assume that there is an isomorphism $\varphi\colon P_{1,\mathbb{C}}\xrightarrow{\cong}P_{2,\mathbb{C}}$ compatible with the complexified orbit maps $f_{i,\mathbb{C}}$ and their natural  homotopy fixed points structures as above. This in particular means that $\varphi$ is compatible with complex conjugation, and by taking fixed points we obtain an isomorphism $P_1\to P_2$. In particular, ${\rm Hom}(P_1,P_2)$ is nonempty in $[G\backslash X](\mathbb{R})$ if and only if ${\rm Hom}(P_1\otimes \mathbb{C},P_2\otimes\mathbb{C})$ is nonempty in $[G\backslash X](\mathbb{C})^{{\rm hC}_2}$. To get full faithfulness, we want to show that the morphism of Hom-spaces
  \[
    {\rm Hom}_{[G\backslash X](\mathbb{R})}(P_1,P_2)\to {\rm Hom}_{[G\backslash X](\mathbb{C})^{{\rm hC}_2}}(P_1\otimes \mathbb{C},P_2\otimes\mathbb{C})
  \]
  induced by complexification is a bijection. To see this, we can first identify the morphisms ${\rm Hom}_{[G\backslash X](\mathbb{R})}(P_1,P_2)$ over $\mathbb{R}$ as the set of $g\in G(\mathbb{R})$ such that $g\circ f_2=f_1$. This is a torsor under the stabilizer of either orbit map $f_i\colon P_i\to X$ over $\mathbb{R}$. Similarly, the morphisms ${\rm Hom}_{[G\backslash X](\mathbb{C})}(P_1\otimes \mathbb{C},P_2\otimes\mathbb{C})$ over $\mathbb{C}$ are identified as the transporter from $f_{1,\mathbb{C}}\colon P_{1,\mathbb{C}}\to X_{\mathbb{C}}$ to $f_{2,\mathbb{C}}\colon P_{2,\mathbb{C}}\to X_{\mathbb{C}}$, and that is a torsor under the stabilizer of the orbit map $P_{\mathbb{C}}\to X_{\mathbb{C}}$. The latter group is the complexification of the real stabilizer group. Finally, we identify the morphisms which are compatible with the fixed-point structure of $P_{i,\mathbb{C}}$, i.e., the complex conjugation, with the fixed points of the complexified transporter group, which is ${\rm Hom}_{[G\backslash X](\mathbb{C})^{{\rm hC}_2}}(P_1\otimes \mathbb{C},P_2\otimes\mathbb{C})$. 

In the last step, we need to see that the complexification functor $[G\backslash X](\mathbb{R})\to [G\backslash X](\mathbb{C})^{{\rm hC}_2}$ is essentially surjective. So let $f\colon G_{\mathbb{C}}\to X_{\mathbb{C}}$ be a morphism which is isomorphic to its Galois conjugate by an element $g\in G(\mathbb{C})$ with $\overline{g}=g^{-1}$. Then by definition, $g\in Z^1({\rm C}_2,G(\mathbb{C}))$. By \cite{serre:galois-cohom}*{Chapter I, Section 5.3}, cf. also the discussion of Galois cohomology basics in Section~\ref{sec:fp-gpoid-galois} below, the class of $g$ in ${\rm H}^1({\rm C}_2,G)$ corresponds to a $G$-torsor $P$ over $\mathbb{R}$. Moreover, $P$ is realized as fixed points of the involution of $G$ given by the $g$-conjugated Galois action. With this involution, the morphism $f\colon P_{\mathbb{C}}\to X_{\mathbb{C}}$ is Galois equivariant and thus defines a morphism $P\to X$, which means that the homotopy fixed point $f$ we started with is indeed the complexification of a real orbit map for a $G$-torsor. This establishes the essential surjectivity, and concludes the somewhat more explicit proof of Corollary~\ref{cor:main-hc2}. 

We note that spelling the proof out explicitly as above also makes obvious that these constructions are compatible with the topological structures on the groupoids (obtained from equipping $G(\mathbb{C})$ and $X(\mathbb{C})$ with the analytic topologies, similarly for real points and stabilizer subgroups in $G$).

\section{Equivariant homotopy and realization functors}
\label{sec:equivariant-htpy-realization}

In this section, we discuss the definition and basic properties of equivariant and real realization functors on motivic homotopy theory. For this, we need to recall some basics from equivariant homotopy theory. In particular, we discuss fine and coarse equivariant homotopy and the relevant fixed-point functors. 

\subsection{Equivariant homotopy theory}
\label{sec:equivariant-basics}

Equivariant homotopy theory is the study of $\Gamma$-spaces, i.e., spaces with a group action, up to homotopy. As before, most of the statements apply to a general finite  group $\Gamma$, but for our applications, we only care about the case $\Gamma={\rm C}_2$. This section should simply serve as a reminder of the main facts from equivariant homotopy which are relevant for our paper. 

The first point of note is that there are two possible choices for weak equivalences to set up a homotopy theory of $\Gamma$-spaces. 

\begin{definition}
\begin{itemize}
\item A \emph{coarse (or naive) $\Gamma$-equivalence} is a $\Gamma$-equivariant map $X\to Y$ which is a weak equivalence on the underlying spaces.
  \item A \emph{fine (or genuine) $\Gamma$-equivalence} is a $\Gamma$-equivariant map $X\to Y$ which induces weak equivalences $X^H\to Y^H$ for every subgroup $H\leq \Gamma$. 
\end{itemize}
\end{definition}

Each of these notions of weak equivalences is part of an appropriate model structure on $\Gamma$-spaces, giving rise to coarse and fine equivariant homotopy theories. 

\begin{remark}
  As for terminology, we will use the terms coarse/fine, following for example \cite{MorelVoevodsky1999}*{Section~3.3}. These terms seem less judgemental than naive/genuine, and indeed, the coarse equivariant homotopy has a significant role to play for the real realization of \'etale classifying spaces.
\end{remark}

Various choices of models for equivariant homotopy theory are possible, and unfortunately some of the subtleties around fixed-point functors are quite dependent on which models are used. One useful model for fine $\Gamma$-equivariant homotopy theory, by Elmendorf's theorem, is as the homotopy theory of presheaves of spaces on the orbit category of $\Gamma$.

Somewhat similar, but more adapted to the study of realization functors from motivic homotopy, are the models used in \cite{MorelVoevodsky1999}*{Section~3.3}. Morel and Voevodsky define a category of suitably locally contractible ${\rm C}_2$-spaces they denote by ${\rm C}_2{\rm -T}lc$, equipped with two possible Grothendieck topologies (coarse and fine), and show \cite{MorelVoevodsky1999}*{Proposition~3.3, p.~120} that the associated homotopy theories of sites with interval (given by the unit interval with trivial ${\rm C}_2$-action) are equivalent to the coarse and fine equivariant homotopy theory, respectively.

In the case of fine equivariant homotopy theory, the equivalence established in loc.{}~cit.{}~is induced by sending a ${\rm C}_2$-space to the representable presheaf on ${\rm C}_2{\rm -T}lc$.\footnote{Note that Morel and Voevodsky work with a simplicial model of $G$-spaces.} Once we know that this is an equivalence, it follows that the inverse is given by restricting along the obvious inclusion of the orbit category into ${\rm C}_2{\rm -T}lc$. A similar statement is true for the coarse equivariant homotopy theory, only that we have to consider a restriction of the ``\'etale'' topology on the orbit category, in which the map ${\rm C}_2/\langle e\rangle\to {\rm C}_2/{\rm C}_2$ is a covering.

\begin{remark}
  The relation between fine and coarse equivariant homotopy is absolutely parallel to the relation between Nisnevich and \'etale motivic homotopy. We will also see this in the discussion of coarse and fine realization functors below, as well as in the discussion of principal bundles in equivariant homotopy, cf.{}~Section~\ref{sec:coarse}.
\end{remark} 

\subsection{Fixed-point functors}
\label{sec:fixed-points}

For later use in the description of real realization functors, we briefly discuss fixed-point functors in fine and coarse equivariant homotopy.\footnote{At this point, MW would like to thank Jens Hornbostel and Kirsten Wickelgren for illuminating discussions around fixed points and equivariant/real realization.}

We first discuss fixed-point functors on the fine equivariant homotopy. These are usually called \emph{genuine} fixed points. In the Elmendorf model, or the Morel--Voevodsky model, we can essentially take sections of a (pre)sheaf of spaces over the point (with trivial action, i.e., the orbit ${\rm C}_2/{\rm C}_2$). Since this is essentially a pre-sheaf situation, taking fixed-points is a left-adjoint functor. Indeed, it is a left Quillen functor, therefore descends to the homotopy category, and the left-derived functor preserves homotopy colimits. 

\begin{example}
  We should point out that being a left adjoint is unfortunately a model-dependent thing. In ${\rm C}_2$-spaces, ${\rm S}^1$ with the flip action is a standard example to see that fixed points don't commute with colimits (in this case, the coequalizer of identity and flip map). In the (pre)sheaf models of Elmendorf or Morel--Voevodsky, there are more objects (in particular, the sections over ${\rm C}_2/{\rm C}_2$ are not necessarily the fixed points of the ${\rm C}_2$-action on ${\rm C}_2/\langle e\rangle$), so quotients retain finer information, allowing fixed points to be a left adjoint. 
\end{example}

Turning to the coarse equivariant homotopy, it is clear that taking fixed points is not homotopy invariant; a coarse $\Gamma$-homotopy equivalence induces an equivalence of underlying spaces, but doesn't need to preserve fixed points in any way. So the only functor we can consider is the homotopy fixed-points $X^{\rm hC_2}={\rm Map}^{\rm C_2}({\rm EC_2},X)$. Essentially, in the sheaf model of Morel--Voevodsky, this is the usual right-derived functor of taking global sections. In particular, we cannot expect homotopy fixed points to commute with colimits. 

\begin{example}
  As another unsolicited clarification, we note that fixed points and homotopy fixed points differ in general, even for otherwise nice spaces like complex points of real varieties. A difference already appears for $\mathbb{CP}^1$ with the analytic topology and complex conjugation. The fine/genuine fixed points are the actual fixed points, equivalent to $\mathbb{RP}^1$. The homotopy fixed points $\left(\mathbb{CP}^1\right)^{\rm hC_2}$ seem difficult to describe, but in any case are not equivalent to $\mathbb{RP}^1$; to see this, we can use a theorem of Carlsson \cite{carlsson}*{Theorem~B}. The theorem implies that there is a continuous map $\left(\mathbb{CP}^1\right)^{\rm hC_2}\to \left({\rm S}^1\right)^\wedge_2$ with connected fibers, implying a surjection $\pi_1\left(\mathbb{CP}^1\right)^{\rm hC_2}\to \mathbb{Z}_2$ (to the 2-adic integers), which implies that the homotopy fixed points cannot be $\mathbb{RP}^1$.
\end{example}

\begin{remark}
  In the Morel--Voevodsky model, we can consider the change-of-topology map induced by coarse sheafification, which on fixed-points induces the natural map $X^{\rm C_2}\to X^{\rm hC_2}$. The Sullivan conjecture, as proved in \cite{carlsson}, states that this map is an equivalence after 2-completion, for $X$ a finite ${\rm C}_2$-space (such as the complex points of a smooth real variety). Again, we see the coarse equivariant homotopy as a direct analogue of the \'etale topology on the motivic side. 
\end{remark} 

\subsection{Equivariant realization on motivic homotopy}
\label{sec:prelim-realization}

As a next step, we want to recall equivariant realization functors, which provide the connection between motivic homotopy over $\mathbb{R}$ (or more generally schemes with a real point) and ${\rm C}_2$-equivariant homotopy. We focus on the discussion of the functor from unstable Nisnevich-local motivic homotopy to fine ${\rm C}_2$-equivariant homotopy, but we'll also make some remarks on the corresponding version from \'etale motivic homotopy to coarse ${\rm C}_2$-equivariant homotopy. Most of what we'll say here has been said in other places, cf. for example \cite{wickelgren:williams} or \cite{MorelVoevodsky1999}*{Section~3.3}.

Before we get to the ${\rm C}_2$-equivariant realization over $\mathbb{R}$, we briefly recall the complex case. The complex realization functor ${\rm Re}_{\mathbb{C}}\colon {\bf sPre}({\bf Sm}_{\mathbb{C}})\to {\bf Top}$ is a left Quillen functor (for the motivic projective model structure)  determined by sending a smooth scheme $X/\mathbb{C}$ to the topological space $X(\mathbb{C})$ of complex points with the analytic topology. In particular, it preserves homotopy colimits. Alternatively, we can get a functor of the associated $\infty$-categories as left Kan extension of the functor ${\bf Sm}_{\mathbb{C}}\to{\bf Top}\colon X\mapsto X(\mathbb{C})$. Complex realization is also compatible with products, cf.{}~\cite{panin:pimenov:roendigs}*{Appendix A.4}, which in particular means that the complex realization of a smooth scheme with action $G\looparrowright X$ will be a complex Lie group action $G(\mathbb{C})\looparrowright X(\mathbb{C})$ on a complex manifold. Moreover, the complex realization of the bar construction for $G\looparrowright X$ will be the bar construction of $G(\mathbb{C})\looparrowright X(\mathbb{C})$, cf.{}~\cite{wickelgren:williams}*{Section~3.1}.

Over the real numbers, we have analogous ${\rm C}_2$-equivariant realization functors, which are determined by sending a smooth scheme $X/\mathbb{R}$ to the ${\rm C}_2$-space $X(\mathbb{C})$ with the ${\rm C}_2$-action by complex conjugation. There exist different descriptions of equivariant realization functors. One possibility is to simply take a left Kan extension of the complex-points functor ${\bf Sm}_{\mathbb{R}}\to {\bf Top}\colon X/\mathbb{R}\mapsto (X(\mathbb{C}),\sigma)$ to a functor of $\infty$-categories ${\bf Spc}(\mathbb{R})\to{\bf Spc}^{\rm C_2}$.

Alternatively, a description of the functor in terms of model categories was already given by Morel and Voevodsky. In \cite{MorelVoevodsky1999}*{Lemma~3.6 in Section~3}, they show that the functor
\[
\left({\bf Sm}_{\mathbb{R}}\right)_{\rm Nis}\to \left(\mathbf{Top}_{\rm lc}^{\rm C_2}\right)_{\rm fine}\colon X\mapsto \left(X(\mathbb{C}),\sigma\right)
\]
mapping a smooth $\mathbb{R}$-scheme to its complex points with complex conjugation is a reasonable continuous map of sites. As a consequence, there is an induced left derived functor $\mathcal{H}(\mathbb{R})\to\mathcal{H}^{\rm C_2}_{\rm fine}$ from motivic homotopy to fine ${\rm C}_2$-equivariant homotopy. The key point behind the construction is that the equivariant realization of a Nisnevich covering is a covering in the fine topology.

As yet another alternative, it is also possible to replace the target (fine $I$-local sheaves of spaces on $\mathbf{Top}_{\rm lc}^{\rm C_2}$) by the Elmendorf model of equivariant homotopy (presheaves of spaces on the orbit category): the orbit category can be embedded in the obvious way into locally contractible ${\rm C}_2$-spaces, and restriction of sheaves from $\mathbf{Top}_{\rm lc}^{\rm C_2}$ to the orbit category induces a Quillen equivalence of model categories (or equivalence of the associated $\infty$-categories). 

For us, the most relevant consequence in any of these descriptions is that, as in the case of the complex realization functor, ${\rm C}_2$-equivariant realization commutes with (homotopy) colimits. Since it also commutes with products, it will map the bar construction for a group action $G\looparrowright X$ to the bar construction for the realization $G(\mathbb{C})\looparrowright X(\mathbb{C})$, viewed as action in ${\rm C}_2$-spaces.

As a final note, we want to point out that it is also possible to use the category ${\bf Top}^{\rm C_2}$ of ${\rm C}_2$-spaces as target of equivariant realization functors, with exactly the same properties. We don't use that model because there are some model-dependent issues with fixed-point functors that we want to avoid for our discussion of real realization, cf. the discussion in Section~\ref{sec:fixed-points} above. 

\subsection{Note on \'etale-to-coarse realization functors}

We briefly comment on a variation of equivariant realization, mapping from the \'etale $\mathbb{A}^1$-local motivic homotopy theory to coarse ${\rm C}_2$-equivariant homotopy theory:
\[
{\bf Sh}^{\mathbb{A}^1}_{\et}\left({\bf Sm}_{\mathbb{R}}\right)\to {\bf Sh}^{\rm I}_{\rm coarse}\left(\mathbf{Top}_{\rm lc}^{\rm C_2}\right)\colon X\mapsto \left(X(\mathbb{C}),\sigma\right)
\]
As above, the functor is a left Kan extension of the functor sending a real scheme to its ${\rm C}_2$-space of complex points. The key point is that \'etale hypercovers are mapped to coarse hypercovers, and this point is implicitly made in the discussion of realization functors in \cite{dugger:isaksen:realization}*{Theorem~5.2}. (Since we are using coarse equivalences, we only need to check conditions on the underlying space, and then the local splittings in the proof in loc{}.~cit.{}~imply the coarse covering condition.) It should be noted that we cannot hope for a fine equivariant realization on \'etale motivic homotopy theory because \'etale covers are not necessarily surjective on $\mathbb{R}$-points, hence \'etale hypercovers do not in general induce fine weak equivalences. 

We won't make much use of the \'etale-to-coarse realization functor, but its existence already suggests a possible relation between \'etale homotopy orbit spaces (like geometric classifying spaces ${\rm B}_\et G$) and coarse homotopy orbit spaces. This relation will be made more precise in Sections~\ref{sec:coarse} and \ref{sec:real-pts-hc2}.

\begin{remark}
  As a final remark, we note that we can compose the \'etale-to-coarse equivariant realization with the homotopy fixed-point functor to get a real realization functor from \'etale motivic homotopy to ordinary topology. This realization functor sends a smooth $\mathbb{R}$-scheme $X$ to the homotopy fixed-point space $X(\mathbb{C})^{\rm hC_2}$. By Carlsson's results on the Sullivan conjecture, cf.{}~\cite{carlsson}, this functor agrees  (on smooth schemes) with the ``usual'' real realization functors (composing Nisnevich-to-fine equivariant realization with genuine fixed points) after 2-completion. 
\end{remark}

\subsection{Real realization in unstable motivic homotopy}
\label{sec:unstable-real-realization}

Once we have ${\rm C}_2$-equivariant realization functors as described above, we get a real realization functor ${\rm Re}_{\mathbb{R}}\colon {\bf sPre}({\bf Sm}_{\mathbb{R}})\to {\bf Top}$ by composing the ${\rm C}_2$-equivariant realization functor with taking ${\rm C}_2$-fixed points. 

If we use a presheaf model for ${\rm C}_2$-equivariant homotopy theory (such as the Elmendorf model or the Morel--Voevodsky model discussed in Section~\ref{sec:equivariant-basics} above), then the genuine fixed-point functor is a left adjoint (as discussed in Section~\ref{sec:fixed-points} above). The composite real realization functor is then a left adjoint, the left Kan extension of the functor taking a smooth $\mathbb{R}$-scheme $X$ to its space $X(\mathbb{R})$ of real points with the analytic topology. This is the real realization functor $LR_2$ considered in \cite{bachmann:real-etale}, and it is the one we will use in the rest of the paper.\footnote{Note that if we use ${\rm C}_2$-spaces as model of equivariant homotopy, then the genuine fixed-point functor isn't a left adjoint on the point-set level. The composite real realization then doesn't have specific adjointness properties, as mentioned in \cite{wickelgren:williams}.}

\begin{remark}
  The functor $\varphi_{\mathbb{R}}^{-1}\colon X\mapsto X(\mathbb{R})$ defines a continuous map of sites
  \[
  \varphi_{\mathbb{R}}\colon({\rm T}lc)_{\rm open}\to({\bf Sm}_{\mathbb{R}})_{\rm Nis}.
  \]
  It seems that this map is reasonable in the sense of \cite{MorelVoevodsky1999}*{Definition~1.55 in Section~2.1}, because the functor $\varphi_{\mathbb{R}}^{-1}$ takes Nisnevich squares to a square of local diffeomorphisms which are locally split (obtained as fixed-points of a fine covering of locally contractible ${\rm C}_2$-spaces). The real realization above (as composite of fine ${\rm C}_2$-equivariant realization, followed by genuine fixed-points) could then alternatively be written as ${\bf L}\varphi_{\mathbb{R}}^\ast$, providing another way to see that it is the left Kan extension of taking real points.
\end{remark}

\subsection{Note on real realization in stable motivic homotopy}
\label{sec:stable-realization}

We collect a few remarks concerning stable realization functors. The unstable realization functors described above extend to the \emph{stable} motivic homotopy category. For complex and ${\rm C}_2$-equivariant realization, this is rather unproblematic. Any complex embedding $\sigma\colon F\hookrightarrow \mathbb{C}$ induces a symmetric monoidal left Quillen functor from symmetric motivic $\mathbb{P}^1$-spectra over $F$ to symmetric $\mathbb{CP}^1$-spectra of topological spaces, sending a symmetric spectrum to its topological realization, cf.{}~\cite{panin:pimenov:roendigs}. In particular, it sends suspension spectra $\Sigma^\infty_{\mathbb{P}^1} X_+$ to suspension spectra $\Sigma^\infty X(\mathbb{C})_+$ and is compatible with colimits. Similarly, the stable ${\rm C}_2$-equivariant realization is a symmetric monoidal left adjoint, cf.{}~\cite{heller:ormsby}*{Proposition~4.8}.

The stable real realization functor ${\rm Re}_{\mathbb{R}}\colon \mathbf{SH}(\mathbb{R})\to\mathbf{SH}$ can be described in two ways. One possibility, used by Bachmann in \cite{bachmann:real-etale} uses the real-\'etale topology and establishes an equivalence 
\[
\mathbf{SH}(\mathbb{R})[\rho^{-1}]\cong\mathbf{SH}.
\]
In this description, real realization can be interpreted as the $\rho$-localization functor. The other possibility, as in \cite{heller:ormsby}, is to define the real realization as the composite
\[
\mathbf{SH}(\mathbb{R})\xrightarrow{{\rm Re}_{\rm C_2}} \mathbf{SH}^{\rm C_2}\xrightarrow{\Phi^{\rm C_2}}\mathbf{SH}
\]
of the ${\rm C}_2$-equivariant realization and the geometric ${\rm C}_2$-fixed points. In \cite{bachmann:real-etale}*{Section~10}, Bachmann shows that these two versions of the realization functor coincide. Essentially, once we know that the functors are left adjoints, we can check that they agree on smooth schemes. The second description immediately implies the following statements:
\begin{proposition}\,%Hack
  \begin{enumerate}
  \item Stable real realization is a symmetric monoidal left adjoint.
  \item There is a commutative diagram
    \[
    \xymatrix{
      \mathbf{H}_\bullet(\mathbb{R}) \ar[r]^{\Sigma^\infty_{\mathbb{P}^1}} \ar[d]_{{\rm Re}_{\mathbb{R}}} &\mathbf{SH}(\mathbb{R}) \ar[d]^{{\rm Re}_{\mathbb{R}}} \\
      \mathbf{H}_\bullet \ar[r]_{\Sigma^\infty} &\mathbf{SH}
    }
    \]
    relating unstable and stable real realization (left and right vertical arrows, respectively).
  \end{enumerate}
\end{proposition}

\begin{proof}
  (1) Both equivariant realization and geometric fixed points are symmetric monoidal left adjoints, so the same is true for their composition.

  (2) One (if not \emph{the}) key characteristic of geometric fixed points is the equivalence
  \[
  \Phi^{\rm C_2}(\Sigma^\infty X_+)\cong \Sigma^\infty X_+^{\rm hC_2},
  \]
  i.e., the geometric fixed points of a suspension spectrum are equivalent to the suspension spectrum of the homotopy fixed points, for any ${\rm C}_2$-space. The diagram is just a restatement of this.
\end{proof}

While we are mostly interested in the unstable real realization of quotient stacks in this paper, our reason for bringing up the stable real realizations is to provide the general setting for realization maps on cohomology theories for quotient stacks. For any motivic spectrum $E$, the \emph{stable} real realization functor ${\rm Re}_{\mathbb{R}}$ induces a map
\[
E^{**}(X)=\left[\Sigma^\infty_{\mathbb{P}^1} X_+, E(*)[*]\right]_{\mathbf{SH}(\mathbb{R})}\to
\left[\Sigma^\infty \left({\rm Re}_{\mathbb{R}} X_+\right), {\rm Re}_{\mathbb{R}} (E)[*]\right]_{\mathbf{SH}}={\rm Re}_{\mathbb{R}}(E)^*({\rm Re}_{\mathbb{R}}X)
\]
from the $E$-cohomology of a space $X$ to the ${\rm Re}_{\mathbb{R}}(E)$-cohomology of the real realization ${\rm Re}_{\mathbb{R}}X$ of $X$. Such maps are generally powerful tools to relate algebraic and topological cohomologies and can be used to transport information from topological to motivic (e.g.{}~detecting non-vanishing of maps) or vice versa (e.g.{}~deducing vanishing of maps from weight reasons). 

In the rest of the paper, we will investigate the realization ${\rm Re}_{\mathbb{R}} [G\backslash X]$ of (the unstable motivic space associated to) the quotient stack for a variety with action $G\looparrowright X$. In general, the target of the realization map for a general motivic spectrum $E$ would be the ${\rm Re}_{\mathbb{R}}(E)$-cohomology of the real points of the quotient stack. For the current paper we are mostly interested in Witt-sheaf cohomology, which is represented by Eilenberg--Mac~Lane spectrum $\mathbb{H}{\bf W}_\bullet$ for the Witt-sheaf homotopy module. Using Jacobson's isomorphism $\op{colim}_n{\bf I}^n\xrightarrow{\cong}a_{\ret}\mathbb{Z}$, cf.{}~\cite{jacobson}*{Theorem~8.5} and \cite{bachmann:real-etale}*{Section~7}, the real realization of this spectrum is the Eilenberg--Mac~Lane spectrum $\mathbb{HZ}$ representing integral singular cohomology. In this case, the equivariant real realization map discussed above takes the form
\[
{\rm H}^p(X,{\bf W})\cong[\Sigma^\infty_{\mathbb{P}^1}X_+,\mathbb{H}{\rm W}_\bullet[p]]_{\mathbf{SH}(\mathbb{R})}\to [\Sigma^\infty({\rm Re}_{\mathbb{R}} X_+),\mathbb{HZ}[p]]\cong{\rm H}^p(X(\mathbb{R}),\mathbb{Z}).
\]
We will give a more explicit definition of this equivariant real cycle class map in Section~\ref{sec:equivariant-cycle-class}.

To end this section, we note that the equivariant real realization map would also be interesting to study for other motivic spectra as well; we list some examples. 

\begin{example}
  Natural examples to consider are spectra related to hermitian K-theory. The real realization of the hermitian K-theory spectrum ${\rm KO}$ is ${\rm KO}^{\rm top}[1/2]$, cf.{}~\cite{bachmann:hopkins}*{Lemma~3.9}. The corresponding equivariant real realization map would be 
\[
{\rm KO}^{**}([G\backslash X])\to {\rm KO}^{\rm top}[1/2]^*({\rm Re}_{\mathbb{R}}[G\backslash X]),
\]
in particular the non-connective version would lose some 2-torsion information in the real realization. Note that the real realization map is not interesting for non-connective algebraic K-theory, as ${\rm Re}_{\mathbb{R}}{\rm KGL}=0$, cf.{}~\cite{bachmann:hopkins}*{Lemma~3.9}, essentially since the real realization of the Bott element is nilpotent. It could be interesting for the effective cover kgl, though, which should be related to connective real topological K-theory ko, as fixed points of the infinite Grassmannian. In this context, it would be interesting to compare this to the Bruner--Greenlees computations of connective real K-theory of classifying spaces of finite groups, cf.{}~\cite{bruner:greenlees}. 
\end{example}

\begin{example}
  Similarly, it would be interesting to consider real realization maps for cobordism spectra. By~\cite{bachmann:hopkins}*{Lemma~4.4}, we have ${\rm Re}_{\mathbb{R}}{\rm MSp}\cong{\rm MU}$ and ${\rm Re}_{\mathbb{R}}{\rm MSL}\cong{\rm MSO}$. The real realization map in the special linear case would then have the form
  \[
  {\rm MSL}^{**}([G\backslash X])\to {\rm MSO}^*({\rm Re}_{\mathbb{R}}[G\backslash X]),
  \]
  which could be a way to better understand algebraic cobordism of quotient stacks using our description of the real realization in this paper.
\end{example}

\section{Coarse homotopy orbits in equivariant homotopy, and homotopy fixed points}
\label{sec:coarse}

In this section, we want to offer a somewhat different perspective on equivariant principal bundles and their classifying spaces. We will reinterpret equivariant principal bundles as torsors for the coarse topology on ${\rm C_2-T}lc$; from this perspective, the classifying spaces for equivariant principal bundles can then be constructed in direct analogy with the \'etale classifying spaces in the motivic setting. This description of equivariant principal bundle theory will be used in Section~\ref{sec:real-pts-hc2} to show that the ${\rm C}_2$-equivariant realization of quotient stacks in motivic homotopy will be quotient stacks in equivariant homotopy. We also discuss the relation of coarse Borel constructions with homotopy fixed points. What is written below is a sketch of the main ideas, details might appear in some other place.

\subsection{Equivariant principal bundles and coarse torsors}

We briefly recall from \cite{TomDieck}*{Chapters~I.6--8} the definition of a principal bundle, specialized to the situation where $\Gamma={\rm C}_2$ acts as an anti-holomorphic involution on a complex Lie group $G=G(\mathbb{C})$. Denote by $S=G\rtimes {\rm C}_2$ the semidirect product. 

\begin{definition}
  \label{def:principal-bundle}
A \emph{${\rm C}_2$-equivariant principal $G$-bundle}, also called $({\rm C}_2,G)$-bundle, is an $S$-space $E$, such that the $G$-action restricted from $S$ makes it a principal $G$-bundle. This is equivalent to saying that \emph{${\rm C}_2$ acts compatibly with $G$} on the principal $G$-bundle $E$, i.e.,\ for $\gamma\in {\rm C}_2$, $x\in E$ and $g\in G$:
$$ \ga.(x.g)=(\ga.x).(\ga.g).$$
This induces a unique ${\rm C}_2$-action on the quotient $B=E/G$, such that the projection $p\colon E\to B$ is ${\rm C}_2$-equivariant. A morphism of ${\rm C}_2$-equivariant principal $G$-bundles is an $S$-equivariant map between them.

Following the terminology of \cite{TomDieck}*{p.58}, a $({\rm C}_2,G)$-bundle is \emph{locally trivial}, if there exists a covering of $B$ by ${\rm C}_2$-stable open subsets $U_j$ such that $E|_{U_j}$ admits a bundle map to some local object, i.e., a $({\rm C}_2,G)$-bundle over an orbit ${\rm C}_2/\Lambda$ for some subgroup $\Lambda\leq{\rm C}_2$. The bundle is called \emph{numerable} if the covering $\{U_j\}$ of $B$ admits a subordinate partition of unity. 
\end{definition}

\begin{remark}
  As usual, classification results only work for locally trivial numerable bundles. In some places in the literature (such as \cite{TomDieck}*{Proposition~8.10} or \cite{GuillouMayMerling2017}), the group $G$ is assumed to be compact, to ensure automatic local triviality for principal bundles with completely regular total space. In our setting we consider nice spaces, which are locally contractible paracompact Hausdorff, so the compactness assumptions play no role and can be avoided.
\end{remark}

\begin{proposition}
  Let $G$ be a topological group with ${\rm C}_2$-action, and let $B$ be a ${\rm C}_2$-space in ${\rm C_2-T}lc$. Then a locally trivial ${\rm C}_2$-equivariant principal $G$-bundle over $B$ (in the sense of Definition~\ref{def:principal-bundle} above) is the same thing as a (representable) $G$-torsor over $B$ (in the sense of Definition~\ref{def:torsors}) in ${\rm C}_2$-spaces which is locally trivial in the coarse topology. 
\end{proposition}

\begin{proof}
  Let $p\colon E\to B$ be a ${\rm C}_2$-equivariant principal bundle. This means that $p$ is a morphism of ${\rm C}_2$-spaces, there is an action $\rho\colon G\times E\to E$ in ${\rm C}_2$-spaces, and $p$ is a locally trivial principal $G$-bundle. On the other hand, a (representable) $G$-torsor in ${\rm C}_2$-spaces is a ${\rm C}_2$-equivariant morphism $p\colon E\to B$ with an action $\rho\colon G\times E\to E$ such that $(\rho,{\rm pr}_2)\colon G\times E\to E\times_B E$ is an isomorphism. So the data in both situations are the same, and the isomorphism condition in the torsor setting encodes exactly that $p\colon E\to B$ is a principal $G$-bundle (non-equivariantly).

  To identify the local triviality conditions, we first note that all local objects over ${\rm C}_2/{\rm C}_2$ become isomorphic after pullback along the projection map ${\rm C}_2/1\to{\rm C}_2/{\rm C}_2$, which is a coarse covering. In particular, for a locally trivial equivariant principal bundle, we can take the trivializing open covering $\{U_j\}$ and refine it to a coarse covering which trivializes the corresponding $G$-torsor.\footnote{Note that the definition of local triviality of $({\rm C}_2,G)$-bundles in Definition~\ref{def:principal-bundle} actually doesn't contain a condition on the fixed-point sets, relating it more to the coarse than the fine topology.}
 Conversely, recall that the definition of coarse covering $X\to Y$ in \cite{MorelVoevodsky1999} is a ${\rm C}_2$-equivariant morphism such that any point $y\in Y$ has an open neighbourhood $U$ such that the projection $X\times_YU \to U$ splits as morphism of topological spaces. In particular, this provides an open covering of $Y$ (which we can assume is good in the sense of \cite{MorelVoevodsky1999}*{Definition~3.3.1}). This provides the trivializing open cover for the equivariant principal bundle, which can then be refined to a ${\rm C}_2$-stable open cover. The local models are obtained from the fiber over an orbit, and the coarse-local triviality and the open cover being good ensure the equivariant trivialization condition.
\end{proof}

\begin{remark}
  The intuition here is that the coarse topology adds the ${\rm C}_2$-equivariant versions of \'etale coverings like $\op{Spec}(\mathbb{C})\to\op{Spec}(\mathbb{R})$ which allow to identify different real forms, in other words, different local models of the ${\rm C}_2$-group $G$.
\end{remark}

As a consequence of this identification, we have descent of ${\rm C}_2$-equivariant $G$-principal bundles along coarse coverings. 

\begin{corollary}
  Coarse-locally trivial $G$-torsors satisfy descent in the coarse topology, i.e., they can be glued in coarse coverings. More precisely, let $f\colon X\to Y$ be a coarse covering, and denote by ${\rm pr}_1,{\rm pr}_2\colon X\times_YX\to X$ the two projections. If $p\colon E\to X$ is a coarse-locally trivial $G$-torsor with an isomorphism $\alpha\colon{\rm pr}_1^\ast E\cong{\rm pr}_2^* E$, then there exists a unique coarse-locally trivial $G$-torsor $\overline{E}\to Y$ with $f^*\overline{E}\cong E$.
\end{corollary}

\begin{proof}
  For a point $y\in Y$, the descent isomorphism $\alpha$ fixes how the different fibers of $E$ over points $x\in f^{-1}(y)$ have to be identified to define the $G$-torsor $\overline{E}$ over $Y$ with $f^*\overline{E}\cong E$. By definition, this $G$-torsor is going to be locally trivial in the coarse topology (compose a trivializing cover for $E$ with $f$ to a trivializing cover for $\overline{E}$). 
\end{proof}

\begin{proposition}
  Equivariant principal bundles satisfy homotopy invariance, in the sense that, for $X$ in ${\rm C_2-T}lc$, bundles over $X\times{\rm I}$ (interval with trivial ${\rm C}_2$-action) are extended from $X$.
\end{proposition}

\begin{proof}
  This is \cite{TomDieck}*{Proposition~8.13, Theorem~8.15}.
\end{proof}

\subsection{Universal equivariant principal bundle as coarse homotopy orbit space}

Having reinterpreted equivariant principal bundles as torsors in ${\rm C}_2$-spaces which are locally trivial for the coarse topology, the classification theory of equivariant principal bundles can be developed closely along the lines of the algebraic setting described in Sections~\ref{sec:prelims} and \ref{sec:stacks-motivic}. We outline the main points.

For a ${\rm C}_2$-space with action $G\looparrowright X$, we can consider a quotient stack $[G\backslash X]$ whose groupoid of points $[G\backslash X](S)$ over a space $S\in{\rm C_2-T}lc$ is the groupoid of spans $S\xleftarrow{p} P\xrightarrow{f}X$ where $p\colon P\to S$ is a coarse-locally trivial $G$-torsor and $f\colon P\to X$ is $G$-equivariant. The resulting quotient stack is a sheaf of groupoids for the coarse topology on ${\rm C_2-T}lc$ (essentially by the coarse descent property for equivariant principal bundles). In this setting, we also have versions of the quotient and induction equivalences of Proposition~\ref{prop:quotient-induction-stacks}.

We can also view the equivariant quotient stack $[G\backslash X]$ as an object in equivariant homotopy in several ways, following the material in Section~\ref{sec:stacks-motivic}. Via an analogue of Definition~\ref{def:stack-motivic-nerve}, we can take the geometric realization of the nerve of the sheaf of groupoids, viewed as an object in the homotopy category of fine I-local sheaves.\footnote{Note that by the discussion of Section~\ref{sec:equivariant-basics}, this is equivalent to the usual genuine ${\rm C}_2$-equivariant homotopy category.} In complete analogy to the \'etale-vs-Nisnevich story in Section~\ref{sec:stacks-motivic}, this realization of the quotient stack $[G\backslash X]$ will be the coarse homotopy orbit space, which we denote here by $\mathcal{X}_{{\rm h}G}$; it can be constructed as homotopy orbits in coarse equivariant homotopy, followed by a derived pushforward along the change-of-topology morphism from the coarse to the fine equivariant homotopy theory. 

Coarse descent plus homotopy invariance (as discussed in the previous section) imply the following representability result for the realization $\mathcal{X}_{{\rm h}G}$ of the quotient stack, using the techniques in and around \cite{affine-representability}:

\begin{theorem}
  \label{thm:equivariant-representability}
  Let $G$ be a complex Lie group with antiholomorphic ${\rm C}_2$-action, and let $G\looparrowright X$ be a ${\rm C}_2$-space with action. We view the coarse homotopy orbit space $\mathcal{X}_{{\rm h}G}$ as fine sheaf of spaces on ${\rm C_2-T}lc$. Then for any $S\in{\rm C_2-T}lc$, there are equivalences 
  \[
  {\rm Map}_{\rm fine, I}\left(S,\mathcal{X}_{{\rm h}G}\right)\simeq\left|{\rm Sing}_*^{\rm I}\left([G\backslash X](S)\right)\right|
  \]
  which are natural in $S\in{\rm C_2-T}lc$. 
  Spelling it out, this means that (the singular resolution of) $[G\backslash X](S)$ is a representative of the coarse homotopy orbits in the fine equivariant homotopy category. Consequently, the sections of $\mathcal{X}_{{\rm h}G}$ over a space $S$ can be computed as geometric realization of the topological groupoid $[G\backslash X](S)$.

  In particular, in the case where ${\rm B}_{\rm cs}G$ is the coarse homotopy orbit space/realization of the quotient stack for the trivial action $G\looparrowright \pt$, we get
  \[
  \left[X,{\rm B}_{\rm cs}G\right]_{\rm fine, I}\cong {\rm H}^1_{\rm cs}(X,G).
  \]
  This means that the coarse homotopy orbit space ${\rm B}_{\rm sc}G$ is a classifying space for ${\rm C}_2$-equivariant principal $G$-bundles, and the naive ${\rm I}$-homotopy classes of maps into $[G\backslash X]$ compute isomorphism classes of coarse $G$-torsors.
\end{theorem}

Following the path of Section~\ref{sec:stacks-motivic} further, we have alternative descriptions of the realizations of equivariant quotient stacks.

On the one hand, we can build a simplicial object, realizing a homotopy orbit space for the fine equivariant homotopy, along the lines of Section~\ref{sec:stacks-motivic-simplicial}. The result will be a ${\rm C}_2$-space ${\rm E}_G(X)$ whose underlying space is the Borel construction of the underlying action $G\looparrowright X$, and whose fixed-point space is the Borel construction of the fixed-point action $G^{\rm C_2}\looparrowright X^{\rm C_2}$. As an analogue of Proposition~\ref{prop:simplicial-comparison}, the coarse stackification of this fine homotopy orbit space will be the realization of the quotient stack $[G\backslash X]$. 

In the equivariant setting, we can call a ${\rm C}_2$-group $G$ \emph{special}, if its coarse cohomology set
\[
  {\rm H}^1_{\rm cs}(\pt,G)\cong\pi_0([G\backslash\pt](\pt))
\]
(of isomorphism classes of $G$-torsors over the point $\pt$ with trivial ${\rm C}_2$-action) has only one element. For actions $G\looparrowright X$ of a special group $G$, the fine homotopy orbit space ${\rm E}_G(X)$, given by a simplicial construction, will be equivalent to the realization of the quotient stack. This is the equivariant analogue of Krishna's equivalence in \cite{krishna}, see also the second part of Proposition~\ref{prop:simplicial-comparison}. 

On the other hand, it is also possible to approximate the coarse homotopy orbit space $\mathcal{X}_{{\rm h}G}$ by quotients of open subsets in linear $G$-representations; this is the equivariant analogue of the Totaro--Edidin--Graham construction in Section~\ref{sec:stacks-motivic-approx}. We will not have opportunity in this paper to use this. 

\begin{remark}
  The outline above can be seen as a different take on the categorical models of classifying spaces of equivariant principal bundles in \cite{GuillouMayMerling2017}, in particular Theorems~3.10 and 3.11. It is not clear to us if the connection of equivariant principal bundles and coarse-locally trivial torsors has been noted in the equivariant literature before. We found it useful and natural, in particular for the study of equivariant realizations of motivic quotient stacks. It also seems to us that the approach via sheaves of spaces on ${\rm C_2-T}lc$ avoids some of the point-set or compactness conditions in \cite{GuillouMayMerling2017} and similar places. In any case, whenever the point-set conditions are satisfied (and we believe that they should be for complexifications of reductive real Lie groups), the classifying spaces in \cite{GuillouMayMerling2017} will be equivalent to the coarse homotopy orbit space in Theorem~\ref{thm:equivariant-representability} above.
\end{remark}

\subsection{Fixed-point groupoids and homotopy fixed points}
\label{sec:fp-gpd-hfp}

From Theorem~\ref{thm:equivariant-representability}, we can quickly read off a description of the genuine fixed-point space of the coarse homotopy orbit space $\mathcal{X}_{{\rm h}G}$ for a ${\rm C}_2$-space with action $G\looparrowright X$:

\begin{corollary}
  \label{cor:equivariant-bg-fp}
  Let $G$ be a complex Lie group with antiholomorphic ${\rm C}_2$-action, and let $G\looparrowright X$ be a ${\rm C}_2$-space with action. We view the coarse homotopy orbit space $\mathcal{X}_{{\rm h}G}$ as a fine sheaf of spaces on ${\rm C_2-T}lc$. Then we have the following:
  \begin{enumerate}
  \item The genuine fixed-point space $\mathcal{X}_{{\rm h}G}^{\rm C_2}$ is equivalent to the geometric realization of topological groupoid $[G\backslash X](\pt)$.
  \item The natural map ${\rm C}_2\to\pt$ of ${\rm C}_2$-spaces induces an equivalence of topological groupoids
    \[
    [G\backslash X](\pt)\simeq [G\backslash X]({\rm C}_2)^{\rm hC_2},
    \]
    i.e., $[G\backslash X](\pt)$ is equivalent to the fixed-point groupoid for the natural ${\rm C}_2$-action on $[G\backslash X]({\rm C}_2)$.
  \item The realization of $[G\backslash X](\pt)$ is the homotopy fixed-point space of the ${\rm C}_2$-action on the realization of $[G\backslash X]({\rm C}_2)$. 
  \end{enumerate}
\end{corollary}

\begin{proof}
  (1) is a direct consequence of Theorem~\ref{thm:equivariant-representability}, in the special case where $S=\pt$.

  (2) The natural ${\rm C}_2$-action on $[G\backslash X]({\rm C}_2)$ is induced from the ${\rm C}_2$-action on ${\rm C}_2$ switching the two points. The identification of the fixed-point groupoid with $[G\backslash X](\pt)$ follows from coarse descent, similar to the arguments in Section~\ref{sec:real-pts-fp-gpd}.

  (3) As remarked above, we can construct $\mathcal{X}_{{\rm h}G}$ (as a sheaf of spaces on ${\rm C_2-T}lc$) as homotopy orbits in the coarse model structure, followed by change-of-topology pushforward. By construction, this means that the sections of $\mathcal{X}_{{\rm h}G}$ over $\pt$ are identified as homotopy fixed-points of the natural ${\rm C}_2$-action on $\mathcal{X}_{{\rm h}G}({\rm C}_2)$, see also our discussion on fixed-point functors in Section~\ref{sec:fixed-points}. 
\end{proof}

\begin{remark}
  This provides a different perspective on the computation of the fixed-points of an equivariant classifying space in \cite{GuillouMayMerling2017}, in particular Theorems~4.18 and 4.23. After identifying the fixed-point space with the realization of the topological groupoid $[G\backslash X](\pt)$, the description of the connected components in terms of ${\rm H}^1_{\rm cs}(\pt,G)$ (as an equivariant-topology version of Galois cohomology) is immediate.
\end{remark}

\begin{remark}
  That the realization of fixed-point groupoids compute homotopy fixed-points is also a standard result found in many places in the literature. For discrete groups, it is in Thomason's homotopy limit paper \cite{thomason:holim}, Section~3 and in particular (3.3) on page 410, see also the discussion in \cite{virk:hc2}*{Example~2.0.1 and Proposition~2.1.4}. For compact groups, it also appears in \cite{GuillouMayMerling2017}, Section~4, in the discussion of the fixed-point groupoids $\mathscr{C}at(\mathcal{E}G,\Pi)^G$. 
\end{remark}

\section{Equivariant and real realization for quotient stacks}
\label{sec:real-pts-hc2}

In this section, we will compute the equivariant and real realization of quotient stacks, thus proving Theorem~\ref{thm:main1} in the introduction. The main result is the following: for a group action $G\looparrowright X$ over $\mathbb{R}$, the equivariant realization of the associated quotient stack $[G\backslash X]$ is the quotient stack for the equivariant realization  $G(\mathbb{C})\looparrowright X(\mathbb{C})$ of the action (in ${\rm C}_2$-spaces). Put differently, the equivariant realization of \'etale homotopy orbits are coarse-equivariant homotopy orbits. In particular, the equivariant realization of the geometric classifying space ${\rm B}_\et G$ of a linear group $G$ is the ${\rm C}_2$-space classifying equivariant principal $({\rm C}_2,G(\mathbb{C}))$-bundles, as discussed in Section~\ref{sec:coarse}. Taking fixed points, this means that the real realization of a real quotient stack $[G\backslash X]$ is given by the homotopy fixed points of the complex realization. Combining this with the results from Section~\ref{sec:hc2-galois} we see that the (unstable) real realization can also be computed as the geometric realization of the topological groupoid $[G\backslash X](\mathbb{R})$ of real points of the quotient stack; see also the explicit formula in Proposition~\ref{prop:bg-formula}. 

To prove the equivariant results, we first note that Nisnevich homotopy orbits are mapped to fine equivariant homotopy orbits because Nisnevich-to-fine realization is a left Kan extension. In particular, for a group action $G\looparrowright X$ of a special group $G$, the equivariant realization ${\rm Re}_{\rm C_2}[G\backslash X]$ is the action groupoid of $G(\mathbb{C})\looparrowright X(\mathbb{C})$ (viewed as group action in ${\rm C_2}$-spaces). For the general case, we can then use the induction equivalence $[G\backslash X]\cong[{\rm GL}_n\backslash({\rm GL}_n\times_{/G}X)]$ of Proposition~\ref{prop:quotient-induction-stacks} to give a ${\rm GL}_n$-presentation. This shows generally that equivariant realizations of quotient stacks are coarse homotopy orbit spaces.

We want to point out that, for our purposes in this paper, both descriptions of real points of a quotient stack are important. On the one hand, identifying the real points of the quotient stack with the homotopy fixed points of the complex conjugation on the complex points allows for very direct computations of the real realization of classifying spaces, see the formula in Proposition~\ref{prop:bg-formula} as well as the various examples discussed in Section~\ref{sec:examples}. On the other hand, the scheme approximation of the algebraic Borel construction is the key for our definition of equivariant real cycle class map in Section~\ref{sec:equivariant-cycle-class}.

\begin{convention}
  For the purposes of this section, whenever we want to specifically distinguish between the real points of a quotient stack and its real realization, we will use the following notation: $[G\backslash X](\mathbb{R})$ denotes the (topological) groupoid of real points of the quotient stack (as in Example~\ref{ex:real-points} or Section~\ref{sec:hc2-galois}), and ${\rm Re}_{\mathbb{R}}[G\backslash X]$ denotes the real realization (as in Section~\ref{sec:unstable-real-realization}). After having shown that these are equivalent, the notational distinction might be blurred somewhat in later sections.
\end{convention}

\subsection{Complex realization} 

We begin with an aside concerning the description of the complex realization of a quotient stack. As we recalled in Section~\ref{sec:prelims}, there are several possible (but equivalent) ways to consider a quotient stack as an object in motivic homotopy. Consequently there are also different possible ways to describe its complex realization: on the more simplicial side, we can take the realization of the action groupoid, while on a more geometric side, we could use an admissible gadget and take the complex realization of the finite-dimensional approximations of the Borel construction, in the style of Totaro~\cite{totaro:bg} and Edidin--Graham~\cite{edidin:graham:equivariant}. 

We first consider the complex realization of an admissible gadget $X_G(\rho)$ for $G\looparrowright X$, cf.{}~Section~\ref{sec:stacks-motivic-approx}. The result is a Borel construction for $G(\mathbb{C})\looparrowright X(\mathbb{C})$ since complex realization is a left Kan extension. 

\begin{proposition}
  \label{prop:cplx-admissible-gadget}
  Let $G\looparrowright X$ be a smooth quasi-projective variety with action over $\mathbb{C}$, and let $\rho=(U_i,V_i)_{i\geq 1}$ be an admissible gadget for $G$. Then the complex realization of $X_G(\rho)$ is a Borel construction for $G(\mathbb{C})\looparrowright X(\mathbb{C})$. More precisely,
  \[
  {\rm Re}_{\mathbb{C}}\left(X_G(\rho)\right)\simeq \op{colim}_{i\geq 1} \left(U_i(\mathbb{C})\times_{/G(\mathbb{C})} X(\mathbb{C})\right)\simeq{\rm E}G(\mathbb{C})\times_{/G(\mathbb{C})}X(\mathbb{C}).
  \]
  In particular, the complex realization can be identified as homotopy colimit of the complex points of scheme approximations (in the sense of Totaro and Edidin--Graham) of $X_G(\rho)$.
\end{proposition}

\begin{proof}
  Recall from Section~\ref{sec:stacks-motivic-approx} that by definition $X_G(\rho)=\op{colim}_i \left(U_i\times_{/G} X\right)$, and this colimit is indeed a homotopy colimit since it is taken over a filtered index category. Since complex realization is a left Quillen functor, ${\rm Re}_{\mathbb{C}}(X_G(\rho))$ is the (homotopy) colimit of ${\rm Re}_{\mathbb{C}}(U_i\times_{/G}X)$.

  As we recalled above, e.g.{}~from \cite{wickelgren:williams}*{Section~3.1}, complex realization commutes with products, and therefore ${\rm Re}_{\mathbb{C}}(U_i\times X)\cong U_i(\mathbb{C})\times X(\mathbb{C})$. Moreover, the diagonal action $G\looparrowright U_i\times X$ realizes to a complex Lie group action  $G(\mathbb{C})\looparrowright U_i(\mathbb{C})\times X(\mathbb{C})$. Finally, we can again use that complex realization commutes with colimits (more specifically, the coequalizer of the action groupoid $G\times U_i\times X\rightrightarrows U_i\times X$) to see that ${\rm Re}_{\mathbb{C}}(U_i\times_{/G}X)\cong U_i(\mathbb{C})\times_{/G(\mathbb{C})} X(\mathbb{C})$. This proves the first equivalence in the proposition.
  
  The second equivalence is (also) standard. We have:
  \[
  \op{colim}_{i\geq 1} \left(U_i(\mathbb C)\times_{/G(\mathbb C)} X(\mathbb C)\right)\simeq \left(\op{colim}_{i\geq 1} U_i(\mathbb C) \right)\times_{/G(\mathbb C)} X(\mathbb C)\simeq \op{Re}_{\mathbb C}(\op{colim}_{i\geq 1} U_i)\times_{G(\mathbb C)} X(\mathbb C).
  \]
  But then the motivic space $\op{colim}_i U_i$ is contractible, cf. Proposition~2.3 in Section~4.2 of \cite{MorelVoevodsky1999}, so its realisation is also contractible. The action of $G(\mathbb C)$ on $\op{colim}_{i\geq 1} U_i(\mathbb C)$ is free since it was so before taking realisations. We thus conclude that $\op{Re}_{\mathbb C}(\op{colim}_{i\geq 1} U_i)$ is a model for ${\rm E} G(\mathbb C)$.
\end{proof}

We next discuss the complex realization using the action groupoid viewpoint. We use the notation ${\rm E}_G(X)$ for the bar construction associated to a variety with action $G\looparrowright X$, cf. also the discussion in Section~\ref{sec:stacks-motivic-simplicial}. 

\begin{proposition}
  \label{prop:cplx-simplicial}
  Let $G\looparrowright X$ be a smooth variety with action over $\mathbb{C}$. Then the complex realization  ${\rm Re}_{\mathbb{C}}(E_G(X))$ of the action groupoid ${\rm E}_G(X)$ is the bar construction of the complex Lie group action $G(\mathbb{C})\looparrowright X(\mathbb{C})$. 
\end{proposition}

\begin{proof}
  Recall from Definition~\ref{def:action_groupoid} that ${\rm E}_G(X)$ is the simplicial scheme
  \[\begin{tikzcd}
  \cdots \ar[r,yshift=4pt,->]\ar[r,yshift=1.5pt,->]\ar[r,yshift=-1.5pt,->]\ar[r,yshift=-4pt,->] &
  G\times G\times X \ar[r,yshift=3pt,->]\ar[r,yshift=0pt,->]\ar[r,yshift=-3pt,->] & 
  G\times X \ar[r,yshift=2pt,->]\ar[r,yshift=-2pt,->] &X
  \end{tikzcd}\]
  with the face maps given by group multiplication, group action and projection.

  There are two ways we can think of the complex realization of ${\rm E}_G(X)$. On the one hand, we can consider it as a simplicial scheme, and apply the functor ${\rm Re}_{\mathbb{C}}$ termwise to get a simplicial space. On the other hand, we can think of ${\rm E}_G(X)$ as an object in the motivic homotopy category and apply ${\rm Re}_{\mathbb{C}}$ to get a space. The latter space is then the geometric realization of the former simplicial space. In particular, it suffices to show that the termwise realization is the bar construction for the action $G(\mathbb{C})\looparrowright X(\mathbb{C})$ as claimed in the proposition.

  Since complex realization is compatible with finite products, ${\rm Re}_{\mathbb{C}}({\rm E}_G(X))$ is a simplicial object built from $G(\mathbb{C})$ and $X(\mathbb{C})$, with face maps given by the complex realization of the group multiplication, group action and projections. But this simplicial space is exactly the bar construction for the complex Lie group action $G(\mathbb{C})\looparrowright X(\mathbb{C})={\rm Re}_{\mathbb{C}}(G\looparrowright X)$, cf.{}~also the discussion of realization of bar constructions in \cite{wickelgren:williams}*{Section~3.1}. 
\end{proof}

Combining the above results, we get the following description of complex realization of a variety with action: 

\begin{corollary}
  Let $G\looparrowright X$ be a smooth quasi-projective variety with action.
  \begin{enumerate}
  \item The natural morphism ${\rm E}_G(X)\to [G\backslash X]$ of Proposition~\ref{prop:simplicial-comparison} induces an equivalence after complex realization:
    \[
      {\rm Re}_{\mathbb{C}}{\rm E}_G(X)\xrightarrow{\cong} {\rm Re}_{\mathbb{C}}[G\backslash X].
    \]
  \item The natural morphism $X_G(\rho)\to [G\backslash X]$ of Proposition~\ref{prop:approx-comparison} induces an equivalence after complex realization:
    \[
    {\rm Re}_{\mathbb{C}}X_G(\rho)\xrightarrow{\cong}{\rm Re}_{\mathbb{C}}[G\backslash X].
    \]
  \item In case $G$ is special, the equivalences in (1) and (2) are the realizations of Krishna's equivalence ${\rm E}_G(X)\cong X_G(\rho)$, cf. \cite{krishna}*{Proposition~3.2} or Sections~\ref{sec:stacks-motivic-simplicial} and \ref{sec:stacks-motivic-approx}. 
  \end{enumerate}
  In particular, the complex realization of a quotient stack can be computed either using a Borel construction for the action $G(\mathbb{C})\looparrowright X(\mathbb{C})$, or as the nerve of the topological action groupoid $(G\times X)(\mathbb{C})\rightrightarrows X(\mathbb{C})$.\footnote{For emphasis, the objects $G(\mathbb{C})$ and $X(\mathbb{C})$ are equipped with the analytic topologies.}
\end{corollary}

\begin{remark}
  Note that (1) in the corollary does not depend on $G$ being special. The natural morphism ${\rm E}_G(X)\to [G\backslash X]$ on the motivic side is not generally an equivalence, but its complex realization always is. 
\end{remark}

\subsection{Equivariant and real realization of Nisnevich homotopy orbits}

We now turn to ${\rm C}_2$-equivariant and real realization of quotient stacks for real group actions. We first consider the case of special groups where we can use Krishna's equivalence ${\rm E}_G(X)\cong[G\backslash X]$, cf.{}~Proposition~\ref{prop:simplicial-comparison}, to identify the realization of the quotient stack. More generally, we also consider the equivariant realization of Nisnevich homotopy orbit spaces. The following proves the Nisnevich homotopy orbits part of Theorem~\ref{thm:main1} in the introduction:

\begin{proposition}
  \label{prop:equiv-realization-bar}
  Let $G\looparrowright X$ be variety with action over $\mathbb{R}$. Then the equivariant realization ${\rm Re}_{\rm C_2}({\rm E}_G(X))$ of the (motivic) bar construction ${\rm E}_G(X)$ is the (fine equivariant) bar construction ${\rm E}_{G(\mathbb{C})}(X(\mathbb{C}))$ of the ${\rm C}_2$-equivariant group action $G(\mathbb{C})\looparrowright X(\mathbb{C})$.
  \end{proposition}

\begin{proof}
  The proof is essentially the same as in the complex case, cf.{}~Proposition~\ref{prop:cplx-simplicial}: ${\rm C}_2$-equivariant realization is a left Kan extension, and therefore maps bar constructions to bar constructions, or alternatively (Nisnevich-motivic) homotopy orbits to (fine-equivariant) homotopy orbits. In slightly more and model-dependent detail, equivariant realization of the simplicial scheme ${\rm E}_G(X)$ is given as the geometric realization of degreewise equivariant realization; the geometric realization of a simplicial object is the homotopy colimit of the simplicial diagram, and then we can simply use that equivariant realization is a left adjoint. 
\end{proof}

\begin{remark}
  The result shows, in particular, that the ${\rm C}_2$-equivariant realization of a Nisnevich classifying space ${\rm B}_{\rm Nis}G$ is the ${\rm C}_2$-space with underlying space ${\rm B}G(\mathbb{C})$ and fixed-point space ${\rm B}G(\mathbb{R})$. 
\end{remark}

\begin{corollary}
  Assume now that $G$ is special in the sense of Serre, i.e., \'etale-locally trivial $G$-torsors are Nisnevich-locally trivial. Then we have a fine ${\rm C_2}$-equivariant equivalence
  \[
  {\rm Re}_{\rm C_2}({\rm E}_G(X))\xrightarrow{\simeq}{\rm Re}_{\rm C_2}([G\backslash X]).
  \]
\end{corollary}

\begin{proof}
  This follows from Krishna's equivalence ${\rm E}_G(X)\xrightarrow{\simeq}[G\backslash X]$, cf.{}~Proposition~\ref{prop:simplicial-comparison}, by applying equivariant realization. 
\end{proof}

\begin{proposition}
  \label{prop:krishna-equiv-realization}
  Let $G\looparrowright X$ be a variety with action over $\mathbb{R}$. Then the real realization of the (motivic) bar construction ${\rm E}_G(X)$ is the (genuine) fixed-point space of the equivariant bar construction: 
  \begin{align}
    \label{eq:real-realization-hfp}
  {\rm Re}_{\mathbb{R}}({\rm E}_G(X))\cong \left({\rm E}_{G(\mathbb{C})}(X(\mathbb{C})\right)^{{\rm C}_2}\cong {\rm E}_{G(\mathbb{R})}\left(X(\mathbb{R})\right).
  \end{align}
  If $G$ is special in the sense of Serre, then the real realization can be identified with the nerve of the groupoid $[G\backslash X](\mathbb{R})$ of real points of the quotient stack $[G\backslash X]$. In particular, the real realization can also be computed as geometric realization of the simplicial space ${\rm E}_G(X)(\mathbb{R})\cong\mathcal{N}\left([G\backslash X](\mathbb{R})\right)$.
\end{proposition}

\begin{proof}
  This follows from Proposition~\ref{prop:equiv-realization-bar} by taking (genuine) fixed points. Alternatively, we can make an argument as in the proof of Proposition~\ref{prop:equiv-realization-bar}, using that real realization is a left Kan extension, see the discussions in Section~\ref{sec:fixed-points} and Section~\ref{sec:unstable-real-realization}.

  For the assertion about the case that $G$ is special, note that Krishna's equivalence ${\rm E}_G(X)\xrightarrow{\simeq}[G\backslash X]$, cf.{}~Proposition~\ref{prop:simplicial-comparison}, implies that the groupoid of real points $[G\backslash X](\mathbb{R})$ is equivalent (as a topological groupoid) to the action groupoid of the real action $G(\mathbb{R})\looparrowright X(\mathbb{R})$. The claim follows by taking nerves (of topological groupoids).
\end{proof}

In particular, for a special group $G$, this already provides the computation of the realization of quotient stacks for $G$-actions. In the next section, we will deduce the general case from this special case by using ${\rm GL}_n$-presentations for quotient stacks. 

\subsection{Equivariant and real realization of quotient stacks}
\label{sec:real-points}

Now we turn to the ${\rm C}_2$-equivariant and real realization of general quotient stacks, proving Theorem~\ref{thm:main1} in the introduction. We will see that several different ways of describing the realizations (via admissible gadgets, as coarse homotopy orbits, or via the nerve of the groupoid of real points) are all equivalent. 

We first consider the realization of a quotient stack $[G\backslash X]$ using an admissible gadget, cf.{}~the discussion in Section~\ref{sec:stacks-motivic-approx}. As in the previous discussions of complex realization, the admissible gadget viewpoint allows to present the stack as a (homotopy) colimit $[G\backslash X]\cong\op{colim}_i (U_i\times_{/G}X)$, and then take equivariant or real realization. 

\begin{proposition}
  \label{prop:real-admissible-gadget}
  Let $G\looparrowright X$ be a smooth quasi-projective variety with action over $\mathbb{R}$, and let $\rho=(U_i,V_i)_{i\geq 1}$ be an admissible gadget for $G$. Then the ${\rm C}_2$-equivariant realization of $X_G(\rho)$ is the complex realization of Proposition~\ref{prop:cplx-admissible-gadget}, equipped with the involution induced from  complex conjugation. The real realization of $X_G(\rho)$ is computed as the real points, i.e., (genuine) ${\rm C}_2$-fixed points, of the Borel construction ${\rm E}G(\mathbb{C})\times_{/G(\mathbb{C})}X(\mathbb{C})$:
  \[
  {\rm Re}_{\mathbb{R}}(X_G(\rho))\simeq \op{colim}_{i\geq 1} \left(U_i\times_{/G} X\right)(\mathbb{R})\simeq \left({\rm E}G(\mathbb{C})\times_{/G(\mathbb{C})}X(\mathbb{C})\right)^{{\rm C}_2}.
  \]
 In particular, the real realization can be identified as homotopy colimit of the real points of scheme approximations (in the sense of Totaro and Edidin--Graham) of $X_G(\rho)$. 
\end{proposition}

\begin{proof}
  We again use the notation $X_G(\rho)=\op{colim}_u(U_i\times_{/G}X)$ from Section~\ref{sec:stacks-motivic-approx}. As for complex realization, ${\rm C}_2$-equivariant realization is a left Kan extension, cf.{}~\cite{wickelgren:williams}*{Section~3.2}, implying that ${\rm Re}_{\rm C_2}(X_G(\rho))$ is the (homotopy) colimit of ${\rm Re}_{\rm C_2}(U_i\times_{/G}X)$. Also, as before, the colimits appearing here are in fact homotopy colimits.   
  
  The claim on ${\rm C}_2$-equivariant realization follows easily, similar to the proof of Proposition~\ref{prop:cplx-admissible-gadget}: for the smooth $\mathbb{R}$-schemes $U_i$ and $X$, the involution on  ${\rm Re}_{\rm C_2}(U_i\times_{/G}X)$ is the natural complex conjugation on $U_i(\mathbb{C})\times_{/G(\mathbb{C})}X(\mathbb{C})\cong (U_i\times_{/G}X)(\mathbb{C})$, coming from the real structure. The transition maps between $U_i\times_{/G}X$ are also compatible with complex conjugation. This induces an involution on ${\rm Re}_{\mathbb{C}}(X_G(\rho))\cong\op{colim}_{i\geq 1}\left(U_i(\mathbb{C})\times_{/G(\mathbb{C})}X(\mathbb{C})\right)$, and this is the involution for the ${\rm C}_2$-equivariant realization.

  The statement identifying the real realization as ${\rm C}_2$-fixed points on the Borel construction follows immediately from this by taking ${\rm C}_2$-fixed points. The alternative description of real realization as colimit of real points of the scheme approximations in the admissible gadget follows, since real realization is a left Kan extension, cf. the discussions in Section~\ref{sec:fixed-points} and Section~\ref{sec:unstable-real-realization}.\footnote{If we were using ${\rm C}_2$-spaces as our model of equivariant topology, the claim would still be true, using that taking ${\rm C}_2$-fixed points commutes with pushouts along closed inclusions and sequential colimits along closed inclusions, cf.{}~\cite{schwede}*{Proposition~B.1}. Note that the maps $U_i\times_{/G} X\to U_{i+1}\times_{/G}X$ are closed immersions on the algebraic side, and therefore induce closed inclusions in ${\rm C}_2$-equivariant realization.}
\end{proof}

\begin{remark}
  Note that $\left(U_i(\mathbb{C})\times_{/G(\mathbb{C})} X(\mathbb{C})\right)^{{\rm C}_2}$ is not necessarily equivalent to the quotient construction $U_i(\mathbb{R})\times_{/G(\mathbb{R})}X(\mathbb{R})$ on real points. The simplest example is the quotient $(\mathbb{C}^n\setminus\{0\})/\mu_2$ for the scalar multiplication action of $\mu_2=\{\pm 1\}$. The quotient has two components of real points (coming from points in $\mathbb{C}^n\setminus\{0\}$ with all coordinates real, or all coordinates imaginary). The quotient $\mathbb{R}^n\setminus\{0\}$ of points with all coordinates real is only one of these components. In the end, the whole point of the present discussion is to better understand the real points of the admissible gadget in terms of homotopy fixed points.
\end{remark}

\begin{remark}
  As in the complex case, see the last paragraph of the proof of Proposition~\ref{prop:cplx-admissible-gadget}, the ${\rm C}_2$-equivariant realization of $\op{colim}_iU_i$ is a (coarse-equivariant) contractible ${\rm C}_2$-space. Therefore, ${\rm Re}_{\rm C_2}(\op{colim}_iU_i)$ is a model for ${\rm E}G^{\rm C_2}$ (the contractible ${\rm C}_2$-space which is the total space of the universal principal $({\rm C}_2,G(\mathbb{C}))$-bundle). Consequently, the equivariant realization ${\rm Re}_{\rm C_2}X_G(\rho)$ is an equivariant Borel construction for the action $G(\mathbb{C})\looparrowright X(\mathbb{C})$ in ${\rm C}_2$-spaces. 
\end{remark}

We next turn to the more simplicial viewpoint on equivariant and real realization of the quotient stack $[G\backslash X]$. From the case of special groups in Propositions~\ref{prop:equiv-realization-bar} and \ref{prop:krishna-equiv-realization}, we can now compute the real realization in general, using a ${\rm GL}_n$-presentation of the quotient stack $[G\backslash X]$. 

\begin{proposition}
  \label{prop:real-simplicial}
  Let $G\looparrowright X$ be a variety with action over  $\mathbb{R}$. Then the ${\rm C}_2$-equivariant realization of the quotient stack $[G\backslash X]$ is the (coarse-equivariant) homotopy orbit space for the ${\rm C}_2$-action $G(\mathbb{C})\looparrowright X(\mathbb{C})$. 

  For the real realization, we have equivalences
  \[
    {\rm Re}_{\mathbb{R}}([G\backslash X])\cong \left({\rm E}_{G(\mathbb{C})}(X(\mathbb{C})\right)^{{\rm hC}_2} \cong \mathcal{N}\left([G\backslash X](\mathbb{R})\right).
  \]
  This means that the real realization of the quotient stack $[G\backslash X]$ can be computed as homotopy fixed-points of the bar construction of $G(\mathbb{C})\looparrowright X(\mathbb{C})$, or equivalently as the realization of the nerve of the (topological) groupoid $[G\backslash X](\mathbb{R})$ of real points.   
\end{proposition}

\begin{proof}
  For the equivariant realization, assume first that $G$ is special. Consider the universal $({\rm C}_2,G(\mathbb{C}))$-space ${\rm E}G^{\rm C_2}$, cf. the discussion of equivariant principal bundles in Section~\ref{sec:coarse}. This space is coarse-contractible; in the case that $G$ is special, it is even fine-contractible (i.e., has contractible ${\rm C}_2$-fixed points). It follows that for $G$ special, fine and coarse homotopy orbits are equivalent (in the fine ${\rm C}_2$-equivariant homotopy theory). Therefore, our claim follows from Proposition~\ref{prop:equiv-realization-bar}, which implies that ${\rm Re}_{\rm C_2}({\rm E}_G(X))$ is the (fine) homotopy orbit space of the ${\rm C}_2$-action $G(\mathbb{C})\looparrowright X(\mathbb{C})$. 

  The general case for equivariant realization now follows from the case of special $G$. Choose a faithful representation $G\hookrightarrow{\rm GL}_n$, and consider the induced action ${\rm GL}_n\looparrowright\left({\rm GL}_n\times_{/G}X\right)$. By the induction equivalence, cf.{}~Proposition~\ref{prop:quotient-induction-stacks} (2),
  \begin{equation}
    {\rm Re}_{\rm C_2}\left(\left[{\rm GL}_n\backslash\left({\rm GL}_n\times_{/G}X\right)\right]\right)\simeq {\rm Re}_{\rm C_2}\left([G\backslash X]\right).
  \end{equation}
  By the arguments in the previous paragraph, the left-hand side is the coarse homotopy orbit space for ${\rm GL}_n(\mathbb{C})\looparrowright \left({\rm GL}_n(\mathbb{C})\times_{/G(\mathbb{C})}X(\mathbb{C})\right)$. Finally, by a version of the induction equivalence for ${\rm C}_2$-equivariant homotopy, as discussed in Section~\ref{sec:coarse}, this coarse homotopy orbit space is equivalent to the coarse homotopy orbit space for $G(\mathbb{C})\looparrowright X(\mathbb{C})$. 

  For the real realization, the case of special groups is addressed in Proposition~\ref{prop:krishna-equiv-realization}; we deduce the general case from that. As above, we consider the induced action ${\rm GL}_n\looparrowright\left({\rm GL}_n\times_{/G}X\right)$ for a faithful representation $G\hookrightarrow {\rm GL}_n$. We get an induced commutative diagram of realizations
  \[
  \xymatrix{
    \mathcal{N}\left([G\backslash X](\mathbb{R})\right) \ar[r] \ar[d] & {\rm Re}_{\mathbb{R}}\left([G\backslash X]\right) \ar[d] \\
    \mathcal{N}\left(\left[{\rm GL}_n\backslash\left({\rm GL}_n\times_{/G}X\right)\right](\mathbb{R})\right) \ar[r] & {\rm Re}_{\mathbb{R}}\left(\left[{\rm GL}_n\backslash \left({\rm GL}_n\times_{/G}X\right)\right]\right).
  }
  \]
  The horizontal morphisms are induced from the natural inclusions of real points. The left and right vertical morphisms are equivalences induced from the induction equivalence of Proposition~\ref{prop:quotient-induction-stacks} (2) and its equivariant version. The lower horizontal morphism is an equivalence from Proposition~\ref{prop:krishna-equiv-realization}.

  Finally, it remains to relate homotopy fixed points and coarse homotopy orbits. By the results in Sections~\ref{sec:real-pts-fp-gpd} and \ref{sec:coarse}, in particular Corollary~\ref{cor:equivariant-bg-fp} and Corollary~\ref{cor:main-hc2}, the geometric realization of the nerve $\mathcal{N}\left([G\backslash X](\mathbb{R})\right)$ of the groupoid of real points is also naturally identified with the homotopy fixed-points space of complex conjugation on the complex points $\mathcal{N}\left([G\backslash X](\mathbb{C})\right)$. This proves the remaining claims. 
\end{proof}

We can combine the above results to the following description of the real realization of a quotient stack:
  
\begin{corollary}
  Let $G\looparrowright X$ be a smooth quasi-projective variety with action over $\mathbb{R}$.
  \begin{enumerate}
  \item The natural morphism ${\rm E}_G(X)\to [G\backslash X]$ of Proposition~\ref{prop:simplicial-comparison} induces an equivalence after real realization:
    \[
    \mathcal{N}\left([G\backslash X](\mathbb{R})\right)\cong {\rm Re}_{\mathbb{R}}{\rm E}_G(X)^{\rm hC_2}\xrightarrow{\cong} {\rm Re}_{\mathbb{R}}[G\backslash X].
    \]
  \item The natural morphism $X_G(\rho)\to [G\backslash X]$ of Proposition~\ref{prop:approx-comparison} induces an equivalence after real realization:
    \[
    {\rm Re}_{\mathbb{R}}X_G(\rho)\xrightarrow{\cong}{\rm Re}_{\mathbb{R}}[G\backslash X].
    \]
  \item In case $G$ is special, the equivalences in (1) and (2) are the real realizations of Krishna's equivalence ${\rm E}_G(X)\cong X_G(\rho)$, cf. \cite{krishna}*{Proposition~3.2} or Sections~\ref{sec:stacks-motivic-simplicial} and \ref{sec:stacks-motivic-approx}. 
  \end{enumerate}
  In particular, the real realization of a quotient stack can be computed either using the realization of an algebraic Borel construction for the action $G\looparrowright X$, or as the nerve of the topological  groupoid $[G\backslash X](\mathbb{R})$.\footnote{As before, we emphasize that this is a topological groupoid, inherited from the analytic topologies on $G(\mathbb{C})$ and $X(\mathbb{C})$ via inclusions of fixed points.}
\end{corollary}

\begin{remark}
  We point out again that Krishna's equivalence ${\rm E}_G(X)\cong X_G(\rho)$ only works for groups $G$ which are special in the sense of Serre, i.e., where all \'etale-locally trivial torsors are already Nisnevich-locally trivial. For non-special groups, the natural map
  \[
  {\rm E}_G(X)\to {\rm E}_{{\rm GL}_n}({\rm GL}_n\times_{/G}X)\cong [{\rm GL}_n\backslash ({\rm GL}_n\times_{/G}X)]\cong[G\backslash X]
  \]
  will not be an equivalence in general. For example, if $X=\pt$ we have ${\rm E}_G(\pt)\cong {\rm B}_{\rm Nis} G$ while $[G\backslash\pt]\cong{\rm B}_\et G$. Similarly, the action groupoid $G(\mathbb{R})\times X(\mathbb{R})\rightrightarrows X(\mathbb{R})$ is generally not equivalent to the groupoid  $[G\backslash X](\mathbb{R})$ of real points of the quotient stack.
\end{remark}

\begin{remark}
  We would like to take this opportunity to point out that a lot of the statements and ideas in the identification of real realization in terms of homotopy fixed points is inspired by a paper of Acosta \cite{acosta}, even though homotopy fixed points don't figure in that paper. Still, in some ways, there is some homotopy fixed point space of real points of a quotient stack somewhere in the background of the constructions in \cite{acosta}. In particular, changing to strong real forms of the group $G$ relates the GIT-quotients to different components of the fixed-point groupoid of real points of the quotient stack $[G\backslash X]$. In this spirit, we wonder if there is a Kirwan-style surjective map from the cohomology of $[G\backslash X](\mathbb{R})$ (or from the Witt-sheaf cohomology of $[G\backslash X]$) to the cohomology of Richardson--Slodowy real GIT quotients. Part of this is probably related to the papers \cites{realDM1,realDM2} on real points of Deligne--Mumford stacks.
\end{remark}

\section{Fixed-point groupoids, Galois cohomology and strong real forms}
\label{sec:hc2-galois}

In this section, we will survey results describing the real points of classifying spaces in terms of Galois cohomology and strong real forms, which will be useful for our example computations in Section~\ref{sec:examples}. All we will say in this section can be found in the literature.
% , and appears to have been rediscovered a number of times in varying settings.
% reminiscent of Gian-Carlo Rota's quip about symmetric functions

We recall the relation between fixed-point groupoids and Galois cohomology in Section~\ref{sec:fp-gpoid-galois}, discuss the notion of strong real forms and its relevance for real points in Section~\ref{sec:strong-real-forms}, and finally state the main formula describing real points of classifying spaces in Section~\ref{sec:bg-formula}. 

To fix the setting, let $G$ be a linear algebraic group over $\mathbb{R}$; we will denote by $\sigma$ the corresponding complex conjugation on the group $G(\mathbb{C})$. Every Galois cohomology set ${\rm H}^1({\rm C}_2,G(\mathbb{C}))$ appearing below will always be with respect to the complex conjugation for the given real structure on $G$.

\subsection{Fixed-point groupoid vs Galois cohomology}
\label{sec:fp-gpoid-galois}

For the real group $G$, we want to consider the quotient stack $[G\backslash \pt]$. Recall from Example~\ref{ex:X/G_vs_action_groupoid} that the groupoid of complex points $[G\backslash\pt](\mathbb{C})$ has a single object given by the trivial $G_{\mathbb{C}}$-torsor on the point $\op{Spec}(\mathbb{C})$,  while the morphisms are given by $G(\mathbb{C})$. Complex conjugation $\sigma$ induces a ${\rm C}_2$-action on $[G\backslash\pt](\mathbb{C})$. 

We recall, e.g.{}~from Serre's book \cite{serre:galois-cohom}, the definition of Galois cohomology in the case at hand, for the group $G(\mathbb{C})$ with its ${\rm C}_2$-action via complex conjugation $\sigma$ (corresponding to the real structure we started with). The definition of the 1-cocycle space reduces to
\begin{align}
  {\rm Z}^1({\rm C}_2,G(\mathbb{C}))=\left\{g\in G(\mathbb{C})\mid g\sigma(g)=1 \right\},
\end{align}
and two such 1-cocycles $g,g'$ are cohomologous if there exists $h\in G(\mathbb{C})$ such that $g'=\sigma(h)gh^{-1}$. 

With this, it is now immediate to identify the components of the fixed-point groupoid with elements in the Galois cohomology set. 

\begin{proposition}
  \label{prop:fp-gpoid-galois}
  Let $G$ be a linear algebraic group over $\mathbb{R}$, with corresponding complex conjugation $\si$ on $G(\mathbb{C})$. Then there is a natural correspondence between the connected components of $[G\backslash\pt](\mathbb{C})^{\rm hC_2}$ and the Galois cohomology set ${\rm H}^1({\rm C}_2,G(\C))$.
\end{proposition}

\begin{proof}
  Recall the description of the fixed-point groupoid from Example~\ref{ex:hc2-orbits}. In the situation of complex points of the classifying stack $[G\backslash \pt](\mathbb{C})$, a homotopy fixed-point in $[G\backslash\pt](\mathbb{C})^{\rm hC_2}$ is an element $g\in G(\mathbb{C})$ with $g\sigma(g)=1$, i.e., the homotopy fixed-points are by definition the Galois 1-cocycles in ${\rm Z}^1({\rm C}_2,G(\mathbb{C}))$. Morphisms between such homotopy fixed-points $g,g'\in G(\mathbb{C})$ are elements $h\in G(\mathbb{C})$ with $\sigma(h)gh^{-1}=g'$, and by definition such morphisms exist if and only if the corresponding Galois 1-cocycles are cohomologous. 
\end{proof}

Using the identification of the fixed-point groupoid with the real points $[G\backslash\pt](\mathbb{R})$ from Corollary~\ref{cor:main-hc2}, this provides the Galois-cohomological description of the connected components of real points already mentioned in Example~\ref{ex:real-points}. 

The other piece of information in the fixed-point groupoid $[G\backslash\pt](\mathbb{C})^{\rm hC_2}$ are the fundamental groups of the individual components, i.e., the automorphism groups of objects. From the description in Example~\ref{ex:hc2-orbits}, we find that, for an element $g\in{\rm Z}^1({\rm C}_2,G(\mathbb{C}))$, its automorphism group in the fixed-point groupoid is the fixed group of the twisted complex conjugation $\si_g=\intt(g)\circ \si$:
\begin{align}
  \label{eq:automorphism-group}
  \Aut(g)=G(\C)^{\si_g}=\{h\in G\mid \si(h)g=gh\}.
\end{align}
Via its inclusion in $G(\mathbb{C})$, the automorphism group has a natural topology, making the fixed-point groupoid (or equivalently, the groupoid of real points) a topological groupoid.

In the next subsection, we will recall how the automorphism groups can be interpreted as strong real forms of $G(\mathbb{R})$.

\begin{remark}
  As is apparent from the above discussion, Galois cohomology will also play a role in the description of fixed-point groupoids for more general quotient stacks $[G\backslash X]$. Unwinding the definition of fixed-point groupoids, we get an identification of ${\rm H}^1({\rm C}_2,\op{Aut}(g))$ with the isomorphism classes of homotopy fixed point data on $g\in[G\backslash X](\mathbb{C})$. This is related to the Galois-cohomological characterization of real orbits for GIT quotients in \cite{acosta}*{Section~3.3.2}. 
\end{remark}
 
\subsection{Recollection on real and strong real forms}
\label{sec:strong-real-forms}

  We recall the notions of strong real forms  and their relation to Galois cohomology from \cite{adams:galois}*{Example 4.16, Chapter 6}, \cite{adams:taibi}. For emphasis of the varying involutions appearing in Galois cohomology groups, we use the notation ${\rm H}^1(\sigma,G)$ from loc.{}~cit., which means that we consider the Galois cohomology set of the involution $\sigma$ acting on $G$. We start by recalling terminology of forms from \cite{serre:galois-cohom}.

  \begin{definition}
    Given an algebraic group $G$ over a field $F$, a \emph{form} of $G$ is an algebraic group $G'$ over $F$ which becomes isomorphic to $G$ over the algebraic closure $\overline{F}$. These are classified by the Galois cohomology set ${\rm H}^1_\et(F,{\rm Aut}(G))$, and the various forms can be obtained from $G$ by twisting with an ${\rm Aut}(G)$-valued 1-cocycle. A form of $G$, corresponding to a 1-cocycle $\gamma\in{\rm H}^1_\et(F,{\rm Aut}(G))$, is called \emph{inner} if the cocycle is in the image of the natural map
  \begin{equation}
  {\rm H}^1_\et(F,G_{\rm ad})\to {\rm H}^1_\et(F,{\rm Aut}(G)),
  \end{equation}
  where $G_{\rm ad}=\Int(G)$ is the adjoint group, alternatively, the group of inner automorphisms of $G$.
  \end{definition}

  Following this, a \emph{real form} of a real linear algebraic group $G$ is another real linear algebraic group $G'$ such that $G_{\mathbb{C}}\cong G'_{\mathbb{C}}$ (as complex linear algebraic groups). 

  A real group $G$ gives rise to an antiholomorphic involution $\sigma$ on its complexification $G(\mathbb{C})$. One can then alternatively view real forms as ${\rm Aut}(G(\mathbb{C}))$-conjugacy classes of antiholomorphic involutions. This is the point of view taken in \cite{adams:taibi}: a \emph{real form} is an antiholomorphic involution $\si\in \Aut(G)$, and two real forms $\si,\si'$ are equivalent, if $\si'=\intt(g)\circ \si\circ \intt(g^{-1})$; in this case $G^{\si'}=gG^\si g^{-1}$. There is a short exact sequence of inner and outer automorphism groups:
\begin{equation}\label{eq:int_aut_out}
	\xymatrix@R-2pc{
		1\ar[r]&\Int(G)\ar[r]&\Aut(G)\ar[r]^-{p}&\Out(G)\ar[r]&1
	}	
\end{equation}
where $\Int(G)\iso G_{\rm ad}=G/{\rm Z}(G)$. An element $\si'\in \Aut(G)$ is said to be \emph{in the same inner class as $\si\in \Aut(G)$} (or, an \emph{inner form} of $\sigma$), if $p(\si)=p(\si')$, i.e. if $\si'=\intt(g)\circ \si$ for some $g\in G$. 

The set of equivalence classes of inner real forms of $\si$ are parametrized by ${\rm H}^1(\si,G_{\rm ad})$. There is an exact sequence
\begin{equation}\label{eq:Galois_exact}
	\xymatrix{
		{\rm H}^1(\si,G)\ar[r]^-{\rho}&{\rm H}^1(\si,G_{\rm ad})\ar[r]^-{\de}&{\rm H}^2(\si,{\rm Z}(G)),
	}
\end{equation}
in which $\rho$ is typically not surjective. The \emph{central invariant of a real form $\tau$ (relative to $\si$)} is $\de(\tau)\in {\rm H}^2(\si,Z)$. To describe ${\rm H}^1(\si,G)$ (as we will need for the components of fixed-point groupoids) as well as the morphism $\rho$ and its image, Adams and Ta\"ibi \cite{adams:taibi} use the notion of strong real forms. 

\begin{remark}\label{rmk:innerclassvariance}
  If $G$ has nontrivial center, then even though $\si$ and $\si'={\rm int}(g)\circ \si$ are involutions in the same inner class, their corresponding Galois cohomology can be different, i.e., different elements in ${\rm H}^1(\sigma,G_{\rm ad})$ might have preimages (under $\rho$ in \eqref{eq:Galois_exact}) of different cardinalities. For example, $(\si,{\rm SU}(2))$ and $(\si', {\rm SL}_2(\R))$ are two (inequivalent) real forms of ${\rm SL}_2(\C)$ in the same inner class; ${\rm H}^1(\si',G)$ has one element by Hilbert 90, while ${\rm H}^1(\si,G)$ has two elements, see \cite{adams:taibi}*{Example 4.9} and also Example \ref{ex:sl2} below. 
\end{remark}

Strong real forms enable discussing the exact sequences \eqref{eq:Galois_exact} simultaneously as $\si$ varies within its inner class: strong real forms enlarge the set ${\rm H}^1(\si,G)$, such that $\rho$ from \eqref{eq:Galois_exact} becomes surjective. In the discussion below, the relation to the fixed-point groupoid should be apparent.

\begin{lemma}\label{lemma:Galois_surjective}
	Let $\si\in \Aut(G)$ be an involution. Then ${\rm int}(g)\circ \si$ is an involution iff $g\sigma(g)\in Z(G)^\si$. All involutions in the same inner class as $\si$ are obtained in this manner, i.e.\ the map
	\[
\xymatrix@R-2pc{
	\{g\in G\mid g\sigma(g)\in {\rm Z}(G)\}\ar[r]^-{\rho}& \II_\si(G)\\
	g\ar@{|->}[r]& {\rm int}(g)\circ\si=(h\mapsto g\sigma(h)g^{-1})
}
\]
mapping $g\mapsto {\rm int}(g)\circ \si$ is surjective.
\end{lemma}
\begin{proof}
  First, ${\rm int}(g)\circ \si$ is an involution iff $g\sigma(g)\in Z(G)$:
  \[
    {\rm int}(g)\circ\si\circ	{\rm int}(g)\circ \si(x)
    =g\sigma(g)x(g\sigma(g))^{-1}.
  \]
  If $g\sigma(g)\in Z(G)$, $(g\sigma(g))g^{-1}=\sigma(g)$, so $g\sigma(g)\in Z(G)^\si$. The short exact sequence \eqref{eq:int_aut_out} proves the claim about the inner class.
\end{proof}

For a fixed real form $\si$, such elements $g\in G$ with $g\sigma(g)\in {\rm Z}(G)^\sigma$ are called \emph{strong involutions of $G$ in the inner class of $\si$}, with \emph{central invariant $z=\inv(g)=g\sigma(g)\in {\rm Z}(G)^\si$}. Two strong involutions are equivalent if they are conjugate under $G\rtimes_\si {\rm C}_2$, i.e.\ $g\sim hg\si(h^{-1})$, and then they have the same central invariant. 

A \emph{strong real form} is a conjugacy class of strong involutions; their set is denoted $\mathcal{S}_{\si}(G)$ and those with central invariant $z$ are denoted $\mathcal{S}_{\si, z}(G)$. Strong real forms refine the notion of a real form: by Lemma \ref{lemma:Galois_surjective}, the map $\mathcal{S}_\si(G)\to \I_\si(G)\iso {\rm H}^1(\si,G_{\rm ad})$ maps strong real forms in the inner class of $\si$ surjectively to real forms in the inner class of $\si$. The set $\mathcal{S}_{\sigma,[z]}(G)$ is mapped to the set of inner forms whose image under $\delta\colon {\rm H}^1(\sigma,G_{\rm ad})\to{\rm H}^2(\sigma,{\rm Z}(G))$ corresponds to $z$ under the Tate duality isomorphism ${\rm H}^2(\sigma,{\rm Z}(G))\cong {\rm Z}(G)^\sigma/(1+\sigma){\rm Z}(G)$. It follows that the Galois cohomology set ${\rm H}^1(\si,G)$ is in bijection with the set $\mathcal{S}_{\si,1}(G)$ of strong real forms of central invariant 1, and the image of $\rho$ consists of inner forms with trivial central invariant (naturally, by the exactness of \eqref{eq:Galois_exact}).

Summing up, there are various interpretations for the automorphism groups $G(\mathbb{C})^{\sigma_g}$ in \eqref{eq:automorphism-group}: we can view them as strong real forms of $G(\mathbb{C})$ with central invariant 1 (relative to the given complex conjugation on $G(\mathbb{C})$) -- this is the interpretation in the context of \cite{adams:taibi}; alternatively, they are the real points of automorphism groups of $G$-torsors over ${\rm Spec}(\mathbb{R})$ -- this is the interpretation in the Galois cohomology context where ${\rm H}^1(\sigma,G)\cong{\rm H}^1_{\et}(\op{Spec}(\mathbb{R}),G)$ classifies $G$-torsors.

We end the section by discussing the related notion of strong inner forms,  terminology used by Tits in \cite{tits}*{p.~653}. This will be relevant for our discussion of invariance properties of classifying stacks in Section~\ref{sec:bitorsors}. 

\begin{definition}
  \label{def:strong-inner-forms}
  Let $G$ be an algebraic group over a field $F$. A form of $G$, corresponding to a cocycle $\gamma\in{\rm H}^1_\et(F,{\rm Aut}(G))$, is called \emph{strongly inner} if the cocycle is in the image of the natural map
  \begin{equation}
    \label{eq:strong-underlying}
  {\rm H}^1_\et(F,G)\to {\rm H}^1_\et(F,{\rm Aut}(G)).
  \end{equation}
\end{definition}

\begin{remark}
  \label{rem:comparison-strong-inner}
  To compare this to the notion of strong real form with given central invariant of \cite{adams:taibi}, note that the map ${\rm H}^1_\et(F,G)\to {\rm H}^1_\et(F,{\rm Aut}(G))$ (which is the composition of $\rho$ in \eqref{eq:Galois_exact} and the inclusion of inner automorphisms in \eqref{eq:int_aut_out}) is not necessarily injective. In particular, strong real forms can be isomorphic as algebraic groups but different as strong real forms. In any case, for a real algebraic group $G$, a strong real form of $G$ with the same invariant as $G$ (in the sense of \cite{adams:taibi}) will in particular be a strongly inner form (in the sense of \cite{tits}, or Definition~\ref{def:strong-inner-forms} above).\footnote{Bezrukavnikov and Vilonen \cite{bezrukavnikov:vilonen} call the image of a strong real form under the morphism \eqref{eq:strong-underlying}  the \emph{underlying real form of a strong real form}.} The notion of strongly inner form of \cite{tits} as above is a property of a form of an algebraic group, whereas the notion of strong real form of \cite{adams:taibi} includes additional data. 
\end{remark}

\subsection{Description of homotopy fixed points}
\label{sec:bg-formula}

In this section, we finally state the explicit formula for the groupoid of real points of a classifying stack $[G\backslash \pt]$ for a real linear group $G$. This formula allows to easily compute the real realization of the \'etale classifying space ${\rm B}_\et G$, some example computations can be found in Section~\ref{sec:examples}.

\begin{proposition}
  \label{prop:bg-formula} 
  Let $G$ be a linear algebraic group over $\mathbb{R}$, with corresponding complex conjugation $\si$ on $G(\mathbb{C})$. Then there is an equivalence
  \begin{equation}\label{hom_fix_pt_of_BG}
    {\rm B}G(\mathbb{C})^{{\rm hC_2}}\simeq\mathcal{N}([G\backslash \pt](\mathbb{R}))\simeq\bigsqcup_{[g]\in{\rm H}^1({\rm C_2},G)}{\rm B}(G(\C)^{\si_g}).
  \end{equation}
  On the right-hand side, $\sigma_g=\op{int}(g)\circ\sigma$ denotes the corresponding twisted complex conjugation on $G(\mathbb{C})$, for a Galois cocycle $[g]\in{\rm H}^1({\rm C}_2,G(\mathbb{C}))$, as discussed in Section~\ref{sec:fp-gpoid-galois} above. 
\end{proposition}

\begin{proof}
  By the results in Section~\ref{sec:real-pts-fp-gpd}, specifically Corollary~\ref{cor:main-hc2}, we have an identification of $[G\backslash\pt](\mathbb{R})$ as the fixed-point groupoid $[G\backslash\pt](\mathbb{C})^{\rm hC_2}$ (of the complex conjugation $\sigma$). The discussion in Section~\ref{sec:fp-gpd-hfp} then implies the identification of homotopy fixed points ${\rm B}G(\mathbb{C})^{{\rm hC}_2}$ as the realization of $[G\backslash\pt](\mathbb{R})$. The other identification, the main point of the proposition, follows from the explicit description of the fixed-point groupoid $[G\backslash\pt](\mathbb{R})$ in Proposition~\ref{prop:fp-gpoid-galois}: the components of the fixed-point groupoid are in bijection with the Galois cohomology set ${\rm H}^1({\rm C}_2,G(\mathbb{C}))$, and the automorphism groups of the components are given by the corresponding strong real forms, i.e., the fixed groups of the twisted complex conjugations on $G(\mathbb{C})$. The description in \eqref{hom_fix_pt_of_BG} above follows directly from that.
\end{proof}

\begin{remark}
  As mentioned several times, this formula has of course appeared before, in varying forms. It appears for example in \cite{virk:hc2}, although the discussion there is focused on discrete $G$. It appears in \cite{GuillouMayMerling2017}, although their focus is on the case where is $G$ discrete or a compact Lie group.\footnote{It should also be noted that the first cohomology set appearing in \cite{GuillouMayMerling2017} is actually an equivariant-topological version, closer to the coarse cohomology set ${\rm H}^1_{\rm sc}(\pt,G)$ that we have used in Section~\ref{sec:coarse}. For a real reductive group $G$, it can naturally be identified with Galois cohomology, essentially since the definition of 1-cocycles doesn't make use of the topology.} For finite groups, it also appears in \cites{realDM1,realDM2}. And probably in a number of other places, too.
\end{remark}

\begin{remark}
  In Proposition~\ref{prop:bg-formula}, the set ${\rm H}^1({\rm C}_2,G(\mathbb C))$ denotes the set of isomorphism classes of $G$-torsors over $\mathbb R$. For a cocycle $[g]$ representing a $G$-torsor $P$ over $\mathbb{R}$, the automorphism group $G(\mathbb{C})^{\sigma_g}$ in the groupoid of real points is literally the automorphism group $\mathrm{Aut}_G(P)$ of $G$-equivariant automorphisms of $P$, those are strong inner forms of $G$. In particular, we can interpret the formula~\eqref{hom_fix_pt_of_BG} as saying that the real realization of a classifying space is equivalent to the disjoint union of classifying spaces of automorphism groups of real $G$-torsors.

  On the other hand, we will see in Theorem~\ref{thm:strong-inner--eq_formulations} that for any strong inner form $H$ of $G$ we can construct a $(G,H)$-bitorsor $P$, and thus $H\cong \mathrm{Aut}_G(P)$. Our discussion of strong inner forms in Section~\ref{sec:bitorsors} will also explain the apparent invariance of the formula in \eqref{hom_fix_pt_of_BG} under strong inner forms. We can then interpret formula~\eqref{hom_fix_pt_of_BG} as saying that the real realization of a classifying space is a disjoint union of the classifying spaces of all strong real forms of $G$, where the strong real forms appear "with multiplicity" in the sense that different $G$-torsors can have automorphism groups which are isomorphic (as groups, but not as torsors). We will discuss particular cases of this, e.g.{}~in the cases of orthogonal groups, in Section~\ref{sec:examples}. 
\end{remark}

\section{Equivariant real cycle class map}
\label{sec:equivariant-cycle-class}

In this section, we will now define the equivariant real cycle class map, using approximations of the Borel-construction for a group action $G\looparrowright X$. We also establish compatibility of the real cycle class map with pullbacks, pushforwards and intersection products. 

\subsection{Equivariant cohomology for homotopy modules}
\label{sec:equivariant-cohomology}

We begin with a brief recollection on equivariant cohomology, as defined and discussed by di~Lorenzo and Mantovani in \cite{dilorenzo:mantovani}*{Section~2.2}. Their setting includes many theories relevant for our purposes, such as equivariant Witt-sheaf cohomology, equivariant Chow--Witt groups and equivariant mod 2 Milnor K-cohomology.

Recall from Section~\ref{sec:stacks-motivic-approx} that we can use admissible gadgets $\rho=(V_i,U_i)_{i\geq 1}$ to obtain approximations of a quotient stack $[G\backslash X]$ via the quotients $X_G^i(\rho)=X\times_{/G}U_i$. With these spaces, we can then define equivariant Borel--Moore homology and equivariant cohomology, following \cite{dilorenzo:mantovani}*{Definitions~2.2.6, 2.2.7}. 

\begin{definition}
  Let $G\looparrowright X$ be a variety with action over a field $F$, satisfying the assumptions of our Conventions~\ref{conventions-actions}. Let $M_*$ be a homotopy module and let $\mathscr{L}\in{\rm Ch}^1_G(X)$ be a $G$-equivariant line bundle. Let $\rho$ be an admissible gadget for $G$, and let $(V,U)$ be one of the equivariant scheme approximations in $\rho$. 

  Then the \emph{equivariant Borel--Moore homology with $M_*$-coefficients} is defined as 
  \[
  A_i^G(X,M_*,\mathscr{L}):=A_{i+l-g}(X\times_{/G}U,M_*,\omega_{X\times_{/G}U}\otimes\mathscr{L})
  \]
  where $l=\dim(V)$, $g=\dim G$, and $\op{codim}(V\setminus U)\leq \dim X-i+1$. The tautological bundle $\omega_{X\times_{/G}U}$ is (the dual of) the determinant of the relative cotangent complex of $X\times_{/G}U/[G\backslash X]$, induced by $V^\vee$ on $X\times U$, cf. the discussion in \cite{dilorenzo:mantovani}*{Remark~2.2.8}.

  Similarly, if $G\looparrowright X$ is smooth, then \emph{equivariant cohomology with $M_*$-coefficients} is defined as
  \[
  A^i_G(X,M_*,\mathscr{L}):=A^i(X\times_{/G}U,M_*,\mathscr{L})
  \]
  where again $\op{codim}(V\setminus U)\leq \dim X-i+1$.
\end{definition}

\begin{remark}
  \begin{itemize}
  \item 
    There are also versions with supports but we won't need those. For now. Similarly, the twisting works more generally for virtual vector bundles $e\in\underline{K}([G\backslash X])$, but for now the line bundle situation is sufficient for our needs. For vector bundles on stacks resp. equivariant vector bundles, cf. Section~\ref{sec:equivariant-vb}.
  \item For the homotopy module $M_*=({\bf I}^q)_{q\in\mathbb{Z}}$, we will use the following notation ${\rm H}_p^G(X,{\bf I}^q(\mathscr{L}))$ and ${\rm H}^p_G(X,{\bf I}^q(\mathscr{L}))$ instead of the notation from \cite{dilorenzo:mantovani}.
  \item For more precise information on how to use twists by line bundles, and how $G$-equivariant line bundles induce line bundles on the smooth scheme approximations, cf.{}~\cite{dilorenzo:mantovani}*{Section~2.2.3 and Remark~2.2.8}. 
  \end{itemize}
\end{remark}

In the following, we will list some of the properties of equivariant (co)homology which will be most relevant for us. The information below is essentially contained in \cite{dilorenzo:mantovani}*{Theorem~2.2.12}. Some of the properties are specialized to the ${\bf I}$-cohomology setting most relevant for us; some of the structure can also be formulated in greater generality for morphisms of (quotient) stacks as in \cite{dilorenzo:mantovani}*{Theorem~2.3.3}. 

\begin{description}
\item[pullbacks] Equivariant $M_*$-homology has pullbacks along $G$-equivariant maps $f\colon X\to Y$ which are lci, i.e., compositions of smooth morphisms and regular embeddings. The pullback maps are of the form
  \[
  f^*\colon A^G_i(X,M_*,\mathscr{L})\to A_{i+d}^G(Y,M_*,f^*\mathscr{L}\otimes\omega_f)
  \]
  where $d$ is the relative dimension. Note that this means a positive $d$ for smooth maps of relative dimension $d$ and negative $d$ for a regular immersion of codimension $-d$. In case $X$ and $Y$ are smooth, there is also a version for cohomology, without any shifts or twists by relative canonical bundles:
  \[
  f^*\colon A_G^i(X,M_*,\mathscr{L})\to A^{i}_G(Y,M_*,f^*\mathscr{L})
  \]
  These pullback morphisms in equivariant (co)homology are functorial
  \begin{itemize}
  \item in $f$, i.e., we have a functor from $G$-varieties to abelian groups,
  \item in the coefficients $M_*$, i.e., pullbacks are compatible with change-of-coefficients maps between (co)homology groups,
  \item in the line bundle $\mathscr{L}$, i.e., pullbacks are compatible with the maps induced by isomorphisms $\mathscr{L}\cong\mathscr{L}'$ of (graded) line bundles.
    \end{itemize}

  More generally, there are pullback morphisms along representable lci morphisms of quotient stacks $[H\backslash Y]\to [G\backslash X]$.
  These have the form
  \[
  f^*\colon A_G^i(X,M_*,\mathscr{L})\to A^{i}_H(Y,M_*,f^*\mathscr{L})
  \]
  As particular examples, these provide restrictions to subgroups and forgetful morphisms from $G$-equivariant to non-equivariant cohomology.\footnote{Actually, one might expect these to exist \emph{in cohomology} for arbitrary morphisms of varieties with action, basically using compatible resolutions as in \cite{bernstein:lunts}*{Section~6} also used in \cite{BGK} to define general pullbacks for equivariant motives. In homology, this only works for representable morphisms, due to the shifts and twists involved.} 
\item[intersection products] For a smooth variety with action  $G\looparrowright X$, equivariant ${\bf I}$-cohomology has an intersection product
  \[
    {\rm H}^{p_1}_G(X,{\bf I}^{q_1}(\mathscr{L}_1))\otimes {\rm H}^{p_2}_G(X,{\bf I}^{q_2}(\mathscr{L}_2))\to {\rm H}^{p_1+p_2}_G(X,{\bf I}^{q_1+q_2}(\mathscr{L}_1\boxtimes\mathscr{L}_2)),
  \]
  where $\mathscr{L}_1,\mathscr{L}_2\in{\rm Ch}^1_G(X)$ are $G$-equivariant line bundles on $X$. As in \cite{dilorenzo:mantovani}, these intersection products can be viewed as providing an algebra structure over the monoidal functor ${\bf K}^{\rm MW}_*(F,-)$ on the category of graded line bundles over the base field $F$. The intersection products are compatible with pullbacks along equivariant maps between smooth schemes. 
\item[pushforwards] Equivariant $M_*$-homology has pushforwards along proper  $G$-equivariant maps $f\colon X\to  Y$. The pushforward maps in equivariant $M_*$-homology have the form
  \[
  f_*\colon A^G_i(X,M_*,f^*\mathscr{L})\to A_i^G(Y,M_*,\mathscr{L}).
  \]
  In case $X$ and $Y$ are smooth, there is also a version in cohomology, involving additional shifts by the relative dimension $d$ and twists by the relative canonical bundle $\omega_f$ (or better, the [determinant of the] cotangent complex of the morphism of stacks)
  \[
  f_*\colon A_G^i(X,M_*,f^*\mathscr{L}\otimes\omega_f)\to A^{i+d}_G(Y,M_*,\mathscr{L}).
  \]
  As for the pullbacks, the pushforwards are functorial in $f$, $M_*$ and $\mathscr{L}$, in the sense discussed before.

  More generally, there are pushforward maps along representable proper morphisms of quotient stacks $[H\backslash Y]\to [G\backslash X]$, which would have the form 
  \[
  f_*\colon A_H^i(Y,M_*,f^*\mathscr{L}\otimes\omega_f)\to A^{i+d}_G(X,M_*,\mathscr{L}).
  \]
\item[localization sequence] Let $X$ be a smooth scheme and let $\mathscr{L}\in{\rm Ch}^1(X)$ be a $G$-equivariant line bundle. For an equivariant closed embedding $\iota\colon Z\hookrightarrow X$ of a smooth closed subscheme of codimension $c$, with open complement $j\colon U=X\setminus Z\hookrightarrow X$ and normal bundle $\mathscr{N}=\omega_\iota\in{\rm Ch}^1_G(Z)$, there is a long exact sequence
  \[
  \cdots\to{\rm H}_G^p(X,{\bf I}^q(\mathscr{L}))\xrightarrow{j^\ast}{\rm H}^p_G(U,{\bf I}^q(\mathscr{L}))\xrightarrow{\partial}{\rm H}^{p-c+1}_G(Z,{\bf I}^{q-c}(\mathscr{L}\otimes\mathscr{N}))\xrightarrow{\iota_\ast}\cdots
  \]
\end{description}

A number of additional properties, like homotopy invariance, base-change, projection formulas etc.~ can be found in \cite{dilorenzo:mantovani}*{Theorem~2.2.12}. 

We list two further well-known equivalences which are often very useful in computations, for motivic versions of these, cf. \cite{BGK}*{Chapter~I.7}. 

\begin{proposition}[quotient equivalence]
  \label{prop:quotient-equiv}
  Let $G\looparrowright X$ be a variety with action and let $N\subset G$ be a closed normal subgroup such that the restricted action $N\looparrowright X$ is free. Then the isomorphism $[G\backslash X]\cong \left[\left(G/N\right)\backslash\left(N\backslash X\right)\right]$ of Proposition~\ref{prop:quotient-induction-stacks} (1) induces an isomorphism of equivariant homology groups
  \[
  {\rm H}_*^{G/N}(N\backslash X,M_*,\mathscr{L})\xrightarrow{\cong} {\rm H}_*^{G}(X,M_*,p^*\mathscr{L})
  \]
  for any $G/N$-equivariant line bundle $\mathscr{L}\in{\rm Ch}^1_{G/N}(N\backslash X)$. If $X$ is smooth, then the stacks involved are smooth, and consequently there is also a multiplicative isomorphism in equivariant cohomology:\footnote{Multiplicativity has the following forms: On the untwisted cohomology, this will be a ring isomorphism. Cohomology with line bundle twists will be a module over the untwisted cohomology ring. One can also consider a total cohomology ring, involving all possible twists, graded over the mod 2 Picard group. The usual caveats on choices of representatives of classes in the mod 2 Picard group detailed in Balmer--Calm\`es \cite{balmer:calmes} apply.}
  \[
  {\rm H}^*_{G/N}(N\backslash X,M_*,\mathscr{L})\xrightarrow{\cong} {\rm H}^*_{G}(X,M_*,p^*\mathscr{L})
  \]
  These isomorphisms are functorial and compatible with pullbacks.
\end{proposition}

\begin{proposition}[induction equivalence]
  \label{prop:induction-equiv}
  Let $i\colon H\hookrightarrow G$ be the inclusion of a smooth closed subgroup into a smooth affine algebraic group $G$ and let $H\looparrowright X$ be a variety with action such that the quotient $G\times_{/H}X$ of the diagonal $H$-action on $G\times X$ exists as a smooth quasi-projective variety. Then the isomorphism of quotient stacks $[H\backslash X]\cong [G\backslash (G\times_{/H}X)]$ of Proposition~\ref{prop:quotient-induction-stacks} (2) induces an isomorphism of equivariant homology groups
  \[
  {\rm H}_*^{G}(G\times_{/H}X,M_*,\mathscr{L})\xrightarrow{\cong} {\rm H}_*^{H}(X,M_*,p^*\mathscr{L})
  \]
  for any $G$-equivariant line bundle $\mathscr{L}\in{\rm Ch}^1_G(G\times_{/H}X)$. If $X$ is smooth, then the stacks involved are smooth, and consequently there is also a multiplicative isomorphism in equivariant cohomology:
  \[
  {\rm H}^*_{G}(G\times_{/H}X,M_*,\mathscr{L})\xrightarrow{\cong} {\rm H}^*_{H}(X,M_*,p^*\mathscr{L})
  \]
  These isomorphisms are functorial and compatible with pullbacks. 
\end{proposition}

\subsection{Real cycle class map}

The natural way to relate the Witt-sheaf or ${\bf I}$-cohomology of a real scheme $X/\mathbb{R}$ with the singular cohomology of its space of real points $X(\mathbb{R})$ is Jacobson's real cycle class map \cite{jacobson}. The key observation is that the colimit of multiplication by $\langle\!\langle-1\rangle\!\rangle$ can be identified with $a_{\ret}\mathbb{Z}$, the real-\'etale constant sheaf $\mathbb{Z}$. We can view the real cycle class map ${\rm H}^p_{\rm Zar}(X,{\bf I}^q)\to{\rm H}^p_{\rm sing}(X(\mathbb{R}),\mathbb{Z})$ as induced by the signature map ${\rm sgn}\colon{\bf I}^q\to a_{\ret}\mathbb{Z}$. The real cycle class map is compatible with pullbacks, pushforwards (along closed immersions) and intersection products, cf.{}~\cite{4real}. 

We now want to describe an equivariant version of the real cycle class map. The basic idea of the construction is to use the ordinary real cycle class map on smooth scheme approximations to the quotient stack, and check that this provides a well-defined morphism.

\begin{theorem}
  \label{thm:real-cycle-class}
  Let $G\looparrowright X$ be a smooth scheme with action by a linear algebraic group over $\mathbb{R}$, and let $\mathscr{L}\in{\rm CH}^1_G(X)$ be a $G$-equivariant line bundle on $X$. Then the signature map ${\rm sgn}\colon{\bf I}^q\to a_{\ret}\mathbb{Z}$ induces a well-defined equivariant real cycle class map
  \[
  {\rm cyc}_{\mathbb{R}}\colon {\rm H}^p_G(X,\mathbf{I}^q(\mathscr{L}))={\rm H}^p([G\backslash X],\mathbf{I}^q(\mathscr{L}))\to {\rm H}^p([G\backslash X](\mathbb{R}),\mathbb{Z}(\mathscr{L}))
  \]
  which satisfies the following properties:
  \begin{enumerate}
  \item The equivariant real cycle class map only depends on the stack $[G\backslash X]$.
  \item The case of the trivial group $G=\{1\}$ recovers the non-equivariant real cycle class map of Jacobson~\cite{jacobson}, cf. also \cite{4real}. 
  \item The equivariant real cycle class map is an isomorphism for big enough $q$.
  \item The equivariant real cycle class map is an isomorphism after inverting 2.
  \end{enumerate}
\end{theorem}

\begin{proof}
  Fix an admissible gadget $\rho=(V_j,U_j)_{j\geq 1}$ for $G$, cf.{}~Definition~\ref{def:admissible-gadget}. Recall from Section~\ref{sec:equivariant-cohomology} that the $p$-th equivariant cohomology group for $G\looparrowright X$ can be defined as follows. Let $(V,U)$ be an approximation pair in $\rho$ where $\op{codim}(V\setminus U)\geq p+2$. The quotient $U\times_{/G}X$ is a smooth scheme and the $G$-equivariant line bundle $\mathscr{L}$ induces a well-defined line bundle $\overline{\mathscr{L}}=(\mathscr{L}\times U\times X)/G$ on $U\times_{/G}X$. In this situation, the $\overline{\mathscr{L}}$-twisted cohomology of $U\times_{/G}X$ is independent of $V$ and $S$ and we have an isomorphism
  \[
  {\rm H}^p_G(X,{\bf I}^q(\mathscr{L}))\cong {\rm H}^p(U\times_{/G}X,{\bf I}^q(\overline{\mathscr{L}}))
  \]
  induced by the morphism $f\colon U\times_{/G}X\to [G\backslash X]$.

  Now we consider the real realization side. Recall from Proposition~\ref{prop:real-admissible-gadget} that ${\rm Re}_{\mathbb{R}}([G\backslash X])\simeq \op{colim}_{j\geq 1}(U_j\times_{/G} X)(\mathbb{R})$ for an admissible gadget $\rho=(U_j,V_j)_{j \geq 1}$. We also have stabilization results on the topological side: if $\op{codim}(V_j\setminus U_j)\geq p+2$, then we get an isomorphism
\[
{\rm H}^p((U_j\times_{/G}X)(\mathbb{R}),\mathbb{Z}(\overline{\mathscr{L}}))\to {\rm H}^p([G\backslash X](\mathbb{R}),\mathbb{Z}(\mathscr{L}))
\]
which is the inverse of the map induced by $f_{\mathbb{R}}\colon (U\times_{/G}X)(\mathbb{R})\to [G\backslash X](\mathbb{R})$. In particular, as long as $\op{codim}(V_j\setminus U_j)\geq p+2$ then $U_j\times_{/G}X$ can be used to compute equivariant cohomology.

This can be proved exactly as in the algebraic case, e.g.{}~\cite{totaro:bg} or \cite{dilorenzo:mantovani}*{Proposition~2.2.10}, with the following modifications. We use that $\op{codim}(V\setminus U)\geq p+2$ implies that the real codimension of $(X\times(V\setminus U))(\mathbb{R})$ in $(X\times V)(\mathbb{R})$ is also $\geq p+2$. Moreover, if we have $G$-stable $U\subseteq U'\subseteq V$ such that the $G$-action on $U'$ is free, then the complement of $(X\times_{/G} U)(\mathbb{R})$ in $(X\times_{/G}U')(\mathbb{R})$ also has codimension $\geq p+2$. In this situation, the inclusion $(X\times_{/G}U)(\mathbb{R})\hookrightarrow(X\times_{/G}U')(\mathbb{R})$ induces an isomorphism in singular cohomology in degree $p$.

The algebraic way to see this is to identify singular cohomology as the Zariski cohomology of $\op{colim}{\bf I}^n$ as in \cite{jacobson} and use the algebraic excision results. The topological way to see this is to note that for a quasi-projective real variety $X$ with a closed subvariety $Z$, there is a finite-dimensional CW-structure on $X(\mathbb{R})$ such that $Z(\mathbb{R})$ is a subcomplex; moreover, the cells of the subcomplex will have dimension bounded by the dimension of $Z$. In particular, if $X$ is smooth equi-dimensional (so that codimension makes sense), all the cells in $Z(\mathbb{R})$ will have codimension at least the codimension of $Z$ in $X$. We then get an excision result for singular cohomology by comparison to cellular cohomology. One possible reference for the triangulation is the Real Triangulation Theorem formulated in \cite{karoubi:weibel} (and also references in there).\footnote{Likely there are other references, in particular in the semi-algebraic geometry literature or in the direction of Goresky--MacPherson stratified Morse theory.} 

  After this preparation,  we can now define the equivariant real cycle class map as the following composition
  \[
  {\rm H}^p_G(X,{\bf I}^q(\mathscr{L})) \xrightarrow{\cong} {\rm H}^p(U\times_{/G}X,{\bf I}^q(\overline{\mathscr{L}}))\xrightarrow{{\rm cyc}_{\mathbb{R}}} {\rm H}^p((U\times_{/G}X)(\mathbb{R}),\mathbb{Z}(\overline{\mathscr{L}}))\xrightarrow{\cong}{\rm H}^p([G\backslash X](\mathbb{R}),\mathbb{Z}(\mathscr{L}))
  \]
  where the map in the middle is the real cycle class map of Jacobson, cf.{}~\cite{jacobson}. 
  
  In the next step, we need to check that this is independent of choices. This is covered by the commutativity of the next diagram which follows from compatibilities of the real cycle class map with pullbacks, cf.{}~\cite{4real}*{Proposition~4.1}:
  \[
  \xymatrix{
    {\rm H}^p(U_{i+1}\times_{/G}X,{\bf I}^q(\overline{\mathscr{L}}))\ar[r]^-{{\rm cyc}_{\mathbb{R}}} \ar[d]  & {\rm H}^p((U_{i+1}\times_{/G}X)(\mathbb{R}),\mathbb{Z}(\overline{\mathscr{L}})) \ar[d] \\ 
    {\rm H}^p(U_i\times_{/G}X,{\bf I}^q(\overline{\mathscr{L}}))\ar[r]_-{{\rm cyc}_{\mathbb{R}}}  & {\rm H}^p((U_i\times_{/G}X)(\mathbb{R}),\mathbb{Z}(\overline{\mathscr{L}})) 
  }
  \]
  Essentially, this means that we get a morphism ${\rm cyc}_{\mathbb{R}}$ between the limit diagrams
  \[
  \left({\rm H}^p(U_i\times_{/G}X,{\bf I}^q(\overline{\mathscr{L}}))\right)_i \qquad\textrm{ and }\qquad \left({\rm H}^p((U_i\times_{/G}X)(\mathbb{R}),\mathbb{Z}(\overline{\mathscr{L}}))\right)_i
  \]
  whose limits are ${\rm H}^p_G(X,{\bf I}^q(\mathscr{L}))$ and ${\rm H}^p([G\backslash X](\mathbb{R}),\mathbb{Z}(\mathscr{L}))$, respectively. The equivariant real cycle class map is the limit of this morphism of diagrams.

  Now that we have established the existence of a well-defined equivariant real cycle class map, we can prove all the claims. For claim (1), the independence of equivariant cohomology is proven in \cite{dilorenzo:mantovani}*{Proposition~2.2.10}. The same argument, combined with the compatibility of real cycle class maps with pullbacks from \cite{4real}*{Proposition~4.1} shows that the equivariant real cycle class map is also independent of the choices of approximations, i.e., it only depends on the quotient stack $[G\backslash X]$. Claim (2) also follows immediately, since for the trivial group, we get $U\times_{/G}X=X$, and then the definition of equivariant real cycle class map reduces to the non-equivariant one. Finally, the two claims (3) and (4) follow from Corollaries 8.9 and 8.11 of \cite{jacobson}. For a discussion of twisted versions of Jacobson's cycle class map and these corollaries, cf. the discussion in \cite{4real} and in particular Theorem 3.12. 
\end{proof}

\begin{remark}
  \label{rem:big-enough}
  Which $q$ is big enough depends on the group. If a group $G$ has a representation $V$ such that $G$ acts freely outside a $G$-stable closed subset $S\subseteq V$ of codimension $\geq p+2$, then there is an isomorphism
  \[
  {\rm H}^p((V\backslash S)/G,\mathbf{I}^q)\cong {\rm H}^p({\rm B}_{\et}G,\mathbf{I}^q),
  \]
  see stabilization results for ${\bf I}$-cohomology in \cite{chowwitt}*{Proposition~3.1} or \cite{dilorenzo:mantovani}*{Proposition~2.2.10} for a more general statement. In this case, the real cycle class map will be an isomorphism for $q>\dim V-\dim G$. The case $[G\backslash X]$ requires an additional summand $\dim X$ on the right-hand side. A more daring conjecture would posit that it suffices to take $q>p+d$, where $d$ is the dimension of a faithful $G$-representation. 
\end{remark}

\begin{remark}
  \begin{itemize}
  \item 
    We note that the equivariant real cycle class map is also uniquely determined by conditions (1) and (2) in Theorem~\ref{thm:real-cycle-class}. Essentially, we can write $[G\backslash X]$ as a colimit of smooth schemes using an admissible gadget as in Section~\ref{sec:stacks-motivic-approx}, and then the equivariant real cycle class map is left-Kan extended from the non-equivariant one.
  \item
    There is also a version of the equivariant real cycle class map for equivariant (Borel--Moore) homology which has the following form
    \[
    {\rm cyc}_{\mathbb{R}}\colon {\rm H}_p([G\backslash X],{\bf I}^q(\mathscr{L}))\to {\rm H}_p^{\rm BM}([G\backslash X](\mathbb{R}),\mathbb{Z}(\mathscr{L}))
    \]
    This applies more generally to singular schemes with action $G\looparrowright X$. The argument is much the same as for Theorem~\ref{thm:real-cycle-class}, but using compatibility of real cycle class maps with the proper pushforwards along the closed immersions $X\times_{/G}U_i\hookrightarrow X\times_{/G}U_{i+1}$ from the admissible gadget. This requires adjusting a couple of arguments from \cites{jacobson,4real}. 
  \end{itemize}
\end{remark}

\subsection{Compatibilities}

In this section, we briefly formulate various compatibilities of the equivariant real cycle class map with the natural structures in equivariant cohomology discussed in Section~\ref{sec:equivariant-cohomology}: lci pullbacks, proper pushforwards (along closed immersions) and intersection products. We also discuss the localization sequence as well as quotient and induction equivalences. As a consequence, the equivariant real cycle class map is compatible with all the tools necessary for computations using the stratification method as employed e.g.{}~in the computations of Chow rings of classifying spaces, cf.{}~\cite{RojasVistoli}.

\begin{proposition}
  \label{prop:compat-pullback}
  Let $(\varphi,f)\colon (H\looparrowright Y)\to(G\looparrowright X)$ be a morphisms of smooth schemes with actions over $\mathbb{R}$, and let $\mathscr{L}\in{\rm Ch}^1_G(X)$ be a $G$-equivariant line bundle on $X$. Then the real cycle class map is compatible with the natural pullback morphisms in equivariant cohomology, i.e., the following diagram commutes:
  \[
  \xymatrix{
    {\rm H}^p_G(X,{\bf I}^q(\mathscr{L})) \ar[r]^-{{\rm cyc}_{\mathbb{R}}} \ar[d]_{f^\ast} & {\rm H}^p([G\backslash X](\mathbb{R}),\mathbb{Z}(\mathscr{L})) \ar[d]^{f^\ast_{\mathbb{R}}} \\
    {\rm H}^p_H(Y,{\bf I}^q(f^\ast \mathscr{L})) \ar[r]_-{{\rm cyc}_{\mathbb{R}}}  & {\rm H}^p([H\backslash Y](\mathbb{R}),\mathbb{Z}(f^\ast \mathscr{L}))  
  }
  \]
  In particular, the equivariant real cycle class map is compatible with restriction functors ${\rm Res}_G^H=(\varphi,{\rm id})^*$ for a group homomorphism $\varphi\colon H\to G$, and forgetful functors ${\rm Res}_G^1$.
\end{proposition}

\begin{remark}
  To clarify the notation $f^\ast\mathscr{L}$, note that there are also pullback maps on equivariant Chow groups, in particular, we have $f^\ast\colon{\rm Ch}^1_G(X)\to {\rm Ch}^1_H(Y)$. The $f^\ast\mathscr{L}$ in the proposition denotes the image of the class of $\mathscr{L}$ under this pullback map, i.e., $f^\ast\mathscr{L}\in{\rm Ch}^1_H(Y)$ is an $H$-equivariant line bundle on $Y$. 
\end{remark}

\begin{proof}
On the one hand, there is a rather formal argument for real realization (which indeed works for general $\mathbb{A}^1$-representable cohomology theories). Recall that we have a stable real realization functor ${\rm Re}_{\mathbb{R}}\colon \mathbf{SH}(\mathbb{R})\to\mathbf{SH}$, by taking geometric fixed points of the ${\rm C}_2$-equivariant realization. For a motivic spectrum $E\in\mathbf{SH}(\mathbb{R})$ and a motivic space $X$, this induces a morphism
\[
E^{**}(X)={\rm Hom}_{\mathbf{SH}(\mathbb{R})}(\Sigma^\infty_{\mathbb{P}^1}X_+, E(*)[*])\to {\rm Hom}_{\mathbf{SH}}(\Sigma^\infty({\rm Re}_{\mathbb{R}}X_+),{\rm Re}_{\mathbb{R}}(E)[*])={\rm Re}_{\mathbb{R}}(E)^*({\rm Re}_{\mathbb{R}}(X)).
\]
The contravariant functoriality for $E^{**}(X)$ is obtained by pre-composition with a morphism $f\colon Y\to X$ to get a morphism $E^{**}(X)\to E^{**}(Y)$. The same is true on the real-realization side, where we pre-compose with the morphism ${\rm Re}_{\mathbb{R}}(f)\colon {\rm Re}_{\mathbb{R}} Y\to {\rm Re}_{\mathbb{R}} X$. The claim is then really just functoriality of real realization, applied to the morphism $(\varphi,f)\colon [H\backslash Y]\to[G\backslash X]$ of quotient stacks, viewed as motivic spaces, as discussed in Section~\ref{sec:stacks-motivic-simplicial}. The line bundle twists are accounted for by taking the motivic space to be ${\rm Th}(\mathscr{L})$ instead of $[G\backslash X]$, for $\mathscr{L}$ a $G$-equivariant line bundle on $X$, and similarly for $[H\backslash Y]$. 

We can also give a more specific argument which works for equivariant cohomology with homotopy module coefficients. Restrict to the case where the morphism of quotient stacks $[H\backslash Y]\to [G\backslash X]$ is representable. In this case, for given $p$, there is a morphism of smooth varieties $g\colon U\to V$ and a commutative diagram
\[
\xymatrix{
  {\rm H}^p_G(X,{\bf I}^q(\mathscr{L}))\ar[r]^\cong \ar[d]_{f^\ast} & {\rm H}^p(V,{\bf I}^q(\mathscr{L})) \ar[d]^{g^\ast} \\
  {\rm H}^p_H(Y,{\bf I}^q(\mathscr{L}))\ar[r]_\cong & {\rm H}^p(U,{\bf I}^q(\mathscr{L}))
}
\]

We have a similar diagram on the topological side, using the stabilization for singular cohomology of realizations of admissible gadgets, as discussed in the proof of Theorem~\ref{thm:real-cycle-class}:
\[
\xymatrix{
  {\rm H}^p(V(\mathbb{R}),\mathbb{Z}(\mathscr{L})) \ar[r]^\cong \ar[d]_{g_{\mathbb{R}}^\ast} &   {\rm H}^p([G\backslash X](\mathbb{R}),\mathbb{Z}(\mathscr{L})) \ar[d]^{f_{\mathbb{R}}^\ast} \\
  {\rm H}^p(U(\mathbb{R}),\mathbb{Z}(\mathscr{L})) \ar[r]_\cong &  {\rm H}^p([H\backslash Y](\mathbb{R}),\mathbb{Z}(\mathscr{L}))
}
\]

Combining these diagrams with the pullback compatibility diagram for the non-equivariant real cycle class map in \cite{4real}*{Proposition~4.1}, the claim follows.
\end{proof}

\begin{remark}
  \begin{itemize}
    \item The formal argument in the proof shows that the compatibility with restriction functors indeed is true for any representable cohomology theory, in particular also in the more general situations discussed in Section~\ref{sec:stable-realization}.
    \item There is also a version of pullback compatibility for equivariant homology. This works more generally for singular varieties with action. However, it is restricted to morphisms $[H\backslash Y]\to [G\backslash X]$ which are representable and local complete intersection. 
  \end{itemize}
\end{remark}

\begin{proposition}
  \label{prop:compat-product}
  Let $G\looparrowright X$ be a smooth scheme with action over $\mathbb{R}$, and let $\mathscr{L}_1,\mathscr{L}_2\in{\rm Ch}^1_G(X)$ be $G$-equivariant line bundles on $X$. Then the real cycle class map is compatible with the equivariant intersection product, in the sense that the following diagram commutes:
  \[
  \xymatrix{
    {\rm H}^{p_1}_G(X,{\bf I}^{q_1}(\mathscr{L}_1))\otimes {\rm H}^{p_2}_G(X,{\bf I}^{q_2}(\mathscr{L}_2)) \ar[r]^-{{\rm cyc}_{\mathbb{R}}} \ar[d]_\cup & {\rm H}^{p_1}_{\rm sing}([G\backslash X](\mathbb{R}),\mathbb{Z}(\mathscr{L}_1))\otimes    {\rm H}^{p_2}_{\rm sing}([G\backslash X](\mathbb{R}),\mathbb{Z}(\mathscr{L}_2))  \ar[d]^\cup \\
    {\rm H}^{p_1+p_2}_G(X,{\bf I}^{q_1+q_2}(\mathscr{L}_1\otimes\mathscr{L}_2))\ar[r]_-{{\rm cyc}_{\mathbb{R}}} & {\rm H}^{p_1+p_2}_{\rm sing}([G\backslash X](\mathbb{R}),\mathbb{Z}(\mathscr{L}_1\otimes \mathscr{L}_2))
  }
  \]
\end{proposition}

\begin{proof}
  Again, there is a formal argument which applies to any cohomology representable in $\mathbf{SH}(\mathbb{R})$ by a ring spectrum. For this, recall from Section~\ref{sec:stable-realization} that the real realization functor ${\rm Re}_{\mathbb{R}}\colon\mathbf{SH}(\mathbb{R})\to\mathbf{SH}$ is symmetric monoidal, i.e., compatible with the monoidal structures. For a motivic spectrum $E\in\mathbf{SH}(\mathbb{R})$ and a motivic space $X$, we can represent cohomology classes $\sigma,\tau$ by maps $\Sigma^\infty X_+\to E(*)[*]$, and then their product is represented by the map
  \[
  \Sigma_{\mathbb{P}^1}^\infty X_+\xrightarrow{\Delta}\Sigma_{\mathbb{P}^1}^\infty X_+\wedge\Sigma_{\mathbb{P}^1}^\infty X_+\xrightarrow{\sigma\wedge\tau} E(p)[q]\wedge E(p')[q']\xrightarrow{\mu} E(p+p')[q+q'].
  \]
  Since ${\rm Re}_{\mathbb{R}}$ is symmetric monoidal, the realization of this map is the composition
  \[
X_{\mathbb{R}}\xrightarrow{{\rm Re}_{\mathbb{R}}(\Delta)}X_{\mathbb{R}}\wedge X_{\mathbb{R}} \xrightarrow{{\rm Re}_{\mathbb{R}}(\sigma)\wedge{\rm Re}_{\mathbb{R}}(\tau)} E_{\mathbb{R}}[q]\wedge E_{\mathbb{R}}[q']\xrightarrow{{\rm Re}_{\mathbb{R}}(\mu)} E_{\mathbb{R}}[q+q'],
  \]
  where we denote $X_{\mathbb{R}}=\Sigma^\infty{\rm Re}_{\mathbb{R}}(X)_+\cong {\rm Re}_{\mathbb{R}}\left(\Sigma^\infty_{\mathbb{P}^1}X_+\right)$ and $E_{\mathbb{R}}={\rm Re}_{\mathbb{R}}(E)$.\footnote{As before, we blur notational distinction between unstable and stable real realization here.} The resulting composition represents the product of the realizations of $\sigma$ and $\tau$, using the structure of ring spectrum on ${\rm Re}_{\mathbb{R}}(E)$ given by ${\rm Re}_{\mathbb{R}}(\mu)$. 
  
  As for Proposition~\ref{prop:compat-pullback}, we also have a more specific argument, using the approximation of the quotient stack by smooth schemes. Using an admissible gadget $\rho$ for $G$, as in Section~\ref{sec:stacks-motivic-approx}, there is a smooth scheme $X_G^i(\rho)$ and a morphism $X_G^i(\rho)\to [G\backslash X]$ which induces isomorphisms
  \[
  {\rm H}^r_G(X,{\bf I}^s(\mathscr{L}_i))\xrightarrow{\cong} {\rm H}^r(X_G^i(\rho), {\bf I}^s(\mathscr{L}_i))
  \]
  in degrees $r\leq p_1+p_2$, and is compatible with the intersection product in these degrees. By the definition of the equivariant real cycle class map, the commutativity of the diagram claimed in the proposition statement is equivalent to the commutativity of the following diagram:
  \[
  \xymatrix{
    {\rm H}^{p_1}(X_G^i(\rho),{\bf I}^{q_1}(\mathscr{L}_1))\otimes {\rm H}^{p_2}(X_G^i(\rho),{\bf I}^{q_2}(\mathscr{L}_2)) \ar[d]_-{{\rm cyc}_{\mathbb{R}}} \ar[r]^-\cup &
 {\rm H}^{p_1+p_2}(X_G^i(\rho),{\bf I}^{q_1+q_2}(\mathscr{L}_1\otimes\mathscr{L}_2))\ar[d]^-{{\rm cyc}_{\mathbb{R}}}    \\ {\rm H}^{p_1}_{\rm sing}(X_G^i(\rho)(\mathbb{R}),\mathbb{Z}(\mathscr{L}_1))\otimes    {\rm H}^{p_2}_{\rm sing}(X_G^i(\rho)(\mathbb{R}),\mathbb{Z}(\mathscr{L}_2))  \ar[r]_-\cup  & {\rm H}^{p_1+p_2}_{\rm sing}(X_G^i(\rho)(\mathbb{R}),\mathbb{Z}(\mathscr{L}_1\otimes \mathscr{L}_2))
  }
  \]
  But the commutativity of this diagram follows from \cite{4real}*{Proposition~4.3}.
\end{proof}

\begin{remark}
  \begin{itemize}
  \item 
    Again, the formal argument works for other cohomology theories as well, such as variants of hermitian K-theory or cobordism spectra like ${\rm MSL}$ and ${\rm MSp}$.
  \item There is a similar version for external/cross products of cohomology classes:
    \[
      {\rm H}^{p_1}_G(X,{\bf I}^{q_1}(\mathscr{L}_1))\otimes {\rm H}^{p_2}_G(Y,{\bf I}^{q_2}(\mathscr{L}_2))\to {\rm H}^{p_1+p_2}_G(X\times Y,{\bf I}^{q_1+q_2}(\mathscr{L}_1\boxtimes\mathscr{L}_2))
    \]
    As is well-known, the cup product is the pullback along the diagonal $\Delta\colon X\to X\times X$ of the cross product. Similarly, the cross product of cohomology classes on $X$ and $Y$ can be defined as cup product of the pullbacks along the projections $X\xleftarrow{{\rm pr}_X}X\times Y\xrightarrow{{\rm pr}_Y}Y$. In particular, the statement of Proposition~\ref{prop:compat-product} is equivalent to the corresponding statement for the cross product, using the compatibility with pullbacks from Proposition~\ref{prop:compat-pullback}. 
  \end{itemize}
\end{remark}

\begin{proposition}
  \label{prop:compat-pushforward}
  Let $(\varphi,f)\colon (H\looparrowright Y)\to (G\looparrowright X)$ be a morphism of smooth schemes with actions over $\mathbb{R}$ such that the induced morphism $f\colon [H\backslash Y]\to [G\backslash X]$ of quotient stacks is a representable closed immersion. Let $\mathscr{L}\in{\rm Ch}^1_G(X)$ be a $G$-equivariant line bundle on $X$. Then the real cycle class map is compatible with proper pushforward in the sense that the following diagram commutes:
  \[
  \xymatrix{
    {\rm H}^p_H(Y,{\bf I}^q(f^*\mathscr{L}\otimes\omega_f))\ar[r]^-{{\rm cyc}_{\mathbb{R}}} \ar[d]_{f_*} & {\rm H}^p([H\backslash Y](\mathbb{R}),\mathbb{Z}(f^*\mathscr{L}\otimes\omega_f)) \ar[d]^{f_*}\\
   {\rm H}^{p+d}_G(X,{\bf I}^q(\mathscr{L}))\ar[r]_-{{\rm cyc}_{\mathbb{R}}}  & {\rm H}^{p+d}([G\backslash X](\mathbb{R}),\mathbb{Z}(\mathscr{L})) 
  }
  \]
\end{proposition}

\begin{proof}
  As in the proof of Proposition~\ref{prop:compat-pullback}, we use the representability of the morphism of $f$ to approximate $f$ with the help of an admissible gadgets $\rho=(U_j,V_j)_{j\geq 1}$. For the given $p$, there is a closed immersion of smooth varieties $g\colon (Y\times_{/H}U_j)\to (X\times_{/G}U_j)$ and a commutative diagram
  \[
  \xymatrix{
    {\rm H}^p_H(Y,{\bf I}^q(f^\ast(\mathscr{L}\otimes \omega_f))) \ar[r]^-\cong \ar[d]_{f_*} & {\rm H}^p(Y\times_{/H}U_j,{\bf I}^q(f^\ast(\mathscr{L}\otimes \omega_g))) \ar[d]^{g_*} \\
    {\rm H}^{p+d}_G(X,{\bf I}^q(\mathscr{L})) \ar[r]_-\cong & {\rm H}^{p+d}(X\times_{/G}U_j,{\bf I}^q(\mathscr{L}))
  }
  \]
  There is a similar diagram on the topological side, using the stabilization for singular cohomology of realizations of admissible gadgets, as discussed in the proof of Theorem~\ref{thm:real-cycle-class}. 
  \[
  \xymatrix{
    {\rm H}^p((Y\times_{/H}U_j)(\mathbb{R}),\mathbb{Z}(f^\ast\mathscr{L}\otimes \omega_g)) \ar[r]^-\cong \ar[d]_{g^{\mathbb{R}}_*} & {\rm H}^p([H\backslash Y](\mathbb{R}),\mathbb{Z}(f^\ast\mathscr{L}\otimes \omega_f)) \ar[d]^{f^{\mathbb{R}}_*} \\
    {\rm H}^{p+d}((X\times_{/G}U_j)(\mathbb{R}),\mathbb{Z}(\mathscr{L})) \ar[r]_-\cong & {\rm H}^{p+d}([G\backslash X](\mathbb{R}),\mathbb{Z}(\mathscr{L})).
  }
  \]
  Combining these diagrams with the pushforward compatibility diagram for the non-equivariant real cycle class map in \cite{4real}*{Theorem~4.7}, the claim follows.
\end{proof}

\begin{remark}
 For the compatibility with proper pushforwards, it is not so clear if there is similarly general formal argument as for pullbacks and intersection products. Essentially, this is a question of whether the real realization functors of Section~\ref{sec:unstable-real-realization} are compatible with Becker--Gottlieb transfers. We are not aware of such a result in the available literature.
\end{remark}

With the above compatibility results at hand, several further consequences can be deduced. We list the compatibility with quotient and induction equivalence. 

\begin{corollary}
  \label{cor:quotient-equiv}
  Let $G\looparrowright X$ be a smooth variety with action and let $N\subset G$ be a closed normal subgroup such that the restricted action $N\looparrowright X$ is free. Then the equivariant cycle class map is compatible with the quotient equivalence in the sense that the following diagram commutes:
  \[
  \xymatrix{
    {\rm H}^*_{G/N}(N\backslash X,{\bf I}^q(\mathscr{L}))\ar[r]^-{{\rm cyc}_{\mathbb{R}}} \ar[d]_\cong & {\rm H}^*_{\rm sing}([(G/N)\backslash(N\backslash X)](\mathbb{R}),\mathbb{Z}(\mathscr{L})) \ar[d]^\cong \\
    {\rm H}^*_{G}(X,{\bf I}^q(p^*\mathscr{L})) \ar[r]_-{{\rm cyc}_{\mathbb{R}}} & {\rm H}^*_{\rm sing}([G\backslash X](\mathbb{R}),\mathbb{Z}(p^*\mathscr{L}))
  }
  \]
\end{corollary}

\begin{proof}
  The isomorphism of stacks in the quotient equivalence of Proposition~\ref{prop:quotient-equiv} is induced by pullback along $(\pi,p)\colon (G\looparrowright X)\to (G/N\looparrowright N\backslash X)$. The claim then follows from the pullback compatibility in Proposition~\ref{prop:compat-pullback}. 
\end{proof}

\begin{corollary}
  \label{cor:induction-equiv}
    Let $i\colon H\hookrightarrow G$ be the inclusion of a smooth closed subgroup into a smooth affine algebraic group $G$ and let $H\looparrowright X$ be a smooth variety with action such that the quotient $G\times_{H}X$ of the diagonal $H$-action on $G\times X$ exists as a smooth quasi-projective variety. Then the equivariant real cycle class map is compatible with the induction equivalence in the sense that the following diagram commutes: 
  \[
  \xymatrix{
    {\rm H}^*_{G}(G\times_{/H}X,{\bf I}^q(\mathscr{L}))\ar[r]^-{{\rm cyc}_{\mathbb{R}}} \ar[d]_\cong & {\rm H}^*_{\rm sing}([G\backslash(G\times_{/H} X)](\mathbb{R}),\mathbb{Z}(\mathscr{L})) \ar[d]^\cong \\
    {\rm H}^*_{H}(X,{\bf I}^q(p^*\mathscr{L})) \ar[r]_-{{\rm cyc}_{\mathbb{R}}} & {\rm H}^*_{\rm sing}([H\backslash X](\mathbb{R}),\mathbb{Z}(p^*\mathscr{L}))
  }
  \]
\end{corollary}

\begin{proof}
  The isomorphism of stacks in the induction equivalence of Proposition~\ref{prop:induction-equiv} is induced by pullback along $(H\looparrowright X)\to (G\looparrowright G\times_{/H}X)$. The claim then follows from the pullback compatibility in Proposition~\ref{prop:compat-pullback}. 
\end{proof}

The equivariant real cycle class map is also compatible with the localization sequence. The formulations are along the lines of \cite{4real}*{Proposition~4.8 and Corollary~4.13} which the reader will have no difficulties spelling out.

\subsection{Variations: equivariant cycle class map on Chow groups of real schemes}

In this subsection, we will state a mod 2 version of the main results on equivariant real cycle class maps. 

\begin{theorem}
  \label{thm:borel-haefliger-equivariant}
  Let $G\looparrowright X$ be a smooth scheme with action by a smooth affine algebraic group over $\mathbb{R}$. The identification ${\bf K}^{\rm M}_q/2\cong\mathscr{H}_\et^q(\mu_2)$ from Voevodsky's solution of the Milnor conjecture induces an equivariant mod 2 cycle class map
  \[
  {\rm H}^p_G(X,{\bf K}^{\rm M}_q/2)={\rm H}^p([G\backslash X],{\bf K}^{\rm M}_q/2)\to {\rm H}^p({\rm Re}_{\mathbb{R}}[G\backslash X],\mathbb{F}_2)
  \]
  which satisfies the following properties:
  \begin{enumerate}
  \item The equivariant mod 2 cycle class map only depends on the stack $[G\backslash X]$.
  \item The case of the trivial group $G=\{1\}$ recovers the non-equivariant mod 2 cycle class map of Borel--Haefliger \cite{borel:haefliger} and Colliot-Th\'el{\`e}ne--Parimala \cite{colliot-thelene:parimala}.
  \item The equivariant mod 2 cycle class map is an isomorphism for big enough $q$.
  \end{enumerate}
\end{theorem}

\begin{remark}
  \begin{itemize}
  \item For $p=q$, this is the mod 2 cycle class map
    \begin{equation}
      {\rm Ch}^p_G(X)\to {\rm H}^p({\rm Re}_{\mathbb{R}}[G\backslash X],\mathbb{F}_2)
    \end{equation}
    of Borel--Haefliger \cite{borel:haefliger}. The more general setting and the isomorphism for big enough $q$ follows from work of Colliot-Th\'el\`ene and Parimala \cite{colliot-thelene:parimala}, cf. also \cite{jacobson}.
  \item Concerning which $q$ is big enough, Remark~\ref{rem:big-enough} applies.
  \item Again, there is also a Borel--Moore homology version of the result.
  \item The proof of Theorem~\ref{thm:real-cycle-class} carries over with only minor modifications to give a proof of Theorem~\ref{thm:borel-haefliger-equivariant}. There are no line bundles twists since the mod 2 Milnor K-cohomology is oriented.
  \item The mod 2 equivariant cycle class map enjoys the same compatibilities as the real cycle class map. As before, the proofs of Propositions~\ref{prop:compat-pullback}, \ref{prop:compat-product} and \ref{prop:compat-pushforward} carry over with only minor adjustments. 
  \item It seems likely that there is an equivariant version of the Benoist--Wittenberg cycle class map on Chow groups of real schemes \cite{benoist:wittenberg} which could be used to produce an equivariant cycle class map on equivariant Chow--Witt groups as in \cite{4real}. We will not pursue this direction here.
  \end{itemize}
\end{remark}

\section{Bitorsors and equivalences of quotient stacks}
\label{sec:bitorsors}

In this section, we will discuss some invariance properties of classifying stacks $[G\backslash *]$. The results in our paper show that many constructions we are interested in -- such as the equivariant or real realizations, cohomology with homotopy-module coefficients, as well as the real cycle class map -- only depend on the equivalence class of a quotient stack. In the special case of a classifying stack $[G\backslash *]$, one can then wonder if it is possible to change the group $G$ without changing the equivalence class of the stack.

One particular instance is the formula for the real realization ${\rm Re}_{\mathbb{R}}[G\backslash *]$ of the classifying stack of a real algebraic group $G$ in Proposition~\ref{prop:bg-formula}: the components of ${\rm Re}_{\mathbb{R}}[G\backslash *]\cong{\rm B}G(\mathbb{C})^{\rm hC_2}$ are classifying spaces of strong real forms of $G(\mathbb{R})$ with the same central invariant as $G$. As a consequence, these strong real forms will all have equivalent real realizations. In this section, we will show that this is not a coincidence; in fact, different strong real forms with the same central invariant actually have equivalent classifying stacks, thus explaining the observation from the formula in Proposition~\ref{prop:bg-formula}. This will also imply that the Witt-sheaf cohomology rings of ${\rm B}_\et G$ for different strong real forms with the same central invariant will be isomorphic, cf. the statement of Corollary~\ref{cor:cohom-inv-strong-forms}.

Essentially, what we want to show in this section is that groups $G,G'$ have equivalent quotient stacks if and only if they are strongly inner forms of each other. For the definition of strongly inner forms, see Definition~\ref{def:strong-inner-forms} and the discussion in Remark~\ref{rem:comparison-strong-inner}; for background on the notion of bitorsors which appears in the proof, cf.{}~\cite{GF1} and \cite{giraud:coh_non_ab}. 

\begin{theorem}[Giraud]
  \label{thm:strong-inner--eq_of_groupoids} 
  Let $G$ and $H$ be smooth affine algebraic groups over a field $F$. There is natural equivalence between the groupoid $\mathrm{Fun}([G\backslash\pt], [H\backslash\pt])^{\simeq}$ of invertible morphisms of stacks and natural transformations (automatically isomorphisms) between them, and the groupoid of $(H,G)$-bitorsors\footnote{Note that a $(H,G)$-bitorsor is a left $H$-torsor and a right $G$-torsor.} over the field $F$ and isomorphisms between them. This natural equivalence is given on objects by mapping an equivalence $\Phi\colon[G\backslash\pt] \ra[H\backslash\pt]$ to its evaluation $\Phi_F(G_l)$ on the trivial left $G$-torsor $G_l$ over $\spec F$. A quasi-inverse is given on objects by mapping a $(H,G)$-bitorsor $P$ to the functor $\Phi_P:=P\times_{/G}(-)$.
\end{theorem}

\begin{proof}
  Recall that the classifying stack $[G\backslash\pt]$ is a gerbe over $\mathrm{Sch}_F$ which is trivialised by its global section which assigns to a scheme $X$ the trivial left $G$-torsor $G_l=G\times X$. For a general stack $\mathcal X$, Theorem~2.2.2 in Chapter~III of \cite{giraud:coh_non_ab} identifies the groupoid $\rm{Fun}([G\backslash\pt], \mathcal X)$ of all morphisms of stacks (not necessarily equivalences) out of $[G\backslash\pt]$ with the groupoid of objects with a right $G$-action in $\mathcal X(F)$.

  We briefly recap this identification. On objects, it is given by mapping a morphism $\Phi\colon[G\backslash\pt]\to\mathcal{X}$ to the object $\Phi_{F}(G_l)\in \mathcal X(F)$ with right $G$-action given by the composition $\gamma\circ \mu$ of the following two maps of sheaves of groups
  \begin{align}
    \label{eq:gamma-isom}
  \mu\colon G\ra \underline{\rm{Isom}}_G(G_l, G_l) \quad \gamma\colon\underline{\rm{Isom}}_G(G_l, G_l)\ra \underline{\rm{Isom}}_{\mathcal X}(\Phi(G_l),\Phi(G_l)),
  \end{align}
  where $\mu(g)=(-)\cdot g$ is the right action of $G$ on the trivial $G$-torsor, and $\gamma$ is the map induced by the functor $\Phi$. On morphisms, Giraud's equivalence is defined in the obvious way, mapping a natural transformation $f\colon\Phi\Rightarrow \Psi$ of stack morphisms $\Phi,\Psi\colon[G\backslash\pt]\to\mathcal{X}$ to  the $G$-equivariant map $f_F(G_l)\colon\Phi_F(G_l)\ra \Psi_F(G_l)$ of objects with $G$-action in $\mathcal X(F)$.
  
  In the case where $\mathcal X=[H\backslash\pt]$ is a classifying stack, the morphism $\gamma\circ \mu$ endows the left $H$-torsor $\Phi_F(G_l)$ (over $\spec F$) with a right $G$-action, commuting with that of $H$. To prove the first claim in the theorem, it thus remains to show that a morphism of classifying stacks $\Phi\colon[G\backslash\pt] \ra[H\backslash\pt]$ is an equivalence if and only if the actions of $G$ and $H$ on $\Phi_F(G_l)$ described above make it into a $(H,G)$-bitorsor, or equivalently, that the map $\gamma$ in \eqref{eq:gamma-isom} above is an isomorphism. The second claim, on the precise form of the identification, follows by construction.

  Let thus $\Phi\colon[G\backslash\pt] \ra[H\backslash\pt]$ be a morphism of stacks. It is clear that $\gamma$ is an isomorphism if $\Phi$ is an equivalence of stacks. For the converse, assume that $\gamma$ is an isomorphism, and we want to show that $\Phi$ is an equivalence of stacks. This is essentially a consequence of the fact that every torsor is locally trivial. More precisely, we work as follows. Thanks to Proposition~1.4.5 in Chapter~II of \cite{giraud:coh_non_ab}, in order to prove that $\Phi$ is an equivalence, it is enough to choose a cleavage and check that:
  \begin{enumerate}[label=\roman*)]
  \item For every $F$-scheme $S$ and every pair of objects $X,Y\in [G\backslash\pt](S)$, the following natural map (of sheaves of isomorphisms, induced by $\Phi$) is an isomorphism:
    \[\begin{tikzcd}[row sep=5pt, column sep=10pt]
    \underline{\rm{Isom}}_G|_S(X,Y) \arrow[r,"\Phi_{X,Y}"] &  \underline{\rm{Isom}}_H|_S(\Phi X,\Phi Y)\\
    \left(\varphi\colon X_{|U}\ra Y_{|U}\right) \arrow[r,mapsto] & \left(\Phi_{X,Y}(\varphi)\colon(\Phi X)_{|U}\ra(\Phi Y)_{|U}\right)
    \end{tikzcd}\]
    where $U$ is an $S$-scheme, and $\Phi_{X,Y}(\varphi)$ is the following composition of isomorphisms of $H$-torsors over $U$:\footnote{Here, we use the functoriality of sheaves of isomorphisms for the identifications $\Phi(X)_{|U}\simeq\Phi(X_{|U})$, and similarly for $Y$.}
    \[
    \Phi(X)_{|U}\simeq \Phi(X_{|U})\overset{\Phi(\varphi)}{\ra} \Phi(Y_{|U})\simeq \Phi(Y)_{|U}.
    \]
  \item For every $F$-scheme $S$ and every object $Z\in[H\backslash\pt](S)$ there is a cover $S'\ra S$ such that $Z_{|S'}$ is isomorphic to $\Phi(X)$ for some left $G$-torsor $X$.
  \end{enumerate}
  The second point is clearly verified (as a consequence of local triviality of torsors). To show that $\Phi$ is an equivalence (assuming $\gamma$ is an isomorphism), it remains to show the first point. Let $p \colon S'\ra S$ be a cover that simultaneously trivialises $X$ and $Y$. Since $[G\backslash\pt]$ and $[H\backslash\pt]$ are stacks, the map $\Phi_{X,Y}$ is obtained by descent from its restriction $(\Phi_{X,Y})_{|S'}$. In particular, it suffices to check that $(\Phi_{X,Y})_{|S'}$ is an isomorphism. 

  By the functoriality of sheaves of isomorphisms, the map $(\Phi_{X,Y})_{|S'}$ is identified with the map $\Phi_{X_{|S'},Y_{|S'}}$. Since $S'$ trivialises $X$ and $Y$ simultaneously, the latter map is identified with $(\Phi_{{G_l},{G_l}})_{|S'}=\gamma_{|S'}$ through the chosen trivialisations of $X$ and $Y$. 
  More precisely, the various identifications made in this argument fit in a commutative square\footnote{  Our notation here differs slightly from previous usage in that the index $\underline{\rm Isom}_S$ denotes the scheme on which we consider the relevant isomorphism sheaf. All the isomorphism sheaves on the left are isomorphisms of $G$-torsors in $[G\backslash\pt]$, while all the isomorphism sheaves on the right concern $H$-torsors. 
}
  \[
  \begin{tikzcd}[column sep=40pt]
    \underline{\rm{Isom}}_S(X,Y)_{|S'} \arrow[r,"(\Phi_{X,Y})_{|S'}"] \arrow[d,"\simeq"]& \underline{\rm{Isom}}_S(\Phi (X),\Phi (Y))_{|S'} \arrow[d,"\simeq"]\\
    \underline{\rm{Isom}}_{S'}(X_{|S'},Y_{|S'}) \arrow[r, "\Phi_{X_{S'},Y_{S'}}"] \arrow[d,"\simeq"]& \underline{\rm{Isom}}_{S'}(\Phi (X_{|S'}),\Phi (Y_{|S'})) \arrow[d,"\simeq"]\\
    \underline {\rm{Isom}}_{S'}({G_l}_{|S'},{G_l}_{|S'}) \arrow[r, "\Phi_{{G_l}_{|S'},{G_l}_{|S'}}"] & \underline{\rm{Isom}}_{S'}(\Phi ({G_l}_{|S'}),\Phi ({G_l}_{|S'}))  \\
    \underline {\rm{Isom}}_{F}({G_l},{G_l})_{|S'} \arrow[r,"(\Phi_{{G_l},{G_l}})_{|S'}"]\arrow[u,"\simeq"']& \underline{\rm{Isom}}_{F}(\Phi ({G_l}),\Phi ({G_l}))_{|S'}\arrow[u,"\simeq"']
  \end{tikzcd}
  \]
  From top to bottom, on the left hand side, we have the following maps: the canonical identification of the pull-back of the sheaf of isomorphisms over $S$ to the sheaf of isomorphisms over $S'$, pre- and post-composition of an isomorphism of torsors with the chosen trivialisations of $X$ and $Y$, and finally another canonical identification of the pull-back of an isomorphism sheaf. On the right hand site, top to bottom, we have the same maps, combined with the tautological identifications $\Phi(X_{|S'})\simeq\Phi(X)_{|S'}$ from the functoriality of sheaves of isomorphisms, and similarly for $Y$. 
  Finally, the bottom morphism is just the base-change of $\gamma=\Phi_{{G_l},{G_l}}$ by definition, and so all horizontal arrows in the above diagram are isomorphisms. We conclude that $\Phi$ is an equivalence. 

  It remains to prove the last statement, concerning the quasi-inverse. For this, let now $P$ be any $(H,G)$-bitorsor and let $\Phi_P:= P\times_{/G} (-)\colon[G\backslash\pt] \ra[H\backslash\pt]$ be the associated morphism of stacks. Since we have an isomorphism
  \[
  \chi\colon  P \ra P\times_{/G} G_l=\Phi_P(G_l)\colon p\mapsto (p,1)
  \]
  of $(H,G)$-bitorsors, we deduce that $\Phi_P$ is actually an equivalence. Moreover the identification $P\ra\Phi_P(G_l)$ is natural in $P$, and thus $\Phi_P$ is a quasi-inverse to the restriction of Giraud's equivalence.
\end{proof} 

\begin{theorem}[Giraud]
  \label{thm:strong-inner--eq_formulations} 
  Let $G$ and $H$ be smooth affine algebraic groups over a field $F$. Then the following are equivalent:
  \begin{enumerate}
  \item There exists an equivalence of stacks $[G\backslash\pt]\simeq[H\backslash\pt]$.
  \item There exists a $(H,G)$-bitorsor.
  \item The groups $G$ and $H$ are strongly inner forms of each other, in the sense of Definition~\ref{def:strong-inner-forms}.
  \end{enumerate}
\end{theorem}

\begin{proof}
  The equivalence of the first two points is a consequence of Theorem~\ref{thm:strong-inner--eq_of_groupoids}, so it suffices to show the equivalence of the last two points. This equivalence is essentially again a result of Giraud, contained in Section~2.5 in Chapter~III of \cite{giraud:coh_non_ab}. For convenience, we explicitly spell out the translation between Galois cocycles for strong inner forms and bitorsors.

  Recall that, by definition, an (\'etale) form of an algebraic group $H$ is a sheaf of groups \'etale-locally isomorphic to $H$. There is a canonical bijection 
   \[I\colon\op H^1_{\textrm{\'et}}(F,\mathrm{Aut}(H))\simeq \{F\textrm{-forms of } H\}/F\textrm{-Isomorphism}\]constructed as follows. Given a form $H'$ of $H$, choose local isomorphisms $\varphi_i\colon H'_{|S_i} \ra H_{|S_i}$ on some \'etale cover $(S_i \ra \spec F)_{i\in I}$, and define $I(H')$ to be the class of the cocycle $(\varphi_{i,j})_{i,j \in I}$, where $\varphi_{i,j}:=\varphi_i\circ\varphi_j^{-1}\in \mathrm{Aut}(H)(S_i\times_FS_j)$ for $i,j \in I$. 

   Let $P$ be a right $H$-torsor and choose local trivialisations $\sigma_i\colon P_{|S_i}\ra H_{|S_i}$ on an \'etale cover $(S_i\ra\spec F)_{i\in I}$. We can then use the $\sigma_i$'s to cook up the cohomology class representing $P$, namely the class $[\sigma] \in \op H^1_{\textrm{\'et}}(F,H)$ of the cocycle $\sigma_{i,j}:=\sigma_i\circ\sigma_j^{-1}\in \mathrm{Aut}_H(H_r)(S_{i,j})\simeq H(S_{i,j})$. On the other hand we can use the local trivialisations $\sigma_i$ to provide locally defined group isomorphisms 
   \[\varphi_i\colon\mathrm{Aut}_{H_{|S}}(P_{|S})\simeq\mathrm{Aut}_{H_{|S_i}}({H_{|S_i}}_r)\simeq H_{|S_i}, \quad f\mapsto \sigma_i\circ f\circ \sigma_i^{-1}.\]
   This shows that $\mathrm{Aut}_{H}(P)$ is a form of $H$ with associated cohomology class $[\varphi]\in \op H^1_{\textrm{\'et}}(F,\mathrm{Aut}(H))$ given by $\varphi_{i,j}=\varphi_i\circ\varphi_j^{-1}=c_{\sigma_{i,j}}\in\mathrm{Aut}(H)(S_{i,j})$, where $c\colon H \ra\mathrm{Aut}(H), h\mapsto c_h=h(-)h^{-1}$ is the usual conjugation. In particular, if $P$ is a $(H,G)$-bitorsor, the isomorphism $G\ra\mathrm{Aut}_H(P)$ defining the right action of $G$ shows that $G$ is a form of $H$ whose classifying cocycle is in the image of the map
   \[\op {H}^1_{\textrm{\'et}}(F,H)\ra \op {H}^1_{\textrm{\'et}}(F,\mathrm{Aut}(H))\] 
   induced by conjugation. 

   We now assume that $G$ is a strong inner form of $H$, so that $G$ is classified by a cohomology class $\gamma \in {H}^1_{\textrm{\'et}}(F,\mathrm{Aut}(H))$ which is the image of a class $\beta\in {H}^1_{\textrm{\'et}}(F,H)$. We can then produce a $(H,G)$-bitorsor $P$ as follows. Pick a cover $(S_i\ra\spec F)_{i\in I}$ where $\beta$ is represented by a cocycle $\sigma_{i,j} \in H(S_{i,j})$, and thus $\gamma$ is represented by the cocycle $c_{\sigma_{i,j}}\in \mathrm{Aut}(H))(S_{i,j})$. We define $P$ to be the left $H$-torsor  obtained by gluing the trivial torsors $H_{|S_i\;r}$ on $S_i$ according to the cocycle $\sigma_{i,j}$. We then notice that we have a commutative digram:
   \[
   \begin{tikzcd}
     H_{|S_i}\times H_{|S_i\;r} \arrow[r] \arrow[d,"(c_{\sigma_{i,j}}{,}\sigma_{i,j}\cdot)"] & H_{|S_i\;r} \arrow[d,"\sigma_{i,j}\cdot"] \\
     H_{|S_i}\times H_{|S_i\;r} \arrow[r] & H_{|S_i\;r}
   \end{tikzcd}\]
   where both the horizontal maps are the multiplication map of $H$. This says that the locally defined right action of $H$ on itself by multiplication glues to a right action of $G$ on $P$. This right action induces an isomorphism $G\ra\mathrm{Aut}_H(P)$ since it does so locally by the very construction. So we proved that $P$ is a $(H,G)$-bitorsor.
\end{proof}

\begin{corollary}
  \label{cor:cohom-inv-strong-forms}
  Let $G$ be a linear algebraic group over a field $F$, and let $E$ be a cohomology theory representable in $\mathbf{SH}(F)$. Then the Borel-equivariant $E$-cohomology $E_G^*(\pt)=E^*([G\backslash\pt])$ is invariant under strongly inner forms in the sense of Definition~\ref{def:strong-inner-forms}. More precisely, for any strongly inner form $G'$ of $G$, any $(G,G')$-bitorsor induces an isomorphism
  \[
  E^*([G\backslash\pt])\cong E^*([G'\backslash\pt]).
  \]

  Similarly, in the case $F=\mathbb{R}$, the equivariant real cycle class map is invariant under strong inner forms, in the sense that any $(G,G')$-bitorsor induces a commutative square
  \[
  \xymatrix{
    {\rm H}^p([G\backslash\pt],{\bf I}^q(\mathscr{L})) \ar[r]^{{\rm cyc}_{\mathbb{R}}} \ar[d]_\cong & {\rm H}^p({\rm Re}_{\mathbb{R}}[G\backslash \pt],\mathbb{Z}(\mathscr{L})) \ar[d]^\cong\\
    {\rm H}^p([G'\backslash\pt],{\bf I}^q(\mathscr{L})) \ar[r]_{{\rm cyc}_{\mathbb{R}}} & {\rm H}^p({\rm Re}_{\mathbb{R}}[G'\backslash \pt],\mathbb{Z}(\mathscr{L})) 
  }
  \]
\end{corollary}

\begin{proof}
  By Theorem~\ref{thm:strong-inner--eq_formulations}, we know that strong inner forms have equivalent classifying stacks, and any bitorsor induces an equivalence of stacks. For cohomology theories representable by homotopy modules, we can use \cite{dilorenzo:mantovani}*{Proposition~2.3.1} (or \cite{edidin:graham:equivariant}*{Proposition~16}) to see that equivalences of stacks induce isomorphisms in cohomology. More generally, viewing the quotient stacks as objects in the motivic homotopy category as discussed in Section~\ref{sec:stacks-motivic-simplicial}, equivalent stacks yield equivalent motivic homotopy types. We get the first claim by applying any representable cohomology theory.\footnote{But it should be pointed out that this means we are taking Borel-equivariant cohomology theory, i.e., $E^*([G\backslash\pt])$ is defined by applying the representable cohomology theory to the object $[G\backslash\pt]$ in $\mathbf{H}(F)$ defined in Section~\ref{sec:stacks-motivic-simplicial}.} The invariance claim for the equivariant real cycle class map follows similarly from Theorem~\ref{thm:real-cycle-class}. 
\end{proof}

\begin{remark}
  The fact that strong inner forms have equivalent real realizations can be observed from Proposition~\ref{prop:bg-formula}, since the right-hand side is the disjoint union of classifying spaces of strong real forms. As discussed in Remark~\ref{rem:comparison-strong-inner}, changing from a given group to a strong inner form essentially only changes the indexing of connected components on the right-hand side. Corollary~\ref{cor:cohom-inv-strong-forms} provides a stronger explanation for this observed invariance: strong inner forms have equivalent classifying stacks, which implies equivalences in real realizations as well as isomorphisms in equivariant cohomology. In the case of Chow rings of classifying spaces of orthogonal groups, this was already observed in \cite{RojasVistoli}*{Remark~4.2}. Actually, the remark in their paper prompted our attempts to formulate a general invariance statement for classifying stacks. 
\end{remark}

\begin{example}
  \label{ex:sl2}
  The strong inner form condition is really necessary. One example can be found in \cite{adams:taibi}*{Example~4.9}: ${\rm SL}_2(\mathbb{R})$ and ${\rm SU}(2)$ are inner forms, but have different first Galois cohomology ${\rm H}^1({\rm C}_2,G)$. From Proposition~\ref{prop:bg-formula}, we see that ${\rm Re}_{\mathbb{R}}{\rm B}{\rm SL}_2$ would be connected, while ${\rm Re}_{\mathbb{R}}{\rm BSU}(2)$ would have two connected components. In particular, the real realizations cannot be equivalent.

  We get similar examples in the settings of orthogonal, special orthogonal and spin groups. In the case of orthogonal groups over a field $F$, any two orthogonal groups (of the same rank) will be strong inner forms, which implies that the cohomology or real realization of ${\rm B}_\et{\rm O}(\varphi)$ will be independent of the quadratic form $\varphi$ over $F$ used in the definition of ${\rm O}(\varphi)$. This is no longer true for special orthogonal groups or spin groups. For special orthogonal groups, strong inner forms agree in rank \emph{and discriminant}, and strong inner forms for spin groups additionally agree in their Hasse invariant (or alternatively, the Brauer class of the positive Clifford algebra, as discussed in the examples on p.~655 of \cite{tits}). This means that cohomology and real realization of ${\rm B}_\et{\rm SO}(\varphi)$ will depend on rank and discriminant, and similarly, cohomology and real realization of ${\rm B}_\et{\rm Spin}(\varphi)$ will depend on rank, discriminant and Hasse invariant, cf. e.g.{}~the discussion in \cite{adams:taibi}*{Example~8.21}. The tables in \cite{adams:taibi}*{Section~10} show how numbers of connected components of the real realization of ${\rm B}_\et{\rm Spin}(\varphi)$ vary with rank, discriminant and Hasse invariant.
\end{example}

\begin{remark}
  In the context of the conjectural discussions in Section~\ref{sec:conjectures}, it seems likely that also the symmetric representation ring as well as the augmentation morphism \eqref{eq:augmentation} would only depend on strongly inner forms. Unfortunately, the arguments above do not immediately apply to (genuinely) equivariant Witt groups, but this might be an interesting question for future investigations.
\end{remark}

\section{Examples}
\label{sec:examples}

In this section, we provide a couple of example computations of real realizations of classifying spaces of groups to showcase some of the phenomena arising from the homotopy fixed point description.

Recall from Proposition~\ref{prop:bg-formula} that, for a group $G$ with involution $\sigma\colon g\mapsto\sigma(g)$, there is a natural bijection $\pi_0({\rm B}G^{{\rm hC}_2})\cong{\rm H}^1({\rm C}_2, G(\mathbb{C}))$ and the geometric realization has the form
\begin{align}
  {\rm B}G^{{\rm hC}_2}\simeq \bigsqcup_{[g]\in{\rm H}^1({\rm C}_2,G(\mathbb{C}))} {\rm B}(G(\mathbb{C})^{\sigma_g})
\end{align}
with the stabilizer group $G(\mathbb{C})^{\sigma_g}=\{h\in G\mid \sigma(h)g=gh\}$.

\subsection{Special groups} 

As a first sanity check, we consider special groups in the sense of Serre, i.e., linear algebraic groups $G$ over a field $F$ for which \'etale-locally trivial torsors are already Zariski-locally trivial; in particular, we have ${\rm H}^1_\et(\op{Spec} k,G)=\{\ast\}$ for any field $k/F$.\footnote{In this case, we have equivalences ${\rm B}_{\rm Zar}G\simeq{\rm B}_{\rm Nis}G\simeq{\rm B}_\et G$ and we can drop notational distinctions between the different classifying spaces.} For such a group $G$ over $\mathbb{R}$, Proposition~\ref{prop:bg-formula} reduces to
\[
{\rm Re}_{\mathbb{R}}({\rm B}_\et G)\simeq[G\backslash *](\mathbb{C})^{{\rm hC}_2}\simeq {\rm B}(G(\mathbb{R})),
\]
i.e., the real points of the classifying space have only one component, which is the classifying space of the group $G(\mathbb{R})$ of real points. This is the case for $G={\rm GL}_n,{\rm SL}_n,{\rm Sp}_{2n}$, and for these groups the Witt-sheaf cohomology and Chow--Witt groups are known, cf. \cite{chowwitt} and \cite{realgrassmannian}. Consequently, all the discussion in the present paper reduces to the observations on the strong relations between ${\bf I}$-cohomology of ${\rm B}G$ and the singular cohomology of classifying spaces of $G(\mathbb{R})$, cf. e.g.{}~the discussion of compatibility of algebraic and topological characteristic classes in \cite{4real}*{Section~6}. 

We briefly discuss the conjectural picture (regarding the Borel character and relation to symmetric representations) as set out in Section~\ref{sec:conjectures}. We focus on the case ${\rm SL}_n$; the Witt-sheaf cohomology can be described as follows, cf. \cite{ananyevskiy} or \cite{chowwitt}: ${\rm H}^*({\rm BSL}_n,{\bf W})$ is a polynomial ${\rm W}(F)$-algebra generated by the Pontryagin classes ${\rm p}_{2i}\in{\rm H}^{4i}({\rm BSL}_n,{\bf W})$ for $i=1,\dots,\lfloor (n-1)/2\rfloor$ and an Euler class ${\rm e}_n\in{\rm H}^n({\rm BSL}_n,{\bf W})$ if $n$ is even. Taking inspiration from the Gersten--Witt spectral sequence of \cite{balmer:walter},\footnote{We mention again that the conditions for the existence of the spectral sequence in \cite{balmer:walter} are not satisfied for classifying spaces.} we would expect some spectral sequence of the form
\[
E^{p,q}_2={\rm H}^p({\rm BSL}_n,{\bf W})\Rightarrow {\rm W}({\rm BSL}_n),
\]
concentrated in the lines with $q\equiv 0\bmod 4$. Only the differentials between lines with $q\equiv 0\bmod 4$ can be nontrivial, and these are of the form $d_{4r+1}\colon {\rm H}^p({\rm BSL}_n,{\bf W})\to {\rm H}^{p+4r+1}({\rm BSL}_n,{\bf W})$. The spectral sequence necessarily degenerates because ${\rm H}^*({\rm BSL}_n,{\bf W})$ is concentrated in even degrees; consequently, up to completion issues, we would expect a strong relation between the symmetric representation ring ${\rm W}({\rm SL}_n)$ and the Witt-sheaf cohomology. We briefly indicate what this should look like.

For ${\rm SL}_2$, the fundamental (i.e., defining) representation $\rho$ is self-dual, which means it preserves a bilinear form. This form is a symplectic form (basically by definition), and there is no compatible symmetric bilinear form on the fundamental representation. All irreducible representations arise as symmetric powers of the fundamental representation. This means, in particular, that the symmetric square $\rho^{\otimes 2}$ preserves a symmetric bilinear form. More concretely, the symmetric square $\rho^{\otimes 2}$ is the direct sum of a trivial one-dimensional representation and the adjoint representation; the symmetric form on $\rho^{\otimes 2}$ can also be viewed as evaluation form on $\rho\otimes \rho^\vee$, and the restriction to the adjoint representation is essentially the Killing form. In terms of Witt rings of representations, we find ${\rm W}^*({\rm SL}_2)\cong {\rm W}(F)[\rho]$ with the polynomial generator $\rho\in{\rm W}^2({\rm SL}_2)$ the fundamental representation. Moreover, ${\rm W}^0({\rm SL}_2)\cong{\rm W}(F)[\rho^{\otimes 2}]$ is a polynomial ring generated by the adjoint representation with the Killing form, and ${\rm W}^2({\rm SL}_2)$ is a free rank 1 module over ${\rm W}^0({\rm SL}_2)$. We find a pattern very similar to the Witt-sheaf cohomology of ${\rm SL}_2$, with the fundamental representation corresponding to the Euler class, and the adjoint representation corresponding to its square.

For ${\rm SL}_3$, we get a similar picture. First of all, we note more generally that only self-dual representations of a group $G$ can possibly appear in ${\rm W}^*(G)$; this already rules out the fundamental representation and its dual.\footnote{As an easy check, self-dual representations have weight diagrams which are point-symmetric around the origin.} But for any semisimple group, the adjoint representation with the Killing form is a symmetric representation. In the case ${\rm SL}_3$, the Killing form is actually a polynomial generator for ${\rm W}^*({\rm SL}_3)$ (i.e., every irreducible symmetric representation of ${\rm SL}_3$ is a symmetric power of the adjoint representation). This corresponds exactly to the Witt-sheaf cohomology of ${\rm BSL}_3$ being a polynomial ${\rm W}(F)$-algebra generated by the Pontryagin class ${\rm p}_2\in{\rm H}^4({\rm BSL}_3,{\bf W})$. 

More generally, for ${\rm SL}_n$ we should expect a correspondence as follows: on the cohomological side, we have the Pontryagin classes and possibly an Euler class as generators of Witt-sheaf cohomology. On the representation-theoretic side, these should correspond to generators of the symmetric representation ring ${\rm W}^0({\rm SL}_n)$ given by irreducible summands of $\bigwedge^k\rho\oplus \bigwedge^k\rho^\vee$, and for even $n=2k$ an additional generator in ${\rm W}^{n\bmod 4}({\rm SL}_n)$ corresponding to the middle exterior power $\bigwedge^{k}\rho$.\footnote{Note that for ${\rm SL}_4$, $\bigwedge^2\rho$ preserves a symmetric bilinear form, related to the exceptional isomorphism ${\rm SL}_4\cong{\rm Spin}(6)$.}

For the symplectic groups ${\rm Sp}_{2n}$, we also get a close correspondence between the Witt-sheaf cohomology and symmetric representation ring. In the symplectic case, all representations are self-dual. The fundamental representation $\rho$ by definition preserves a symplectic form, and consequently the exterior powers $\bigwedge^k\rho$ preserve a symplectic form if $k$ is odd and preserve a symmetric form if $k$ is even. In particular, the Witt ring of representations ${\rm W}^*({\rm Sp}_{2n})$ is a polynomial ${\rm W}(F)$-algebra generated by $\bigwedge^k\rho$ of degrees 0 and 2 for $k$ even and odd, respectively. This corresponds precisely to the structure of the Witt-sheaf cohomology ring which is a polynomial ${\rm W}(F)$-algebra generated by the Borel-classes ${\rm b}_{i}\in{\rm H}^{2i}({\rm BSp}_{2n},{\bf W})$ for $i=1,\dots,n$.

We omit a discussion of the situation for ${\rm GL}_n$ which would involve twisted self-dual representations to account for the Euler class being an element in twisted Witt-sheaf cohomology. 

\subsection{Finite groups}
\label{sec:finite}

We make a brief remark about the classifying spaces of a finite group scheme $G$ over $\mathbb{R}$, i.e., a finite group $G$ equipped with an involution $\iota$. For the case of the trivial involution, Proposition~\ref{prop:bg-formula} provides the following general description. This description was also obtained independently by Linda Carnevale.

\begin{proposition}
  Let $G$ be a split finite real group scheme, i.e., a finite group $G$ with trivial involution $\iota$. Then the real points of the classifying space ${\rm B}_\et G$ are described as follows:
  \[
  \left({\rm B}_\et G\right)(\mathbb{R})\simeq\bigsqcup_{\sigma\in\mathcal{I}}{\rm B}G^\sigma,
  \]
  where $\mathcal{I}$ denotes the set of conjugacy classes of elements $\sigma\in G$ with $\sigma^2=1$, and $G^\sigma=\{g\in G\mid g\sigma=\sigma g\}$ is the fixed group of the involution $\sigma$.
\end{proposition}
 
\begin{example}
  In the special case of cyclic groups $G={\rm C}_{2n}$ with trivial involution, we find that there are two conjugacy classes of order 2 elements, namely $\{1\}$ and $\{\gamma^n\}$ for $\gamma\in{\rm C}_{2n}$ a generator. The fixed groups are simply $G$ because $G$ is abelian. In particular,
  \begin{align}
    ({\rm B}_\et G)(\mathbb{R})={\rm BC}_{2n}\sqcup {\rm BC}_{2n}
  \end{align}
  has two connected components. For groups of odd order $n$, there is only one conjugacy class, and $({\rm B}_\et G)(\mathbb{R})={\rm BC}_n$.
\end{example}

\begin{example}
  We consider a case with non-trivial involution, the group $G=\mu_n$ of $n$-th roots of unity, with the complex conjugation. The underlying finite group is ${\rm C}_n$ and the involution is $g\mapsto g^{-1}$. In this case, any element $g\in {\rm C}_n$ is a cycle, since $gg^{-1}=1$. Twisted conjugation by $h\in{\rm C}_n$ maps $g\mapsto \iota(h)gh^{-1}=gh^{-2}$, i.e., the twisted conjugacy class of an element is its square coset $g\{h^2\mid h\in {\rm C}_n\}$. There are two classes for $n$ even and one class for $n$ odd. The fixed group $K_g=\{h\in G\mid \iota(h)g=gh\}$ of a cycle $g\in {\rm C}_n$ consists of $h$ with $\iota(h)g=gh$, which independently of $g$ is the set of order 2 elements. In particular, we get
  \begin{align}
  ({\rm B}_\et \mu_n)(\mathbb{R})=\left\{\begin{array}{ll} \ast & n \textrm{ odd}\\ {\rm BC}_2\sqcup{\rm BC}_2 & n \textrm{ even}\end{array}\right.
  \end{align}
  This is compatible with the computations of Chow--Witt rings of ${\rm B}_\et\mu_n$ in \cite{dilorenzo:mantovani} ($n$ even) as well as Andrea Lachmann's thesis \cite{lachmann} ($n$ odd).
\end{example}

\begin{example}
  If $G$ is a group of odd order with trivial involution, then $({\rm B}_\et G)(\mathbb{R})={\rm B}G$. In this case, the real cycle class map
  \[
  {\rm H}^p({\rm B}_\et G,{\bf I}^q)\otimes\mathbb{Z}[1/2]\to{\rm H}^p({\rm B}G,\mathbb{Z})
  \]
  is an isomorphism because 2 is invertible in singular cohomology. Likely a transfer argument can be used to show that 2 is invertible on ${\bf I}$-cohomology as well. More generally, odd torsion in ${\bf I}^q$-cohomology of classifying spaces of finite groups (over the real numbers) can be read off from singular cohomology of the real points.
\end{example}

\begin{remark}
  Finite groups present an interesting test case for Borel-character type statements, cf.{}~\cite{deglise:fasel:borel-character} and our discussion of the conjectural picture in Section~\ref{sec:conjectures}. On the one hand, we can consider the Witt ring of ${\rm B}_\et G$, which should be related to the real representation ring ${\rm RO}(G)$. On the other hand, we can consider the unramified Witt group of ${\rm B}_\et G$, which would be related (via the real cycle class map, after inverting 2) to the free abelian group generated by conjugacy classes of involutions (i.e., order 2 elements) in $G$. It would be interesting to understand the sheafification map from the real representation ring to the ring of involutions in this case. For now we just note that the numbers of conjugacy classes of order 2 elements and isomorphism classes of irreducible real representations do not agree in general. For example, the dihedral group ${\rm D}_8$ of order 8 has 4 conjugacy classes of involutions and 5 real irreducible representations, and the quaternion group ${\rm Q}_8$ has 2 conjugacy classes of involutions and 4 real irreducible representations.\footnote{The underlying picture here might be similar to the K-theoretic situation: for the group ${\rm C}_2$, the representation ring has underlying abelian group free of rank 2, while the Chow ring has a $\mathbb{Z}$-summand in degree 0 and a copy of $\mathbb{Z}/2\mathbb{Z}$ in each positive degree.}
\end{remark}

\subsection{Normalizer of maximal torus in \texorpdfstring{${\rm SL}_2$}{SL2}}
\label{sec:normalizer}

We now want to discuss the real cycle class map and real realization for the normalizer $N={\rm N}_{{\rm SL}_2}(T)$ of the maximal torus $T\cong\mathbb{G}_{\rm m}$  in ${\rm SL}_2$; in this case, the Witt-sheaf cohomology was computed by Levine in \cite{levine:normalizer}.

As a group, the normalizer sits in a (non-split) exact sequence
\[
1\to \mathbb{G}_{\rm m}\to {\rm N}_{{\rm SL}_2}(T)\to\mu_2\to 1.
\]
Using the standard choice for $T$, the diagonal matrices, the concrete matrix realization of the normalizer is the following:
\[
{\rm N}_{{\rm SL}_2}(T)(K)=\left\{\left(\begin{array}{cc}t&0\\0&t^{-1}\end{array}\right), \left(\begin{array}{cc}0&t\\-t^{-1}&0\end{array}\right)\mid t\in K\right\}
\]
The complex points of the normalizer are an extension of $\mathbb{C}^\times$ by $\mu_2$, with the complex conjugation acting on $\mathbb{C}^\times$.

We compute the homotopy fixed points of complex conjugation on $N(\mathbb{C})$. For the fun of it, we do the explicit matrix calculations. The objects of $[N(\mathbb{C})\backslash*]^{{\rm hC}_2}$ are elements $g\in N(\mathbb{C})$ with $g\overline{g}=1$. Any element $g\in\mathbb{G}_{\rm m}$ is of the form $g=\op{diag}(z,z^{-1})$, with the homotopy fixed point condition reducing to $z\overline{z}=\|z\|^2=1$. So one component of the homotopy fixed points is the unit circle ${\rm S}^1\subseteq\mathbb{C}^\times$. The elements of the other coset are of the form
\[
\left(\begin{array}{cc}0&z\\-z^{-1}&0\end{array}\right)
\]
with $z\in\mathbb{C}^\times$, and for these, the homotopy fixed point condition reduces to $\overline{z}=-z$. So the other component of the homotopy fixed points is given by the matrices as above with $z=b{\rm i}$, $b\in\mathbb{R}$, which form a copy of $\mathbb{R}^\times$ inside $N(\mathbb{C})$.

Next, we compute the morphisms of $[N(\mathbb{C})\backslash*]^{{\rm hC}_2}$ which are of the form $g\to \overline{h}g h^{-1}$ for $g,h\in N(\mathbb{C})$. We do the case-by-case computation. If $g=\op{diag}(z,z^{-1})$ and $h=\op{diag}(w,w^{-1})$ are in $\mathbb{G}_{\rm m}$, the morphism $g\to \overline{h}g h^{-1}$ will have target the diagonal matrix for
\[
\overline{w}zw^{-1}=\frac{\overline{w}^2}{\|w\|^2}z.
\]
In particular, the target is multiplied by a norm 1 element, and by taking square roots and conjugating, every norm 1 element is of the prescribed form. This means that these morphisms connect arbitrary homotopy fixed points from the component ${\rm S}^1$. For $h\not\in\mathbb{G}_{\rm m}$, the morphism will have target the diagonal matrix for $\overline{w}w^{-1}z^{-1}$. This contributes a couple more automorphisms, but doesn't enlarge the connected component ${\rm S}^1$.

For the other component of homotopy fixed points, let $g$ be the anti-diagonal matrix with right upper entry $z$, and $h=\op{diag}(w,w^{-1})$. Then $\overline{h}g h^{-1}$ is the anti-diagonal matrix with upper right entry $\overline{w}wz$. This means that morphisms multiply by the norm of a complex number, i.e., a morphism connects a homotopy fixed point for $z=b{\rm i}$ with any of its positive multiples. If now $h$ is the anti-diagonal matrix with upper right entry $w$, then $\overline{h}g h^{-1}$ is the anti-diagonal matrix with upper-right entry $w\overline{w}z^{-1}$. In particular, in addition to positive multiples, this exchanges the two components of $\mathbb{R}^\times$ since $z=b{\rm i}$.

We conclude that the groupoid of homotopy fixed points has two connected components, one given by matrices in the unit circle (a copy of ${\rm S}^1$), and the other one given by anti-diagonal matrices with upper-right entry totally imaginary (a copy of $\mathbb{R}^\times$).

It remains to compute the automorphism groups for these two components of homotopy fixed points. The easiest way to do this is for the identity matrix (in the component ${\rm S}^1$) and the matrix
\[
\left(\begin{array}{cc}
  0&{\rm i}\\{\rm i}&0\end{array}\right).
\]
For the identity matrix, the stabilizer group is
\[
\left\{\left(\begin{array}{cc}w&0\\0&w^{-1}\end{array}\right), \left(\begin{array}{cc}0&w\\-w^{-1}&0\end{array}\right)\mid w\in \mathbb{R}^\times\right\}.
\]
This group is essentially the split $\mathbb{R}$-form of the normalizer, it is realized as normalizer of the torus in ${\rm SL}_2(\mathbb{R})$. It contains the cyclic group ${\rm C}_4$ of order 4, generated by a Weyl group representative, as deformation retract. For the imaginary matrix, we get the stabilizer group
\[
\left\{\left(\begin{array}{cc}w&0\\0&w^{-1}\end{array}\right)\mid w\overline{w}=1\right\}\cong{\rm S}^1.
\]
In this case, the anti-diagonal matrices don't contribute because the stabilizer condition is $w\overline{w}=-1$ and thus impossible to satisfy.

Putting everything together, we have thus established the following identification:
\begin{align}
{\rm Re}_{\mathbb{R}}({\rm B}_\et{\rm N}_{{\rm SL}_2}(T))\simeq {\rm B}{\rm N}_{{\rm SL}_2}(T)(\mathbb{C})^{{\rm hC}_2}\simeq {\rm B}{\rm S}^1\sqcup {\rm B}{\rm C}_4.
\end{align}

We now briefly discuss the real cycle class map for the case of the normalizer, and relate the above homotopy fixed point computation to Marc Levine's Witt-sheaf computations in \cite{levine:normalizer}.

To understand the real cycle class map with twisted coefficients, we need to understand the cycle class map on the Picard group. The equivariant Picard group is given by the characters, and so we have
\[
  {\rm CH}^1({\rm B}_\et{\rm N}_{{\rm SL}_2}(T))\cong\mathbb{Z}/2\mathbb{Z},
\]
generated by the sign representation  ${\rm N}_{{\rm SL}_2}(T)\twoheadrightarrow\mu_2\hookrightarrow \mathbb{G}_{\rm m}$. In particular, there is one possible non-trivial twist. We can also easily compute the realization of the non-trivial line bundle. On the component ${\rm BS}^1$ we obtain the trivial line bundle, since the induced homomorphism ${\rm S}^1\hookrightarrow\mathbb{C}^\times\hookrightarrow {\rm N}_{{\rm SL}_2}(T)(\mathbb{C})\to\mu_2$ is the trivial homomorphism. On the component ${\rm BC}_4$ we obtain a nontrivial line bundle corresponding to the (nontrivial) homomorphism ${\rm C}_4\hookrightarrow{\rm N}_{{\rm SL}_2}(T)\to \mu_2$.

For the trivial twist, the real cycle class map is a morphism 
\[
{\rm H}^*({\rm B}_\et{\rm N}_{{\rm SL}_2}(T),{\bf W})\to {\rm H}^*({\rm BS}^1,\mathbb{Z})\oplus {\rm H}^*({\rm BC}_4,\mathbb{Z})
\]
which is an isomorphism after inverting 2, by Theorem~\ref{thm:real-cycle-class}. In particular, in degree 0, the two components of the homotopy fixed points are accounted for by the two generators $1$ and $\overline{q}$ (in the notation of \cite{levine:normalizer}). In degree 2, there is a polynomial generator given by the pullback of the Euler class from ${\rm BSL}_2$, which under the real cycle class map is detected on the component ${\rm BS}^1$.

For the non-trivial twist $\mathscr{L}$, we obtain a morphism 
\[
{\rm H}^*({\rm B}_\et{\rm N}_{{\rm SL}_2}(T),{\bf W}(\mathscr{L}))\to {\rm H}^*({\rm BS}^1,\mathbb{Z})\oplus{\rm H}^*({\rm BC}_4,\mathbb{Z}(\mathscr{L})),
\]
which again is an isomorphism after inverting 2. Since the line bundle $\mathscr{L}$ restricts trivially to the component ${\rm BS}^1$, we still see the nontwisted cohomology of ${\rm BS}^1$ in the twisted Witt-sheaf cohomology of ${\rm B}_\et{\rm N}_{{\rm SL}_2}(T)$. In degrees $\geq 2$, this is the free rank one module over ${\rm H}^*({\rm BSL}_2,{\bf W})$ generated by the Euler class $j^\ast{\rm e}(\mathcal{T})$, as described in \cite{levine:normalizer}. In degree 0, the global sections of the trivial local system over ${\rm BS}^1$ contribute to the singular cohomology on the right-hand side. Under the real cycle class map, this corresponds to a degree 0 class in Witt-sheaf cohomology which was missing from the original description in \cite{levine:normalizer}, but described in \cite{dangelo:normalizer}*{Proposition~3.16} and \cite{levine:atiyah-bott}*{Theorem~5.1}. All in all, we find that the equivariant real cycle class map explains satisfactorily the relation between the Witt-sheaf computation in \cite{levine:normalizer} and the homotopy fixed point picture in the real realization.

Finally, we discuss the link between characteristic classes and the symmetric representation ring. As discussed in \cite{levine:normalizer}*{Section~6}, the irreducible representations of the normalizer $N={\rm N}_{\rm SL_2}(T)$ are the following: trivial and sign representation, and then there are 2-dimensional ``dihedral'' representations where the torus $\mathbb{G}_{\rm m}$ acts with weights $n,-n$ and $\mu_2$ switches the two one-dimensional $\mathbb{G}_{\rm m}$-representations. The trivial and sign representations are obviously symmetric. The dihedral representations are essentially restrictions of representations from ${\rm SL}_2$, which means that the odd powers of the defining representation (i.e., those where the torus is acting with odd weights) are symplectic and the even powers are symmetric.

Again, the pattern matches the description of Witt-sheaf cohomology of \cite{levine:normalizer}. The two 1-dimensional representation correspond to the two degree 0 cohomology classes detecting the two real components. The dihedral representations correspond to the positive degree non-twisted cohomology classes, while the ``twisted dihedral'' representations, i.e., tensor products of dihedral and sign representations, correspond to the positive-degree twisted cohomology classes. The interesting relation $(1+\langle\overline{q}\rangle)e=0$ seems to say that the tensor product of the regular representation of $\mu_2$ with the dihedral representations carries a hyperbolic form.

\subsection{The orthogonal group \texorpdfstring{${\rm O}(1,1)$}{O(1,1)}}

We describe the homotopy fixed point space $\op{B}\op{O}(1,1)^{{\rm hC}_2}$ for the split orthogonal group ${\rm O}(1,1)$. In this small case, we again do the explicit matrix calculations. In the next subsection, we will deal with the general case of (special) orthogonal groups using the formulas from Proposition~\ref{prop:bg-formula}. 

\begin{proposition}\label{prop:O11} Let $G={\rm O}(1,1)$ be the split orthogonal group over $\mathbb{R}$, and consider ${\rm O}(1,1;\mathbb{C})$ with the corresponding antiholomorphic involution. Then we have the following equivalence:
  \[
  {\rm BO}(1,1;\mathbb{C})^{{\rm hC}_2}\simeq {\rm BO}(1,1)\sqcup {\rm BO}(2)\sqcup{\rm BO}(2).
  \]
  The orthogonal groups on the right-hand side are the real ones.
\end{proposition}

\begin{proof}
  For the proof, we consider $\op{O}(1,1;\mathbb{C})\cong \mathbb{C}^\times\rtimes\mu_2$ and make case distinctions according to the $\mathbb{C}^\times$-coset decomposition.\footnote{Note that both ${\rm O}(1,1)$ and ${\rm N}_{{\rm SL}_2}(T)$ are extensions of the form $1\to \mathbb{C}^\times\to G\to\mu_2\to 1$, but the first is split while the second is non-split.} An element in the coset for $-1\in \mu_2$ can be written as $Jg$ for $g$ in the coset for $1\in\mu_2$ and $J=\op{diag}(1,-1)$. 
  
  We first determine the possible objects. For $z\in\mathbb{C}^\times$, the condition ${\rm N}(z)^2=z\overline{z}=1$ means $z\in {\rm S}^1$. In the other coset, we have $g=Jz$ with $z\in\mathbb{C}^\times$, and $JzJ\overline{z}=1$ implies $z=\overline{z}$, i.e., $z\in \mathbb{R}^\times$.

  Now we want to determine the morphisms. We first note that because morphisms are of the form $g\to \overline{h}g h^{-1}$, there can be no morphisms between homotopy fixed points in different cosets.

  So let $g\in{\rm S}^1$ and let $h\in\mathbb{C}^\times$. Then $\overline{h}g h^{-1}=\overline{h}h^{-1}g=\frac{\overline{h}^2}{{\rm N}(h)^2}g$. We can thus restrict to $h\in{\rm S}^1$ in which case the morphism has the shape $g\to \frac{g}{h^2}$. Therefore, any two $g,g'$ can be connected by a morphism, using $h=\sqrt{\frac{g}{g'}}$. As a result, we get one connected component. (We can also see that for $h=Jh'$ for $h'\in\mathbb{C}^\times$, there are no additional identifications.) To determine the stabilizer group, we can choose a representative $g\in{\rm S}^1$ of particular easy form - the identity matrix. We see
  \[
  K_{\op{id}}=\{h\in\op{O}(2;\mathbb{C})\mid \overline{h}=h\}=\op{O}(1,1;\mathbb{R}). 
  \]

  Now let $g=J\tau$ for $\tau\in\mathbb{R}^\times$. For $h\in\mathbb{C}^\times$, we have
  \[
  \overline{h}J\tau h^{-1}=J\overline{h}^{-1}\tau h^{-1}=J\tau\overline{h}^{-1}h^{-1}=J\tau{\rm N}(h)^{-1}.
  \]
  This implies that any two $\tau,\tau'$ in the same component of $\mathbb{R}^\times$ can be connected by a morphism (because ${\rm N}(h)>0$). On the other hand, we can also consider $h=Jh'$ for $h'\in\mathbb{C}^\times$ and compute
  \[
  \overline{h} J\tau h^{-1}=J\overline{h'} J\tau (h')^{-1}J=\overline{h'}^{-1}\tau (h')^{-1}J=\overline{h'}^{-1}(h')^{-1}J\tau^{-1}={\rm N}(h')^{-1}J\tau^{-1}.
  \]
  Again, this allows to connect homotopy fixed points in the same connected components of $\mathbb{R}^\times$. However, homotopy fixed points in different connected components of $\mathbb{R}^\times$ cannot be connected by morphisms. In particular, we get two connected components. For the respective stabilizer groups, we can again choose particularly simple representatives, namely $\op{diag}(1,-1)$ and $\op{diag}(-1,1)$. In the first case, we get $\overline{h}\op{diag}(1,-1)=\op{diag}(1,-1)h$ reduces to $\overline{h}=h^{-1}$, hence 
  \[
  K_{\op{diag}(1,-1)}=\left\{h\in \op{O}(2;\mathbb{C})\mid \overline{h}=h^{-1}\right\}=\op{O}(2;\mathbb{R}),
  \]
  and similarly for the other case $\op{diag}(-1,1)$.
\end{proof}

\subsection{Orthogonal groups and characteristic classes for quadratic forms}
\label{sec:orthogonal}

We now consider the general case of orthogonal groups ${\rm O}(p,q)$ and special orthogonal groups ${\rm SO}(p,q)$. In this generality, we will now employ the description in Proposition~\ref{prop:bg-formula} to understand the real realization of their geometric classifying spaces. We will first recall that components of the homotopy fixed point space correspond to different quadratic forms, and the stabilizer group of a homotopy fixed point is the automorphism group of the corresponding quadratic form. Combining all this, we find that the homotopy fixed point space is a disjoint union of indefinite (special) orthogonal groups, which provides preliminary information on Witt-sheaf cohomology.

The following proposition describes the connected components of the homotopy fixed point space for the complex conjugation on ${\rm SO}(p,q;\mathbb{C})$. Of course, all special orthogonal groups are isomorphic over the complex numbers, but the notation is supposed to imply that the complex conjugation on ${\rm SO}(p,q;\mathbb{C})$ is the one induced from the group being the complexification of ${\rm SO}(p,q;\mathbb{R})$. 

\begin{proposition}
  The components of ${\rm BSO}(p,q;\mathbb{C})^{{\rm hC}_2}$ are in bijective correspondence with equivalence classes of quadratic forms over $\mathbb{R}$ of dimension $n$ having the same discriminant as the quadratic form of signature $(p,q)$. More precisely, assuming $q\equiv pq\bmod 2$, the natural inclusion
  \begin{align*}
    \left\{\op{diag}(1_r,-1_s)\mid r+s=p+q \textrm{ and } s\equiv pq\bmod 2\right\}&\hookrightarrow {\rm BSO}(p,q;\mathbb{C})^{{\rm hC}_2}\\
    \op{diag}(1_r,-1_s)&\mapsto \op{diag}(1_r,-1_s)\cdot\op{diag}(1_p,-1_q)
  \end{align*}
  induces a bijection onto $\pi_0{\rm BSO}(p,q;\mathbb{C})^{{\rm hC}_2}$.
\end{proposition}

\begin{proof}  
  From the formula in Proposition~\ref{prop:bg-formula}, we find that the components of the homotopy fixed point space are in bijective correspondence with the first Galois cohomology set ${\rm H}^1({\rm C}_2,{\rm SO}(p,q))$. These are the strong inner forms of the real group ${\rm SO}(p,q)$. The identification in terms of equivalence classes of quadratic form with discriminant condition is classical, cf.{}~e.g.{}~\cite{serre:galois-cohom}*{Section III.3.2}, or \cite{adams:taibi}*{Example 8.20}.

  For the more precise statement giving representatives of homotopy fixed points, we first note that the matrices clearly are homotopy fixed points as they satisfy $g\sigma(g)=1$. The description of the map may seem a bit strange but it takes a diagonal matrix describing a quadratic form $\varphi$ of the same discriminant to the matrix we need to twist the given form $(1_p,-1_q)$ with to obtain $\varphi$. For example, the matrix $\op{diag}(1_p,-1_q)$ will map to the identity matrix which in the homotopy fixed point space corresponds exactly to the given form of signature $(p,q)$.

  We briefly recall how the matrices in the homotopy fixed points are used to change the quadratic form over $\mathbb{R}$. The complex conjugation $\sigma$ is the one associated to the quadratic form/orthogonal group of signature $(p,q)$. For a real vector $x$, i.e., one invariant under $\sigma$, we then obtain the value of the real quadratic form by $x^{\rm t}\cdot \op{diag}(1_p,-1_q)\cdot x$. Given a diagonal matrix $J$ as in the statement, we get another involution $g\mapsto J\sigma(g)J$ by twisting/conjugation. The ``real'' vectors for this involution are the ones having purely imaginary components where $J$ has entries $-1$. For such vectors $x$, which we can think of as $\sqrt{J}y$ with $y$ a real vector, we get $x^{\rm t}\cdot \op{diag}(1_p,-1_q)\cdot x=y^{\rm t}\cdot J\op{diag}(1_p,-1_q)\cdot y$. Similarly, the matrices fixed by this involution will have purely imaginary entries in rows and columns corresponding to $-1$ entries of $J$, and by removing the imaginary multiples, we get real matrices preserving the form given by $J\op{diag}(1_p,-1_q)$. This means that the correspondence between quadratic forms over $\mathbb{R}$ and homotopy fixed points is as claimed.
\end{proof}

The same argument extends to a description of representatives for the homotopy fixed point space in the case of orthogonal groups. 

\begin{proposition}
  The components of ${\rm BO}(p,q;\mathbb{C})^{{\rm hC}_2}$ are in bijective correspondence with equivalence classes of quadratic forms over $\mathbb{R}$ of dimension $n$. More precisely, for ${\rm O}(p,q;\mathbb{C})$, the natural inclusion
  \[
  \left\{\op{diag}(1_r,-1_s)\mid r+s=n\right\}\hookrightarrow {\rm BO}(p,q;\mathbb{C})^{{\rm hC}_2}
  \]
  induces a bijection onto $\pi_0{\rm BO}(p,q;\mathbb{C})^{{\rm hC}_2}$.
\end{proposition}

Now that we have concrete representatives for the components, we can deduce explicit descriptions of the stabilizer groups for such explicit representatives.

\begin{proposition}
  For $J=\op{diag}(1_r,-1_s)\in {\rm BO}(p,q;\mathbb{C})^{{\rm hC}_2}$, the stabilizer group is 
  \[
  K_J=\left\{g\in G\mid \sigma(g)J=J g\right\}.
  \]
  This is an orthogonal group, the automorphism group of the quadratic form corresponding to $\op{diag}(1_p,-1_q)\cdot\op{diag}(1_r,-1_s)$. A similar statement is true for special orthogonal groups.
\end{proposition}

\begin{proof}
  The defining equation of $K_J$ states that this is the fixed group for the new involution $g\mapsto J\sigma(g)J$.
\end{proof}

We can now combine the previous information on connected components and stabilizers to formulate more precisely the description of Proposition~\ref{prop:bg-formula} of the homotopy fixed points for (special) orthogonal groups. 

\begin{theorem}
  \label{thm:realization-orthogonal}
  Let $G={\rm (S)O}(p,q)$ be the real (special) orthogonal group of signature $(p,q)$. We assume that $q\equiv pq\bmod 2$. Then the following statements describe the homotopy fixed point space of the complexification of $G$:
  \begin{align*}
  {\rm BO}(n;\mathbb{C})^{{\rm hC}_2}&\simeq \bigsqcup_{p+q=n}{\rm BO}(p,q)\\
  {\rm BSO}(p,q;\mathbb{C})^{{\rm hC}_2}&\simeq \bigsqcup_{r+s=p+q;\; s\equiv pq\bmod 2}{\rm BSO}(r,s).
  \end{align*}
\end{theorem}

As the Witt-sheaf cohomology of orthogonal groups has not been computed so far, the above results provide some insight into what we should expect. Using the equivariant real cycle class map, we get the following preliminary computation of Witt-sheaf cohomology of classifying spaces of orthogonal groups over $\mathbb{R}$, with half-integer coefficients: 
\begin{align}
{\rm H}^*_{\rm Zar}({\rm B}_\et{\rm O}(n),{\bf W}[1/2])\cong \bigoplus_{p+q=n}{\rm H}^*_{\rm sing}({\rm BO}(p,q),\mathbb{Z}[1/2]).
\end{align}
As a particular consequence, Witt-sheaf cohomology contains information on all the real forms of a given orthogonal group, and the associated characteristic classes for indefinite orthogonal bundles.

We can interpret this as follows: assume we have a real scheme $X/\mathbb{R}$ and an ${\rm O}(n)$-torsor $E/X$. For each connected component of $X(\mathbb{R})$, the real points of $E$ form a torsor for an indefinite orthogonal group ${\rm O}(p,q)$, but the individual $p$ and $q$ can vary between connected components of $X(\mathbb{R})$. Moreover, the relevant characteristic classes and their degrees vary with $(p,q)$. They will always be Pontryagin and Euler classes, but which ones appear exactly depends on $(p,q)$. In this picture, we can interpret the fact that ${\rm H}^0({\rm B}_\et{\rm O}(n),{\bf W}[1/2])\cong\mathbb{Z}[1/2]^{n+1}$: components with different $(p,q)$ can be separated by degree 0 cohomology classes. Considering twisted cohomology, or the Witt-sheaf cohomology of ${\rm B}_\et{\rm SO}(p,q)$, we also find that there are higher-degree algebraic classes corresponding to Euler-classes for ${\rm (S)O}(p,q)$-bundles in case $p$ or $q$ is even. These classes are not algebraic Euler classes for rank $n$ quadratic forms, but they are connected to algebraic Euler classes of smaller rank quadratic forms via stabilization. 

It can be shown that this is indeed (almost) the formula for Witt-sheaf cohomology of orthogonal groups over general base fields of characteristic $\neq 2$. In fact, the present paper arose mostly from an attempt to make sense of small rank computations for ${\rm O}(n)$ with $n\leq 4$. With a candidate description of the Witt-sheaf cohomology ring arising from the above homotopy fixed point description, it is possible to inductively compute Witt-sheaf cohomology of orthogonal groups inductively using Vistoli's stratification method, following the arguments in \cite{RojasVistoli}. In this argument some boundary maps produce 2-torsion in Witt-sheaf cohomology over general fields which isn't visible over the real numbers.

We finally make some brief remarks on the symmetric representations and degree 0 Witt-sheaf cohomology. For the special orthogonal groups, the defining representation by definition preserves a symmetric form. Therefore, all the exterior powers of the defining representation preserve a symmetric form as well. Additionally, there are half-spin representations for ${\rm SO}(2n)$ and these are self-dual for $n$ even. In particular, the symmetric representation ring is very close to the usual representation ring, quite similar to the well-known close relation between ${\rm R}({\rm SO}(n))$ and ${\rm RO}({\rm SO}(n))$ in the Lie-group case.

The degree 0 Witt-sheaf cohomology has already been computed in \cite{garibaldi:merkurjev:serre}*{Theorem~27.16}, the result for the orthogonal group is
\[
{\rm H}^0({\rm B}_\et{\rm O}(n),{\bf W})\cong {\rm W}(F)^{\oplus(n+1)},
\]
with the generators on the right-hand side corresponding to the exterior powers of the fundamental representation, considered as unramified quadratic forms over the classifying space. The augmentation map ${\rm W}^*(G)\to {\rm H}^0({\rm B}_\et{\rm O}(n),{\bf W})$ is then the ring homomorphism mapping the exterior powers of the fundamental representation to the corresponding classes in Witt-sheaf cohomology.

Similarly, for ${\rm SO}(n)$, we get that degree 0 Witt-sheaf cohomology is a free ${\rm W}(F)$-module, with the rank depending on $n$. For $n=2m+1$, the rank is $m+1$, with generators the first half of the exterior powers of the fundamental representation. For $n=4m$ and $n=4m+2$, the rank is $2m+1$. In the former case, the generators are the first half of the exterior powers, plus a half-spin representation. In the latter case, the half-spin representation isn't symmetric, and we only have the first $2m+1$ exterior powers of the fundamental representation.

\subsection{Remarks on cohomological invariants of spin groups and exceptional groups}
\label{sec:cohom-inv}

Using the formula in Proposition~\ref{prop:bg-formula}, we can also easily derive information on Witt-sheaf cohomological invariants for many other real groups, using the computations of Galois cohomology sets in \cite{adams:taibi}. Essentially, using the results on the equivariant real cycle class map in Theorem~\ref{thm:real-cycle-class}, the isomorphism
\[
  {\rm H}^0({\rm B}_\et G,{\bf W})\otimes\mathbb{Z}[1/2]\to {\rm H}^0(({\rm B}_\et G)(\mathbb{R}),\mathbb{Z}[1/2])\cong\mathbb{Z}[1/2]\left(\pi_0\left(({\rm B}_\et G)(\mathbb{R})\right)\right)
\]
implies that the rank of ${\rm H}^0({\rm B}_\et G,{\bf W})$ (as module over ${\rm W}(\mathbb{R})\cong\mathbb{Z}$) can be identified with the number of connected components of the real realization ${\rm Re}_{\mathbb{R}}({\rm B}_\et G)$. In particular, the number of connected components of $({\rm B}_\et G)(\mathbb{R})$ equals the number of $\mathbb{Z}[1/2]$-linearly independent Witt-sheaf cohomological invariants for $G$, see \cite{garibaldi:merkurjev:serre} for a discussion of Witt-invariants of groups. We first formulate this observation:

\begin{proposition}
  \label{prop:cohom-inv}
  Let $G$ be a split reductive group over $\mathbb{R}$. Then there is an isomorphism
  \[
  {\rm H}^0({\rm B}_\et G,\mathbf{W}[1/2])\cong {\rm H}^0({\rm B}G(\mathbb{C})^{\rm hC_2},\mathbb{Z}[1/2])\cong\mathbb{Z}\langle{\rm H}^1({\rm C}_2,G(\mathbb{C}))\rangle.
  \]
  Consequently, up to torsion, the number of Witt-sheaf cohomological invariants for $G$-torsors over $\mathbb{R}$ equals the number $\#{\rm H}^1({\rm C}_2,G(\mathbb{C}))$ of strong real forms of $G(\mathbb{R})$ with trivial central invariant.
\end{proposition}

The relevant numbers of connected components of $({\rm B}_\et G)(\mathbb{R})$, alternatively strong real forms, can then be read off from the tables on pp. 1093/1094 in \cite{adams:taibi}, for any simply-connected almost simple group $G$ over $\mathbb{R}$. In the case of the spin groups, we get the following result (which we only spell out for the split forms of the spin groups):

\begin{corollary}
  \label{cor:spin}
  For the split spin groups over $\mathbb{R}$, we get the following formulas for the number of Witt-sheaf cohomological invariants (up to torsion) from the tables in \cite{adams:taibi}:
  \begin{align}
  \op{rk}_{\mathbb{Z}}{\rm H}^0({\rm B}_\et \op{Spin}(n,n+1),\mathbf{W})&\cong
  \left\lfloor \frac{2n+1}{4}\right\rfloor+\left\{\begin{array}{ll}
  2 & n\equiv 0,3\bmod 4\\
  1 & n\equiv 1\bmod 4\\
  0 & n\equiv 2\bmod 4\\
  \end{array}\right.\\
  \op{rk}_{\mathbb{Z}}{\rm H}^0({\rm B}_\et \op{Spin}(n,n),\mathbf{W})&\cong
  \left\lfloor \frac{n}{2}\right\rfloor+\left\{\begin{array}{ll}
  3 & n\equiv 0\bmod 4\\
  1 & n\equiv 1\bmod 4\\
  0 & n\equiv 2,3\bmod 4\\
  \end{array}\right.
  \end{align}
\end{corollary}

Since the Witt-sheaf cohomological invariants for spin groups haven't been determined in general, the real cycle class map seems to provide at least some preliminary information about these cohomological invariants.

\begin{remark}
  We compared the numbers of connected components inferred from the tables in \cite{adams:taibi} with the numbers of mod 2 cohomological invariants in the known cases, for spin groups $\op{Spin}(n)$ for $n\leq 12$, in \cite{garibaldi:spin}*{Table 21B on p.62}, see also the discussion in the thesis of Weinzierl \cite{weinzierl}. The numbers seem to match up for all $n\leq 11$, with two caveats. First, the constant form $\langle 1\rangle$ as a generator of ${\rm H}^0({\rm B}_\et G,{\bf W})$ would be counted as a cohomological invariant in our formulas above, but isn't counted in \cite{garibaldi:spin}, leading to a consistent discrepancy of 1 in the numbers. Second, mod 2 cohomological invariants would be more related to global sections of Milnor K-theory sheaves arising from the fundamental ideal filtration on Witt-sheaf cohomological invariants, which could potentially lead to differences. This happens for $\op{Spin}(12)$, where the numbers of strong real forms and mod 2 cohomological invariants disagree. Note, however, that the computations in \cite{garibaldi:spin} assume $\sqrt{-1}$ is contained in the base field. In particular, the  results in loc.{}~cit.{}~don't actually apply to the real numbers, and the discrepancy might also be due to the number of cohomological invariants for $\mathbb{R}$ and $\mathbb{C}$ being different.
\end{remark}

As a final remark, we note that results like Corollary~\ref{cor:spin} can also be obtained for all the simply-connected exceptional groups, using the tables in pp. 1093/1094 in \cite{adams:taibi}. For ${\rm B}_\et{\rm G}_2$ the rank of global Witt-sheaf sections is 2, for ${\rm B}_\et{\rm F}_4$ it is 3, in accordance with Serre's computations in \cite{garibaldi:merkurjev:serre}*{Example~27.20}. For ${\rm G}_2$, the additional generator is the norm form, for ${\rm F}_4$, the two generators besides $\langle 1\rangle$ are Pfister forms commonly called $f_3$ and $f_5$. 

For a new result, we deduce from the Adams--Ta\"ibi tables that $\op{rk}_{\mathbb{Z}}{\rm H}^0({\rm B}_\et {\rm E}_6,\mathbf{W})$ is 2 for inner forms and 3 for outer forms of the split group of type ${\rm E}_6$ over $\mathbb{R}$. This difference between inner and outer forms might be related to an observation of Kameko--Tezuka--Yagita \cite{kameko:tezuka:yagita} on Milnor mod 2 cohomological invariants. One might also wonder if it's possible to write down unramified quadratic forms in ${\rm H}^0({\rm B}_\et {\rm E}_6,\mathbf{W})$ explicitly; likely they will arise from suitable symmetric representations.

In any case, the possible relation between strong real forms and cohomological invariants (both for the Witt-sheaf as well as mod 2 Milnor K-theory sheaves) seems an interesting questions to pursue further. 

\subsection{Realization for vector bundles on quotient stacks}
\label{sec:vb-realization}

We briefly discuss the realization of (symmetric) vector bundles on quotient stacks, and outline how these can be described more explicitly using Proposition~\ref{prop:bg-formula}. In the simplest case of equivariant line bundles, this describes the cycle class map ${\rm CH}^1_G(X)\to {\rm H}^1({\rm Re}_{\mathbb{R}}[G\backslash X],\mathbb{F}_2)$ relevant for twisted coefficients. In the case of equivariant symmetric vector bundles, this will allow to relate characteristic classes in the cohomology of $[G\backslash X]$ to their topological counterparts on ${\rm Re}_{\mathbb{R}}[G\backslash X]$.

We begin with a discussion of vector bundles. Viewing a vector bundle as a morphism of quotient stacks $[G\backslash X]\to [{\rm GL}_n\backslash *]$, we can simply apply the real realization functor to this morphism. As discussed in Section~\ref{sec:examples} below, ${\rm Re}_{\mathbb{R}}([{\rm GL}_n\backslash *])\cong{\rm BGL}_n(\mathbb{R})$. Therefore, the real realization of a vector bundle on a stack is a morphism
\[
{\rm Re}_{\mathbb{R}}([G\backslash X])\to {\rm BGL}_n(\mathbb{R}),
\]
i.e., a real vector bundle on the real realization of the quotient stack. 

Let us describe in somewhat more detail how to compute this realization in the case of a classifying stack $[G\backslash *]$ for a linear algebraic group $G$. In this case, any equivariant vector bundle is the associated bundle for the universal $G$-torsor and a representation $\rho\colon G\to{\rm GL}_n$. We'll denote the algebraic vector bundle on $[G\backslash*]$ by $V_\rho$. Now by Proposition~\ref{prop:bg-formula}, each component of $[G\backslash*](\mathbb{R})$ corresponds to a strong real form of $G$. Alternatively, the components of the realization are parametrized by (conjugacy classes of) anti-holomorphic involutions $\sigma$ on $G(\mathbb{C})$, and the component corresponding to $\sigma$ is ${\rm B}G(\mathbb{C})^\sigma$. Now for every conjugacy class of anti-holomorphic involution $\sigma$ on $G(\mathbb{C})$, there is a representative such that $G(\mathbb{C})\to{\rm GL}_n(\mathbb{C})$ is compatible with the involutions ($\sigma$ on $G$ and complex conjugation on ${\rm GL}_n$). This induces a representation $\rho_\sigma\colon G(\mathbb{C})^\sigma\to{\rm GL}_n(\mathbb{R})$ (well-defined up to conjugation), and the induced morphism ${\rm B}G(\mathbb{C})^\sigma\to{\rm BGL}_n(\mathbb{R})$ classifies the real realization ${\rm Re}_{\mathbb{R}}(V_\rho)$ on the component ${\rm B}G(\mathbb{C})^\sigma$ of ${\rm Re}_{\mathbb{R}}[G\backslash *]$.\footnote{Note that in this way, the real realization of algebraic vector bundles links real vector bundles over different strong real forms of $G$. We wonder what the significance of this observation is for the Langlands classification of representations of real Lie groups and Langlands parameter spaces.}

\begin{example}
  One interesting example of this situation is the realization for equivariant line bundles on the classifying space ${\rm B}_\et N$ of the normalizer $N$ of the maximal torus in ${\rm SL}_2$. This is discussed in more details in Section~\ref{sec:normalizer}. The Picard group in this case is ${\rm CH}^1_N(*)\cong\mathbb{Z}/2\mathbb{Z}$, and the nontrivial line bundle is induced by the sign representation $N\to\mu_2\hookrightarrow\mathbb{G}_{\rm m}$. The explicit computation of homotopy fixed points in Section~\ref{sec:normalizer} shows that there are two components, ${\rm BS}^1$ and ${\rm BC}_4$. The group ${\rm S}^1$ is embedded as diagonal matrices with complex norm 1 entries, while the group ${\rm C}_4$ is generated by the Weyl group element. It is then immediate that the real realization of this line bundle is the trivial line bundle on ${\rm BS}^1$ and the line bundle corresponding to the sign representation of ${\rm C}_4$ on ${\rm BC}_4$. The fact that this line bundle is non-trivial (on ${\rm BC}_4$ and therefore on the whole space) but trivial on the component ${\rm BS}^1$ means that the non-twisted Euler class from ${\rm BS}^1$ is also related to the twisted Witt-sheaf cohomology. For more details, see Section~\ref{sec:normalizer}.
\end{example}

Now we can consider a similar question for symmetric vector bundles, viewed as morphisms of quotient stacks $[G\backslash X]\to[{\rm O}(n)\backslash*]$. Note that due to the \'etale-local nature of quotient stacks, it doesn't matter which concrete quadratic form (of rank $n$, over the base field $F$) is chosen in the definition of our orthogonal group ${\rm O}(n)$. As discussed in Section~\ref{sec:examples}, the realization of ${\rm B}_\et{\rm O}(n)$ is a disjoint union of classifying spaces of indefinite orthogonal groups over $\mathbb{R}$. This means that applying real realization to an equivariant symmetric vector bundle of rank $n$ produces a morphism
\begin{align}
  \label{eq:symmetric-realization}
{\rm Re}_{\mathbb{R}}([G\backslash X])\to\bigsqcup_{p+q=n}{\rm BO}(p,q),
\end{align}
meaning that each component of the realization of the quotient stack has an ${\rm O}(p,q)$-bundle over it, with the signature $(p,q)$ depending on the component. 

Again, in the case of a classifying space, an equivariant symmetric vector bundle is a symmetric representation $\rho\colon G\to{\rm O}(n)$ of the group $G$. The representation $\rho$ induces a morphism ${\rm H}^1({\rm C}_2;G)\to {\rm H}^1({\rm C}_2;{\rm O}(n))$ of Galois cohomology groups. This morphism determines the induced map on $\pi_0$ of the morphism~\eqref{eq:symmetric-realization}, i.e., it determines which component ${\rm BO}(p,q)$ a given component of ${\rm Re}_{\mathbb{R}}({\rm B}_\et G)$ will map to. As before, if $\sigma$ is an antiholomorphic involution corresponding to a strong real form of $G$, the orthogonal bundle on the relevant component will be classified by a morphism ${\rm B}G(\mathbb{C})^\sigma\to {\rm BO}(p,q)$. This morphism is induced from the representation $G(\mathbb{C})^\sigma\to{\rm O}(p,q)$ obtained from $\rho\otimes\mathbb{C}$ by passing to fixed groups of the relevant anti-holomorphic involutions.

The description above still doesn't quite explain what the realization of a given symmetric representation of ${\rm O}(n)$ will look like over each component. The short answer is that the fundamental representation of the algebraic group ${\rm O}(n)$ will realize to the fundamental representation of the Lie group ${\rm O}(p,q)$ over the component ${\rm BO}(p,q)$. In particular, the quadratic bundle over ${\rm BO}(p,q)$ will have signature $p-q$, which allows to separate connected components of the real realization of ${\rm B}_\et{\rm O}(n)$ in terms of signatures of symmetric representations. Since all symmetric representations of ${\rm O}(n)$ are built via exterior and symmetric powers from the fundamental one, this explains how to compute realization of symmetric vector bundles over ${\rm B}_\et {\rm O}(n)$. A more detailed explanation might be given elsewhere, for now we discuss one basic example on the realization of symmetric representations, the group ${\rm O}(1)=\mu_2$.

\begin{example}
  Consider the sign representation on $\mu_2$, equipped with the quadratic form $\langle 1\rangle$, as a symmetric representation of the linear group $\mu_2$ over $\mathbb{R}$. We can now complexify, consider the associated vector bundle on $\left(\mathbb{C}^n\setminus\{0\}\right)/\mu_2$, and compute the restriction to the fixed point loci under complex conjugation.

  To compute the restriction of the bundle, we want to take the fixed locus of complex conjugation on the quotient $\left(\left(\mathbb{C}^n\setminus\{0\}\right)\times\mathbb{C}\right)/{\mu_2}$, where of course $\mu_2$ acts on all components by the sign representation. In there, a point is fixed if either all coordinates are real, or all coordinates are imaginary. 

  We first consider the component of the fixed locus where all coordinates are real. The realization is a real line bundle over $\mathbb{RP}^{n-1}$; the bundle is nontrivial, because $\mu_2$ acts by the sign representation on the $\mathbb{C}$-component in $\left(\mathbb{C}^n\setminus\{0\}\right)\times\mathbb{C}$. The bundle also has a quadratic form, induced from the form $x\mapsto x^2$ on the $\mathbb{C}$-component. Since the form on the real bundle is obtained by restriction to the real axis $\mathbb{R}\subset\mathbb{C}$, the real realization of the sign representation with form $\langle1\rangle$ is again the sign representation with form $\langle1\rangle$ on the component with all coordinates real. 

  Now we consider the bundle over the other component, where all coordinates are imaginary. The total space of the real line bundle over $\mathbb{RP}^{n-1}$ in this case consists of those points in $\left(\left(\mathbb{C}^n\setminus\{0\}\right)\times\mathbb{C}\right)/{\mu_2}$, where all coordinates are imaginary. Again, since $\mu_2$ acts by sign representation on $\mathbb{C}$, this is the nontrivial bundle over $\mathbb{RP}^{n-1}$. The quadratic form on this real bundle is now the restriction of the form $x\mapsto x^2$ to the imaginary axis ${\rm i}\mathbb{R}$. This means that the real realization of the algebraic symmetric line bundle now corresponds to the sign representation with the form $\langle-1\rangle$. 

  We see that real realizations of algebraic bundles with symmetric forms can have different signatures over the different components of the real realization.
\end{example}

\end{document}